%% file: concordance_surgery.tex
\documentclass{amsart}
\usepackage{etex}
\usepackage{graphicx}
\usepackage{bm,caption,amsbsy,enumitem,amsmath,amsthm,
amssymb,mathtools,amsfonts,multirow,verbatim,tikz,tikz-cd,thmtools,thm-restate,xr,tensor,xcolor}
\usepackage[alphabetic]{amsrefs}
\usetikzlibrary{matrix,arrows,decorations.pathmorphing}
\usepackage[top=1.25in, bottom=1in, left=1.25in,
right=1.25in,marginpar=1in]{geometry}
\usepackage{url}
\DefineSimpleKey{bib}{myurl}
\usepackage[hidelinks]{hyperref}
\usepackage{chngcntr}
\usepackage[all]{xy}
\counterwithin{figure}{section}
\numberwithin{equation}{section}

\newcommand\myurl[1]{\url{#1}}

\BibSpec{webpage}{%
	+{}{\PrintAuthors} {author}
	+{,}{ \textit} {title}
	+{}{ \parenthesize} {date}
	+{,}{ \myurl} {myurl}
	+{,}{ } {note}
	+{.}{ } {transition}
}

\newtheorem{innercustomthm}{Theorem}
\newenvironment{customthm}[1]
{\renewcommand\theinnercustomthm{#1}\innercustomthm}
{\endinnercustomthm}

\newtheorem{thm}{Theorem}[section]
\newtheorem{prop}[thm]{Proposition}

\newtheorem{cor}[thm]{Corollary}
\newtheorem{lem}[thm]{Lemma}

\theoremstyle{definition}
\newtheorem{define}[thm]{Definition}

\theoremstyle{remark}
\newtheorem{rem}[thm]{Remark}
\newtheorem{example}[thm]{Example}

\newcommand{\ve}[1]{\boldsymbol{\mathbf{#1}}}
\newcommand{\R}{\mathbb{R}}
\newcommand{\F}{\mathbb{F}}
\newcommand{\Z}{\mathbb{Z}}
\newcommand{\N}{\mathbb{N}}

\newcommand{\T}{\mathbb{T}}

\renewcommand{\d}{\partial}
\renewcommand{\subset}{\subseteq}

\renewcommand{\tilde}{\widetilde}
\renewcommand{\hat}{\widehat}
\newcommand{\iso}{\cong}

\DeclareMathOperator{\GL}{{GL}}

\DeclareMathOperator{\Hom}{{Hom}}

\DeclareMathOperator{\id}{{id}}

\DeclareMathOperator{\Int}{{int}}

\DeclareMathOperator{\Lef}{{Lef}}

\DeclareMathOperator{\red}{{red}}

\DeclareMathOperator{\Spin}{{Spin}}

\DeclareMathOperator{\Tors}{{Tors}}

\DeclareMathOperator{\coker}{{coker}}

\newcommand{\bF}{\mathbb{F}}

\newcommand{\bS}{\mathbb{S}}
\newcommand{\bT}{\mathbb{T}}

\newcommand{\cC}{\mathcal{C}}
\newcommand{\cD}{\mathcal{D}}

\newcommand{\cG}{\mathcal{G}}
\newcommand{\cH}{\mathcal{H}}

\newcommand{\cM}{\mathcal{M}}
\newcommand{\cN}{\mathcal{N}}

\newcommand{\cR}{\mathcal{R}}
\newcommand{\cS}{\mathcal{S}}
\newcommand{\cT}{\mathcal{T}}

\newcommand{\cW}{\mathcal{W}}
\newcommand{\cX}{\mathcal{X}}

\newcommand{\frS}{\mathfrak{S}}

\newcommand{\frs}{\mathfrak{s}}
\newcommand{\frt}{\mathfrak{t}}
\newcommand{\fru}{\mathfrak{u}}

\newcommand{\as}{\ve{\alpha}}
\newcommand{\bs}{\ve{\beta}}

\newcommand{\xs}{\ve{x}}
\newcommand{\ys}{\ve{y}}

\newcommand{\zs}{\ve{z}}

\newcommand{\SFH}{\mathit{SFH}}

\newcommand{\CF}{\mathit{CF}}
\newcommand{\HF}{\mathit{HF}}
\newcommand{\HFh}{\widehat{\mathit{HF}}}

\newcommand{\HFKh}{\widehat{\mathit{HFK}}}
\newcommand{\HFK}{\mathit{HFK}}

\renewcommand{\a}{\alpha}
\renewcommand{\b}{\beta}
\newcommand{\g}{\gamma}
\renewcommand{\S}{\Sigma}

\newcommand{\x}{\mathbf{x}}
\newcommand{\y}{\mathbf{y}}

\newcommand{\tD}{\tilde{\cD}}
\newcommand{\PD}{\mathit{PD}}

\newcommand{\SW}{\mathit{SW}}

\DeclareMathOperator{\MOD}{{\mathsf{Mod}}}

\newcommand{\ufrs}{\underline{\frs}}

\newcommand{\omegas}{\ve{\omega}}

\DeclareMathOperator{\stab}{{stab}}
\DeclareMathOperator{\diff}{{diff}}

\newcommand{\dotsim}{\mathrel{\dot{\simeq}}}
\newcommand{\sca}{\a}
\newcommand{\scb}{\b}

\title{Concordance surgery and the Ozsv\'ath--Szab\'o 4-manifold invariant}

\author{Andr\'as Juh\'asz}%
\address{Mathematical Institute, University of Oxford, Andrew Wiles Building,
	Radcliffe Observatory Quarter, Woodstock Road, Oxford, OX2 6GG, UK}%
\email{juhasza@maths.ox.ac.uk}%

\author{Ian Zemke}
\address{Department of Mathematics\\Princeton University\\  Princeton, NJ 08544,
	USA}
\email{izemke@math.princeton.edu}
\begin{document}
	
	\subjclass[2010]{57R58; 57R55; 57M27}%
	\keywords{Heegaard Floer homology, 4-manifolds, concordance}
	\begin{abstract}
		We compute the effect of concordance surgery, a generalization of knot surgery
		defined using a
		self-concordance of a knot, on the Ozsv\'ath--Szab\'o 4-manifold invariant.
		The formula involves the graded Lefschetz number of the concordance map on
		knot Floer homology.
		The proof uses the sutured Floer TQFT, and a version of sutured Floer homology
		perturbed by a 2-form.
	\end{abstract}
	
	\maketitle
	
	\section{Introduction}
	
	Let $X$ be a smooth, connected, closed, and oriented 4-manifold with $b_2^+(X)
	\ge 2$.
	Suppose that $T \subset X$ is a smoothly embedded, homologically essential torus
	with trivial self-intersection, and let $K \subset S^3$ be a knot.
	Fintushel and Stern~\cite{FSKnotSurgery} defined the \emph{knot surgery}
	operation on $X$,
	resulting in the 4-manifold $X_K$. This is obtained by gluing $X \setminus N(T)$
	and $S^1 \times (S^3 \setminus N(K))$ via an orientation-reversing
	diffeomorphism
	of their boundaries that maps a meridian of $T$ to a longitude of $K$. They
	showed that
	\begin{equation}\label{eqn:FS}
	\SW(X_K) = \Delta_K(z) \cdot \SW(X),
	\end{equation}
	where $\SW$ denotes the Seiberg--Witten invariant,
	and $\Delta_K(z)$ is the symmetrized Alexander polynomial of $K$. The variable
	$z$ corresponds to $\exp(2[T])$, where $[T]$ is the homology class induced by $T$ in $H_2(X_K)$.
	
	If $\pi_1(X \setminus T) = 1$, then $X$ and $X_K$ are simply-connected and have
    the same intersection form, and are hence homeomorphic by
	Freedman's theorem. Note that every symmetric integral Laurent polynomial $p(z)$
	satisfying $p(1) = \pm 1$ is the Alexander polynomial of a knot in $S^3$.
	Consequently, if $\SW(X) \neq 0$, then we obtain infinitely
	many pairwise non-diffeomorphic smooth structures on $X$.
	When $X$ is the K3 surface, $\SW(X) = 1$, and hence
	we obtain a different smooth structure on $X$ for every such Laurent polynomial.
	
	Mark~\cite{Mark}*{Theorem~3.1} obtained a result analogous to
	equation~\eqref{eqn:FS} for the Ozsv\'ath--Szab\'o 4-manifold invariant~\cite{OSTriangles},
	which is expected to coincide with the Seiberg--Witten invariant.
	For a closed 4-manifold $X$ with $b_2^+(X)\ge 2$, Ozsv\'{a}th and Szab\'{o}'s
    invariant takes the form of a map
	\[
	\Phi_X \colon  \Spin^c(X)\to \bF_2.
	\]
	We write $\Phi_{X,\frs}$ for the value of $\Phi_X$ on $\frs$.
	It is convenient to organize the invariants of different $\Spin^c$ structures
	into a single polynomial.  Recall that
	$\Spin^c(X)$ is an affine space over $H^2(X)$, so the difference of two
	$\Spin^c$ structures is a well-defined cohomology class. If
	$\ve{b} = (b_1,\dots, b_n)$ is a basis of $H^2(X;\R)$,
    we can arrange the 4-manifold invariant into the element
	\[
	\Phi_{X; \ve{b}} := \sum_{\frs \in \Spin^c(X)} \Phi_{X,\frs} \cdot
	z_1^{\langle\, i_*(\frs-\frs_0) \cup b_1, [X] \,\rangle} \cdots
	z_n^{\langle\, i_*(\frs-\frs_0) \cup b_n, [X] \,\rangle}
	\]
	of the $n$-variable Novikov ring over $\bF_2$, where $\frs_0$ is some choice of base $\Spin^c$ structure on $X$,
    and $i_* \colon H^2(X) \to H^2(X; \R)$ is induced by the map of coefficients $\Z \hookrightarrow \R$.
	If $H^2(X)$ is torsion-free, then $\Phi_{X; \ve{b}}$ completely encodes the map $\Phi_X$.
	It is natural to view $\Phi_{X;\ve{b}}$ as a perturbed version of the mixed invariant;
	 see Proposition~\ref{prop:perturbed-mixed-invariant}.
	
	\emph{Concordance surgery} is a generalization of knot surgery
	due to Fintushel and Stern; see Akbulut~\cite{Akbulut}*{Section~2} and Tange~\cite{TangeConcordance}.
	Let $K$ be a knot in a homology 3-sphere $Y$
	(note that Akbulut only considered the case $Y = S^3$). Given a self-concordance
	$\cC = (I \times Y, A)$ from $(Y, K)$ to itself, we can construct a
	4-manifold $X_{\cC}$, as follows. We glue the ends of $A$
	together to form a 2-torus $T_{\cC}$ embedded in $S^1 \times Y$. After
	removing a neighborhood of $T_{\cC}$, we get a 4-manifold $W_{\cC}$ with
	boundary~$\bT^3$. Viewing $N(T)$ as $T\times D^2$,
	we pick any orientation-preserving diffeomorphism
	$\phi \colon \d (X\setminus N(T))\to \d N(T_{\cC})$ that sends
	$[\{p\}\times \d D^2]$ to $[\{q\} \times \ell_K]$, where
    $p \in T$, $q \in S^1$, and $\ell_K$ is a longitude of $K$.
	We write $X_{\cC}$ for any manifold constructed as the union
	\[
	X_{\cC} := (X\setminus N(T)) \cup_{\phi} W_{\cC}.
	\]
	Fintushel and Stern asked in the late 90s whether a formula similar to
	equation~\eqref{eqn:FS}
	relates $\SW(X)$ and $\SW(X_{\cC})$; see Akbulut~\cite{Akbulut}*{Remark~2.2}.

	Our main result gives a formula relating the Ozsv\'ath--Szab\'o 4-manifold
	invariants
	of $X$ and $X_{\cC}$ in terms of the graded Lefschetz number of the concordance
	map
	\[
	\hat{F}_{\cC} \colon \HFKh(Y, K) \to \HFKh(Y, K)
	\]
	defined by the first author~\cite{JCob}. This map preserves the Alexander and
	Maslov
	gradings~\cite{JMComputeCobordismMaps}*{Theorem~5.18}. The graded Lefschetz
	number is the polynomial
	\[
	\Lef_z(\cC) := \sum_{i \in \Z} \Lef \left(\hat{F}_{\cC}|_{\HFKh(Y, K,i)} \colon
	\HFKh(Y, K,i) \to \HFKh(Y, K,i) \right) \cdot z^i.
	\]
	We note that the concordance map $\hat{F}_{\cC}$ on knot Floer homology depends
	on some extra decorations
	that we are suppressing from the notation. Nonetheless, we will see that the
	graded Lefschetz number is
	independent of these decorations.
	
	If $[T] \neq 0 \in H_2(X;\R)$, then we can pick a basis
    $\ve{b} = (b_1,\dots, b_n)$ of $H^2(X; \R)$, such that
	\begin{equation}
	\langle\, b_1, [T] \,\rangle = 1 \text{ and } \langle\, b_i, [T] \,\rangle = 0 \text{ for } i > 1.
	\label{eq:basisof2-forms1}
	\end{equation}
    There are natural isomorphisms $H^2(X; \R) \cong H^2(X_\cC; \R)$
    and $\Spin^c(X) \cong \Spin^c(X_\cC)$.
    By a slight abuse of notation, we will use the same notation for corresponding second cohomology classes and $\Spin^c$ structures on $X$ and $X_\cC$. In particular, the base $\Spin^c$ structure $\frs_0$ on $X$ corresponds to a base $\Spin^c$ structure $\frs_0$ on $X_\cC$, and we define the 4-manifold invariants $\Phi_{X;\ve{b}}$ and $\Phi_{X_\cC;\ve{b}}$ using this correspondence.
	We now state our main result:
	
	\begin{thm}\label{thm:1}
		Let $X$ be a closed, oriented 4-manifold such that $b_2^+(X) \ge 2$.
		Suppose that $T$ is a smoothly embedded 2-torus in $X$ with trivial
		self-intersection,
		such that $[T] \neq 0 \in H_2(X;\R)$. Furthermore, let $\ve{b} =
		(b_1,\dots, b_n)$ be a basis of $H^2(X; \R)$ satisfying equation~\eqref{eq:basisof2-forms1}.
		If $\cC$ is a self-concordance of $(Y,K)$, where $Y$ is a homology 3-sphere, then
		\[
		\Phi_{X_{\cC}; \ve{b}} = \Lef_{z_1}(\cC) \cdot \Phi_{X; \ve{b}}.
		\]
	\end{thm}

	If $\cC$ is the product concordance $(I \times Y, I \times K)$, then
	$\hat{F}_{\cC}$ is the
	identity of $\HFKh(Y, K)$, so $\Lef_z(\cC)$ is the graded Euler characteristic
	of $\HFKh(Y, K)$,
	which is $\Delta_K(t)$. Hence, as a special case, we recover the formula of
	Mark~\cite{Mark}*{Theorem~3.1};
	i.e., the Heegaard Floer version of the Fintushel--Stern knot surgery formula.

	When $\pi_1(X \setminus T) = 1$ and $Y = S^3$, the manifold $X_{\cC}$
	is homeomorphic to $X$.  In contrast, we have the following corollary to Theorem~\ref{thm:1},
	which we prove in Section~\ref{sec:diffeomorphism-types}:
	
	\begin{cor}\label{cor:not-diffeomorphic}
		If $\Lef_z(\cC)\neq 1$ and $\Phi_{X;\ve{b}}\neq 0$, the
	4-manifold $X_{\cC}$ is not diffeomorphic to $X$.
	\end{cor}	
	
	Since $\Lef_z(\cC)$ is always symmetric and satisfies $\Lef_z(\cC)(1) = \pm 1$,
	it is unclear whether, using concordance surgery, we obtain any smooth structures
    not arising from knot surgery. Nonetheless, in \cite{JMZExotic}, we use the techniques of this paper
    to produce infinite families of exotic orientable surfaces in $B^4$.

	We note that the proofs of the knot surgery formula~\eqref{eqn:FS}
	due to Fintushel and Stern for the Seiberg--Witten invariant,
	and to Mark for the Ozsv\'ath--Szab\'o invariant, are based on the skein
	relation for the Alexander polynomial, and hence are only well-suited to
	knots in $S^3$. 	Our theorem applies to a more general setting,
	where $K$ is allowed to be a null-homologous knot in an arbitrary homology 3-sphere
	$Y$. Our proof of Theorem~\ref{thm:1} also extends to the situation
	where we consider a self-concordance $(W,\cC)$ of a pair $(Y,K)$, such that $W$ is an
	integer homology cobordism from $Y$ to itself, though we restrict to the setting
	where $W = I \times Y$ to simplify the notation. The key technical advancement that led
	to this proof is our previous computation
	of the sutured Floer trace and cotrace cobordism maps
	~\cite{JuhaszZemkeContactHandles}*{Theorem~1.1}.

	Our Theorem~\ref{thm:1} could be used to construct
	exotic smooth structures on 4-manifolds with non-trivial fundamental group.
	Suppose that $\pi_1(X \setminus T) = 1$.
	If $\Phi_{X; \ve{b}} \neq 0$, and $K$ and $K'$ are knots in a homology 3-sphere $Y$
	such that $X_K$ and $X_{K'}$ are homeomorphic
	and $\Phi_{X; \ve{b}} \cdot \Delta_K(z)$ and $\Phi_{X;	\ve{b}} \cdot \Delta_{K'}(z)$
	are not equivalent under the action of automorphisms of $H_2(X)$, then $X_K$ and
	$X_{K'}$ are non-diffeomorphic
	4-manifolds with fundamental group~$\pi_1(Y)/\langle [K] \rangle$,
	where $\langle [K] \rangle$ is the normal subgroup of $\pi_1(Y)$ generated by $K$.
		
	After proving Theorem~\ref{thm:1},
	we give an account of the naturality and functoriality of the perturbed versions
    of sutured Floer homology and Heegaard Floer homology,	
	since these are more subtle than in the unperturbed setting,
	and many details are only sketched in the literature.
		
	Finally, we note that it might be possible to carry out our argument
	for the Seiberg--Witten invariant using the work of Zhenkun Li \cite{ZhenkunSM1}
	to construct gluing and cobordism maps for Kronheimer and
	Mrowka's sutured monopole Floer homology \cite{Kronheimer-Mrowka-Sutured-Monopole}.
	A key technical step which has not yet been completed in this program
	is the computation of the induced maps by the trace and cotrace cobordisms,
	which we performed in the setting of sutured Floer homology
	in \cite{JuhaszZemkeContactHandles}*{Theorem~1.1}.

\subsection{Organization}
	
In Sections~\ref{sec:perturbed} and \ref{sec:perturbed-closed-3-manifolds}, we
give an overview of the construction of the perturbed Floer homology groups, and
the perturbed cobordism maps, and we state the properties that are most relevant
to the proof of Theorem~\ref{thm:1}. In Section~\ref{sec:4-manifold-invariant},
we give some background on the Ozsv\'{a}th--Szab\'{o} 4-manifold invariant.  In
Section~\ref{sec:proof-of-concordance-surgery-formula}, we prove
Theorem~\ref{thm:1}. In Sections~\ref{sec:naturality} and
\ref{sec:cob-construct}, we give a proof of the naturality of the perturbed
sutured Floer groups, the well-definedness of the cobordism maps, and also
several useful properties.
	
\subsection{Acknowledgements}
	
We would like to thank Ciprian Manolescu, Thomas Mark, and Zolt\'an Szab\'o for
helpful discussions,
and Ronald Fintushel and Ronald Stern for their comments on the history of this
problem.
We would also like to thank an anonymous referee for a very careful reading and
helpful suggestions.
The first author was supported by a Royal Society Research Fellowship,
and the second author by an NSF Postdoctoral Research Fellowship (DMS-1703685).
This project has received funding from the European Research Council (ERC)
under the European Union's Horizon 2020 research and innovation programme
(grant agreement No 674978).

\section{Perturbing sutured Floer homology by a 2-form}
\label{sec:perturbed}

	Ozsv\'ath and Szab\'o~\cite{genusbounds}*{Section~3.1} defined a version of
	Heegaard Floer homology for closed 3-manifolds perturbed by a second cohomology class,
	which we now extend to sutured manifolds. The unperturbed
    version of sutured Floer homology was defined by the first author~\cite{JDisks},
    and its naturality was shown by Thurston and the authors~\cite{JTNaturality}.

	Let $\Lambda$ denote the Novikov ring over $\F_2$ in a single
	variable~$z$. Its elements are formal sums $\sum_{x \in \R} n_x z^x$,
	where $n_x \in \F_2$, and the set
	\[
	\{\, x \in (-\infty,c] : n_x \neq 0 \,\}
	\]
	is finite for every $c \in \R$. Note that $\Lambda$ is a field.
	
	Suppose that $(M,\g)$ is a balanced sutured manifold, and $\omega$ is a
	closed 2-form on $M$. Then $\omega$ induces an action of $\bF_2[H^1(M,\d M)]
	\iso
	\bF_2[H_2(M)]$ on $\Lambda$, via the formula
	\[
	e^a \cdot z^x = z^{x + \int_a \omega}
	\]
	for $x \in \R$ and $a\in H_2(M)$. We denote by $\Lambda_\omega$ the ring $\Lambda$
	viewed as a module over $\bF_2[H^1(M, \d M)]$.

    For a sutured manifold $(M,\gamma)$, equipped with a closed 2-form $\omega$ and a relative
	$\Spin^c$ structure $\underline{\frs}$, we  write
	$\SFH(M,\gamma,\underline{\frs};\Lambda_{\omega})$
	for the perturbed sutured Floer homology, which we describe in this section.
	Using the terminology of Baldwin and Sivek~\cite{BaldwinSivekContactInvariant},
	the most natural category for
	$\SFH(M,\g,\underline{\frs};\Lambda_\omega)$ is
	the category of \emph{projective transitive systems}.
	See Section~\ref{sec:transitive-system} for a precise definition.
	We state the following version of naturality for perturbed sutured
	Floer homology:

	\begin{thm}\label{thm:naturality-perturbed}
        Suppose $(M,\g)$ is a sutured
        manifold, and $\omega$ is a closed 2-form on $M$.
		\begin{enumerate}
			\item
			If 	$\underline{\frs}\in \Spin^c(M,\gamma)$, then
			$\SFH(M,\gamma,\underline{\frs};\Lambda_\omega)$
			forms a projective transitive system of $\Lambda$-modules, indexed by the set
			of pairs $(\cH,J)$,
			where $\cH$ is an admissible diagram for $(M,\gamma)$, and $J$ is a generic
			almost complex structure.
			\item If $\omega=d \eta$ for a 1-form $\eta$, then
			$\SFH(M,\gamma;\Lambda_{\omega})$
			(the sum over all $\Spin^c$ structures) forms a projective transitive system
			of $\Lambda$-modules, indexed by the set of pairs $(\cH,J)$, as above.
		\end{enumerate}
	\end{thm}
	
    We will prove Theorem~\ref{thm:naturality-perturbed} in
	Section~\ref{sec:naturality}, though we describe the construction of the
	perturbed groups
	in Section~\ref{sec:perturbed-complexes}.
	
	\begin{rem}
    Our construction of $\SFH(M,\gamma;\Lambda_{\omega})$ gives neither
	a genuine transitive system when we restrict to a single $\Spin^c$ structure on $M$, nor
	a projective transitive system when we sum over all $\Spin^c$ structures.
	See Example~\ref{rem:projective-monodromy} and Lemma~\ref{lem:monodromy} for
	counterexamples.
	\end{rem}
	
	\subsection{Transitive systems and their morphisms}
	\label{sec:transitive-system}
	\begin{define} Suppose that $\cC$ is a category and $I$ is a set. A
		\emph{transitive system in $\cC$, indexed by $I$}, is a collection of objects
		$(X_i)_{i\in I}$, as well as a distinguished morphism
		$\Psi_{i\to j}\colon X_i\to X_j$ for each $(i,j)\in I\times I$, such that
		\begin{enumerate}
			\item $\Psi_{j\to k}\circ \Psi_{i\to j}= \Psi_{i\to k}$, and
			\item $\Psi_{i\to i}=\id_{X_i}$.
		\end{enumerate}
	\end{define}

	\begin{example}\label{ex:categories}  Transitive systems in the following
		categories are important to
		our present paper.
		\begin{enumerate}[label=(T-\arabic*), ref=T-\arabic*]
			\item\label{transitive-1} The category
			$\cC = \cR\text{--}\MOD$ of left modules over a ring $\cR$. The morphism set
			$\Hom_{\cC}(X_1,X_2)$ is equal to the set
			$\Hom_{\cR}(X_1,X_2)$ of $\cR$-module homomorphisms from $X_1$ to $X_2$.
			\item \label{transitive-4} The projectivized category of $\Lambda$-modules
			$\cC=\mathsf{P}(\Lambda\text{--}\MOD)$. The objects are $\Lambda$-modules
			and the morphism
			set $\Hom_{\cC}(X_1,X_2)$ is the projectivization of
			$\Hom_{\Lambda}(X_1,X_2)$
			under the action of elements of $\Lambda$ of the form $z^x\in \Lambda$.
			\item\label{transitive-2} The homotopy category
			$\cC=\mathsf{K}(\cR\text{--}\MOD)$
			of chain complexes over the ring $\cR$. The objects are chain complexes over
			$\cR$. If
			$X_1$ and $X_2$
			are two chain complexes, the set of $\cR$-module homomorphisms
			$\Hom_{\cR}(X_1,X_2)$ is a
			chain complex with differential $\d_{\Hom}(f) = f\circ \d_{X_1}-\d_{X_2}\circ f$
            for $f \in \Hom_{\cR}(X_1,X_2)$. The morphism set
			$\Hom_{\cC}(X_1,X_2)$ in $\cC$
			is the homology $H_*(\Hom_{\cR}(X_1,X_2))$.
			Equivalently, $\Hom_{\cC}(X_1,X_2)$ is the set of chain maps modulo chain
			homotopy.
			\item\label{transitive-3} The projectivized homotopy category
			$\cC=\mathsf{P}(\mathsf{K}(\Lambda\text{--}\MOD))$.
			The objects of $\cC$ are chain complexes over $\Lambda$. The morphism set
			$\Hom_{\cC}(X_1,X_2)$ is the projectivization of
			$H_*(\Hom_{\Lambda}(X_1,X_2))$
			under the action of elements of $\Lambda$ of the form $z^x$.
		\end{enumerate}
	\end{example}

	The categories in \eqref{transitive-1} and \eqref{transitive-2}
	are preadditive (i.e., the morphism sets are abelian groups), while the
	categories in ~\eqref{transitive-4} and \eqref{transitive-3} are not.
	In these latter categories, composition of projective morphisms is
	well-defined, though addition of morphisms is not.
	
	Following the terminology of Baldwin and Sivek
	\cite{BaldwinSivekContactInvariant},
	we call a transitive system over  one of the categories~\eqref{transitive-4}
	and \eqref{transitive-3} a \emph{projective transitive system}.
    In category~\eqref{transitive-4}, given morphisms
    $f$, $g \in \Hom_{\Lambda}(X_1, X_2)$,
    we will use the notation $f \doteq g$ if $f = z^x \cdot g$ for some $x \in \R$.
    Similarly, in case~\eqref{transitive-3}, given chain maps
    $\phi$, $\psi \in H_*(\Hom_{\Lambda}(X_1, X_2))$,
    we write $\phi \dotsim \psi$ if $\phi \simeq z^x \cdot \psi$ for some $x \in \R$,
    where $\simeq$ denotes chain homotopy equivalence. If $\phi\dotsim \psi$,
    we say $\phi$ and $\psi$ are \emph{projectively equivalent}.
	Finally, if $X$ is a $\Lambda$-module and $a$, $b\in X$, we write $a\doteq b$
    if $a=z^x\cdot b$ for some $x\in \R$.
	
    There is a natural notion of morphism between transitive systems:
	
    \begin{define}
		If $(C_i)_{i\in I}$ and $(D_j)_{j\in J}$ are two transitive systems in the
		category $\cC$, a \emph{morphism of transitive systems} is a collection of morphisms
		\[
		F_{(i,j)} \colon C_i\to D_j
		\]
		in $\cC$ such that
		\[
		\Psi_{j\to j'}\circ  F_{(i,j)}\circ \Psi_{i'\to i}= F_{(i',j')}
		\]
		for all $i$, $i'\in I$ and $j$, $j'\in J$.
	\end{define}

	\begin{rem}\label{rem:1-map-determines-morphism}
        If $f\colon C_{i_0} \to D_{j_0}$ is an element of
		$\Hom_{\cC}(C_{i_0},D_{j_0})$ for some fixed $i_0\in I$ and $j_0\in J$, then $f$
		induces a unique morphism $F_{(i,j)}$ of transitive systems from $(C_i)_{i\in
			I}$ to $(D_j)_{j\in J}$, given by
		\[
		F_{(i,j)}=\Psi_{j_0\to j}\circ f\circ \Psi_{i\to i_0}.
		\]
	\end{rem}		
	
	If $\cC$ is a category, then the collection of transitive systems over $\cC$
	itself forms a category, for which we write $\cT(\cC)$. Hence, we can define
	a transitive system of transitive systems over $\cC$.

	\begin{rem}\label{rem:system-of-systems}
		If $\cX=((X_{ij})_{j\in J_i})_{i\in I}$ is a transitive system in $\cT(\cC)$,
		we may naturally view $\cX$ as a transitive system over $\cC$ indexed by
		$K:=\bigcup_{i\in I} J_i.$
	\end{rem}

	\subsection{The perturbed chain complexes}
    \label{sec:perturbed-complexes}	
	In this section, we define the perturbed sutured Floer complexes.
    We use the cylindrical reformulation of Heegaard Floer homology, due to
    Lipshitz~\cite{LipshitzCylindrical}.
	Suppose $(M,\gamma)$ is a balanced sutured manifold with a closed 2-form
	$\omega$.
	If $\cH=(\Sigma,\as,\bs)$ is an admissible diagram, we pick an almost
	complex structure on $\Sigma\times I\times \R$ that is tamed by
	the split symplectic form. The surface $\Sigma$ splits $M$ into two sutured
	compression bodies, for which we write $U_{\sca}$ and $U_{\scb}$. We let
	$D_{\sca}$ and $D_{\scb}$ be two choices of compressing disks for $U_{\sca}$
	and $U_{\scb}$, equipped with radial foliations, such that $D_{\sca}$
	intersects $\Sigma$ along $\as$, and similarly for $D_{\scb}$.
	
	A homotopy class $\phi\in \pi_2(\xs,\ys)$ of disks determines a 2-chain
	$\cD(\phi)$ on $\Sigma$, which has boundary on $\as\cup\bs$.
	We cone $\cD(\phi)$ along the compressing disks $D_{\sca}$ and $D_{\scb}$
	to obtain a 2-chain $\tilde{\cD}(\phi)$. We note
	that the 2-chain $\tilde{\cD}(\phi)$ depends on the choice of
	radial foliations on $D_{\a}$ and $D_{\b}$.
	The 2-chain $\tilde{\cD}(\phi)$ is closed
	if and only if $\ve{x}=\ve{y}$.
	
	We define
	\[
	A_\omega(\phi) := \int_{\tD(\phi)} \omega.
	\]
	When the choice of $\omega$ is clear from the context, we just write $A(\phi)$.
	
	There is a map $H \colon \pi_2(\xs,\xs) \to H_2(M)$, obtained by coning
	off the periodic domain $\cD(\phi)$ for $\phi \in \pi_2(\xs,\xs)$;
	see \cite{JDisks}*{Definition~3.9}. In particular,
	\[
	H(\phi) = \left[\tD(\phi)\right].
	\]

	The chain complex $\CF(\cH, \underline{\frs};\Lambda_{\omega})$
	is the free $\Lambda$-module generated by intersection points
	$\xs\in \bT_{\sca}\cap \bT_{\scb}$ which satisfy $\frs(\xs)=\underline{\frs}$.
	The differential is given by  counting holomorphic
	curves in $\Sigma\times I\times \R$ via the formula
	\[
	\d \ve{x} := \sum_{\ys \in \bT_{\sca} \cap \bT_{\scb}} \sum_{\substack{\phi \in
			\pi_2(\xs,\ys)\\
			\mu(\phi) = 1}} \left(|\cM(\phi)/\R| \mod 2 \right) \cdot z^{A(\phi)}
	\cdot
	\ys
	\]
	for $\ve{x} \in \bT_{\sca} \cap \bT_{\scb}$. The fact that $\d^2 = 0$ follows
	by analyzing the ends of the 1-dimensional moduli spaces $\cM(\phi)/\R$ for
	classes $\phi$ with Maslov index 2.
	We set
	\[
	\SFH(\cH,\underline{\frs};\Lambda_{\omega}) :=
	H_*\left(\CF(\cH,\underline{\frs};\Lambda_{\omega}), \d \right).
	\]
	The group $\SFH(\cH,\underline{\frs};\Lambda_{\omega})$ also depends
	on $J$ and the compressing disks, though we omit the extra data from the
	notation.

    \subsection{Perturbed sutured cobordism maps}\label{sec:sutured-cobordism-maps}

	In \cite{JCob}, the first author defined a notion of cobordism between
	sutured manifolds, and constructed functorial cobordism maps.
	
	\begin{define}\label{def:sutured-cobordism}
    A \emph{cobordism of sutured manifolds}
	\[
	\cW=(W,Z,[\xi])\colon (M_0,\g_0)\to (M_1,\g_1)	
	\]
	is a triple such that
	\begin{enumerate}
		\item $W$ is a compact, oriented 4-manifold with boundary,
		\item $Z$ is a compact, codimension-0 submanifold with boundary
		 of $\d W$, and $\d W\setminus \Int (Z)=-M_0\sqcup M_1$,
		\item $[\xi]$ is an equivalence class of positive contact structures
		 on $Z$ (see \cite{JCob}*{Definition~2.3}),
         such that $\d Z$ is a convex surface with dividing
		 set $\gamma_i$ on $\d M_i$, for $i\in \{0,1\}$.
	\end{enumerate}
	\end{define}

	In Section~\ref{sec:cob-construct}, we will define perturbed versions
	of the sutured manifold	cobordism maps. If $\cW = (W,Z,[\xi])$
	is a sutured manifold cobordism
	from $(M_0,\g_0)$ to $(M_1,\g_1)$, and $\omega$ is a closed 2-form on $W$,
	then we will define a chain map
	\[
	F_{\cW;\omega} \colon \SFH(M_0,\g_0; \Lambda_{\omega|_{M_0}}) \to \SFH(M_1,\g_1;
	\Lambda_{\omega|_{M_1}}),
	\]
	which is only well-defined up to an ambiguity described in
	Proposition~\ref{prop:perturbedcobordismmapswelldefined}.

	If $\cH$ is a Heegaard diagram for $(M,\g)$, we can view
	\[
	\SFH(\cH; \Lambda_{\omega})=\bigoplus_{\ufrs\in \Spin^c(M, \g)}
	\SFH(\cH,\ufrs;\Lambda_{\omega}).	
	\]
	Consequently, there are inclusion and projection maps
	\[
	i_{\ufrs}\colon \SFH(\cH,\ufrs;\Lambda_{\omega}) \to
	\SFH(\cH;\Lambda_{\omega})\quad \text{and} \quad \pi_{\ufrs}\colon
	\SFH(\cH;\Lambda_{\omega})\to \SFH(\cH,\ufrs; \Lambda_{\omega}).
	\]

	\begin{prop}\label{prop:perturbedcobordismmapswelldefined}
		Suppose $\cW=(W,Z,[\xi])\colon (M_0,\g_0) \to (M_1,\g_1)$ is a sutured manifold
		cobordism, and $\omega$ is a closed 2-form on $W$.
		\begin{enumerate}
			\item\label{prop:well-def:restricted}  If $\underline{\frs}_i \in
			\Spin^c(M_i,\g_i)$ for $i \in \{0,1\}$, then the map
			\[
			\pi_{\ufrs_1} \circ F_{\cW;\omega} \circ i_{\ufrs_0} \colon \SFH(M_0,\g_0,\ufrs_0;
			\Lambda_{\omega|_{M_0}}) \to \SFH(M_1,\g_1,\ufrs_1; \Lambda_{\omega|_{M_1}})
			\]
			is well-defined up to an overall factor of $z^x$, for $x\in \R$.
			\item\label{prop:well-def:total} More generally, if $[\omega|_{M_1}]=0$, then
			$F_{\cW;\omega}\circ i_{\ufrs_0}$ is
			well-defined up to an overall factor of $z^x$. If $[\omega|_{M_0}]=0$, then
			$\pi_{\ufrs_1} \circ F_{\cW;\omega}$ is well-defined up to a factor of $z^x$.
			If $[\omega|_{M_0}]=0$ and $[\omega|_{M_1}]=0$, then the total map
			$F_{\cW;\omega}$ is well-defined up to an overall factor of $z^x$.
		\end{enumerate}
	\end{prop}
	
	The main idea of the construction
	is to incorporate the coning construction
	of Ozsv\'{a}th and Szab\'{o}~\cite{genusbounds} at each step
	of the construction of the
	unperturbed sutured cobordism maps in~\cite{JCob}.
	In Section~\ref{sec:cob-construct},
	we describe the construction in detail, and prove
	Proposition~\ref{prop:perturbedcobordismmapswelldefined}.
	We note that, to define the total cobordism map in part~\eqref{prop:well-def:total}
	of Proposition~\ref{prop:perturbedcobordismmapswelldefined},
	we use our formula for the sutured trace cobordism map
	\cite{JuhaszZemkeContactHandles}*{Theorem~1.1}; see Section~\ref{sec:def-general-cob}.
	In Section~\ref{sec:composition-law},
    we will prove the following  composition law for the perturbed sutured cobordism maps:	
	
	\begin{prop}\label{prop:composition-law}
        Suppose the sutured manifold cobordism $\cW=(W,Z,[\xi])$ decomposes
		as $\cW_2 \circ \cW_1$, where
		\[
		\cW_1 = (W_1, Z_1, [\xi_1])\colon (M_0,\g_0) \to (M_1,\g_1)\quad \text{and} \quad
		\cW_2 = (W_2, Z_2, [\xi_2])\colon (M_1,\g_1)\to (M_2,\g_2).
		\]
		Let $\omega$ be a closed 2-form on $W$, and write $\omega_1 = \omega|_{W_1}$
		and $\omega_2 = \omega|_{W_2}$.
		\begin{enumerate}
			\item \label{prop:comp-law-total} If $[\omega]$ restricts trivially to $M_0$,
			$M_1$, and $M_2$, then
			\[
			F_{\cW;\omega} \doteq F_{\cW_2;\omega_2} \circ F_{\cW_1;\omega_1}.
			\]
			\item \label{prop:comp-law-restricted} More, generally, if $[\omega]$
			restricts trivially to $M_1$ and $M_2$, and $\ufrs_0 \in
			\Spin^c(M_0,\g_0)$, then
			\[
			F_{\cW;\omega}\circ i_{\ufrs_0} \doteq F_{\cW_2;\omega_2}\circ
			F_{\cW_1;\omega_1}\circ i_{\ufrs_0}.
			\]
            Similar formulas hold if $[\omega]$ restricts trivially to both $M_0$ and
			$M_1$, or to just $M_1$.
		\end{enumerate}
	\end{prop}

\subsection{Alexander gradings and perturbations on cylinders}
	
	We now state a simple formula for the sutured cobordism map for a perturbation
	of the identity cobordism of a knot complement, which we need for our proof of
	Theorem~\ref{thm:1}.
	
	Suppose that $K$ is a knot in an integer homology sphere $Y$.
    Let $Y(K)$ denote $Y\setminus N(K)$, decorated with two
	oppositely oriented meridional sutures.	
	A sutured Heegaard diagram $(\Sigma,\as,\bs)$ for $Y(K)$ is
	equivalent to a doubly-pointed diagram for $(Y,K)$: To obtain a doubly-pointed
	diagram from $(\Sigma,\as,\bs)$, we collapse each of the boundary components of
	$\Sigma$ to a basepoint. We let $w$ denote the point where $K$ intersects
	$\Sigma$ negatively, and $z$ denotes the point where $K$ intersects $\Sigma$ positively.	
	There is a tautological isomorphism
	\[
	\HFKh(Y,K)\iso \SFH(Y(K)),
	\]
	since the generators and differential coincide.

	The relative Alexander grading on $\HFKh(Y,K)$ is given as follows. If
	$\xs$, $\ys \in \bT_{\sca} \cap \bT_{\scb}$, then we pick a class $\phi \in
	\pi_2(\xs,\ys)$ on $(\Sigma,\as,\bs,w,z)$ (possibly going over $w$ and $z$).
    The relative Alexander grading is given by the formula
	\[
	A(\xs,\ys)=n_{z}(\phi)-n_{w}(\phi).
	\]
	The relative Alexander grading admits an absolute lift, which can be specified
	by a symmetry requirement on $\HFKh(Y,K)$; see \cite{OSKnots}*{Section~3.5}.
	
    Let $\cS_K$ be a Seifert surface of $K$.
	Let
    \[
    \omega_{\cS_K} \in \Omega^2(I\times Y(K), \d I\times Y(K))
    \]
    be a closed 2-form dual to $\{\tfrac{1}{2}\} \times \cS_K$ under
    Poincar\'{e}--Lefschetz duality
	\[
	H^2(I\times Y(K), \d I\times Y(K)) \iso H_2(I\times Y(K), I\times \d Y(K)).
	\]
	By definition, $\omega_{\cS_K}$ vanishes on $\d I\times Y(K)$.
	
	\begin{lem}\label{lem:Alexander-grading}
		Up to an overall factor of $z^\a$, the map $F_{I\times Y(K);\omega_{\cS_K}}$
        is given by
		\[
		F_{I\times Y(K);\omega_{\cS_K}}(z^x\cdot \xs)=z^{x-A(\xs)} \cdot \xs,
		\]
		where $A(\xs)$ denotes the Alexander grading.
	\end{lem}

We will prove Lemma~\ref{lem:Alexander-grading} at the end of
Section~\ref{sec:cylinders}.

\subsection{Changing the 2-form on $W$}

We now state another result which will be helpful for proving Theorem~\ref{thm:1}:

\begin{lem}\label{lem-modify-2-form-on-interior}
Suppose that
$\cW=(W,Z,[\xi])\colon (M_0,\g_0)\to (M_1,\g_1)$ is a sutured manifold cobordism,
$\omega$ is a closed 2-form  on $W$, and $\eta$ is a 1-form that vanishes on a
neighborhood of $M_0$ and $M_1$. If $[\omega]$ vanishes on $M_0 \cup M_1$, then
\[
F_{\cW;\omega}\dot{=}F_{\cW;\omega+d\eta}.
\]
If $[\omega]$ is non-vanishing on $M_0$ and $M_1$, then the above equation holds
when restricted to fixed $\Spin^c$ structures on $M_0$ and $M_1$.
\end{lem}

We will prove Lemma~\ref{lem-modify-2-form-on-interior} in
Section~\ref{sec:change-2-formW}.

\section{Perturbed Heegaard Floer homology of closed 3-manifolds}
	\label{sec:perturbed-closed-3-manifolds}
    We review some background on Heegaard Floer homology, due to Ozsv\'{a}th and
	Szab\'{o} \cite{OSDisks} \cite{OSTriangles}. To a closed 3-manifold $Y$ with
	a $\Spin^c$ structure $\frs$, Ozsv\'{a}th and Szab\'{o} assign
	$\bF_2[U]$-modules $\HF^-(Y,\frs)$, $\HF^\infty(Y,\frs)$, and $\HF^+(Y,\frs)$
	that fit into a long exact sequence
	\begin{equation}
	\cdots \xrightarrow{\delta} \HF^-(Y,\frs)\to HF^\infty(Y,\frs)\to
	\HF^+(Y,\frs)\xrightarrow{\delta} \HF^-(Y,\frs)\to \cdots
	\label{eq:exactsequence}
	\end{equation}
    There is also an
	$\bF_2$-vector space $\HFh(Y,\frs)$.
	
	If $W$ is a cobordism from $Y_0$ to
	$Y_1$, and $\frs \in \Spin^c(W)$ restricts to $\frs_0$ on $Y_0$
	and to $\frs_1$ on $Y_1$, then there are maps
	\[
	F_{W,\frs}^\circ \colon \HF^\circ(Y_0,\frs_0) \to \HF^\circ(Y_1,\frs_1)
	\]
	for $\circ \in \{-,\infty,+,\wedge\}$ that commute with the maps in the long
	exact sequence in equation~\eqref{eq:exactsequence}.

	If $\omega$ is a closed 2-form on $Y$,
	Ozsv\'{a}th and Szab\'{o} \cite{genusbounds} described an
	$\bF_2[H^1(Y)]$-module denoted	$\HF^\circ(Y,\frs; \Lambda_{\omega})$,
	using the same coning
	procedure we described in Section~\ref{sec:perturbed-complexes}.
	Similarly, if
	$\ve{\omega} = (\omega_1,\dots, \omega_n)$ is an $n$-tuple of closed 2-forms on $Y$,
	we can define the $\bF_2[H^1(Y)]$-module $\HF^{\circ}(Y,\frs;
	\Lambda_{\omegas})$, which is also a $\Lambda_n[U]$-module, where $\Lambda_n$
    is the $n$-variable Novikov ring over $\bF_2$. In this section, we focus on perturbing
    by a single 2-form, to simplify the notation.
		
    Ozsv\'ath and Szab\'o~\cite{genusbounds} defined perturbed versions of their
	cobordism maps (and more generally, fully twisted versions in \cite{OSTriangles}).
	The naturality and functoriality results described above for sutured Floer
	homology have analogues for the perturbed versions of the closed
	3-manifold invariants, which we state here.
	
	\begin{thm}
		\begin{enumerate}
			\item Suppose $Y$ is a closed 3-manifold with a chosen basepoint and a closed
			2-form $\omega$.
			If $\frs\in \Spin^c(Y)$ and $\circ \in \{-,\infty,+, \wedge \}$,
            then $\HF^\circ(Y,\frs;\Lambda_{\omega})$
            forms a projective transitive system of $\Lambda[U]$-modules,
			indexed by the set of pairs $(\cH,J)$, where $\cH$ is an $\frs$-admissible
			diagram of $Y$, and $J$ is a generic almost complex structure.
			\item Suppose $W$ is a connected, oriented cobordism from $Y_0$ to $Y_1$,
            with a chosen path connecting the basepoints of $Y_0$ and $Y_1$,
			a $\Spin^c$ structure $\frs\in \Spin^c(W)$,
			and a closed 2-form $\omega$ on $W$.
            Then the cobordism map
			\[
			F^\circ_{W,\frs;\omega}\colon
\HF^\circ(Y_0,\frs|_{Y_0};\Lambda_{\omega|_{Y_0}}) \to
			\HF^\circ(Y_1,\frs|_{Y_1};\Lambda_{\omega|_{Y_1}})
			\]
			due to Ozsv\'ath and Szab\'o~\cite{genusbounds}
            is well-defined up to overall multiplication by $z^x$ for $x \in
\R$.
		\end{enumerate}
	\end{thm}
	
	Ozsv\'{a}th and Szab\'{o}'s construction of the perturbed cobordism maps is
    similar to the construction we describe in Section~\ref{sec:cob-construct} for
    sutured Floer homology. One important difference is how the maps are associated to
	$\Spin^c$ structures on $W$. If $W$ is decomposed as $W_1 \cup W_2 \cup W_3$,
	where $W_i$ is an index $i$ handle cobordism, then the restriction map
    $\Spin^c(W) \to \Spin^c(W_2)$ is an isomorphism.
    If $(\Sigma,\as,\bs,\bs',w)$ is a triple for
	the 2-handle attachment, Ozsv\'{a}th and Szab\'{o}
    \cite{OSTriangles}*{Section~8.1.4} define a map
	\[
	\frs_{w} \colon \pi_2(\xs,\ys,\zs)\to \Spin^c(W_2).
	\]
	The map $F^\circ_{W,\frs;\omega}$ counts only triangles with $\frs_w(\psi) = \frs|_{W_2}$. 
    Note that this construction differs slightly from the $\Spin^c$ restricted versions of the perturbed 
    sutured cobordism maps we gave in Section~\ref{sec:sutured-cobordism-maps}, 
    which took the form $\pi_{\ufrs_1} \circ F_{\cW;\omega} \circ i_{\ufrs_0}$.
	
	The $\Spin^c$ composition law is slightly subtle in the perturbed setting,
    since we are working in a projectivized category; see~Example~\ref{ex:categories}.
    The morphism sets in a projectivized category are not abelian groups, so sums of
	maps are not well-defined. Nonetheless, a $\Spin^c$ composition law can still
	be stated, as we now describe.
	
	Suppose that $\frS \subset \Spin^c(W)$ is a subset of $\Spin^c$ structures.
	We suppose that each $\frs \in \frS$ has the
	same restriction to $\d W$, unless $[\omega|_{\d W}] = 0$.
    If $\circ\in \{-,\infty\}$, we must also assume that there are only finitely many
    $\frs \in \frS$ such that $F^\circ_{W,\frs;\omega} \neq 0$. In this
	situation, we may define a cobordism map
	\[
	F^\circ_{W,\frS;\omega} \colon \HF^\circ(Y_0;\Lambda_{\omega|_{Y_0}}) \to
	\HF^\circ(Y_1;\Lambda_{\omega|_{Y_1}}),
	\]
    which is well-defined up to multiplication by $z^x$ for some $x \in \R$.
	The 2-handle portion of the map $F^\circ_{W,\frS;\omega}$
	counts triangles such that $\frs_w(\psi)$ is the restriction
    of an element of $\frS$.
	
	By construction, we may find representatives
	of the maps $F^\circ_{W,\frs;\omega}$ for
	$\frs \in \frS$ such that
	\[
	F^\circ_{W,\frS;\omega} \doteq \sum_{\frs\in \frS} F^\circ_{W,\frs;\omega}.
	\]
	The proof of the composition law given by Ozsv\'{a}th and Szab\'{o}
	\cite{OSTriangles}*{Theorem~3.4} extends to give the following:

	\begin{prop}\label{prop:composition}
        Suppose $W$ is a cobordism which decomposes as $W_2 \circ W_1$. Suppose further
		that $\omega$ is a closed 2-form on $W$, and
		$\frS_1 \subset \Spin^c(W_1)$ and $\frS_2 \subset \Spin^c(W_2)$
        are subsets as above. Let
        \[
        \frS(W,\frS_1,\frS_2) = \{\, \frs \in \Spin^c(W) :
         \frs|_{W_1} \in \frS_1 \text{ and } \frs|_{W_2} \in \frS_2 \,\}.
        \]
        Then
		\[
		F^\circ_{W,\frS(W,\frS_1,\frS_2);\omega}
		\doteq F^\circ_{W_2, \frS_2; \omega|_{W_2}} \circ
		F^\circ_{W_1, \frS_1; \omega|_{W_1}}.
		\]
	\end{prop}

    We have the following analogue of Lemma~\ref{lem-modify-2-form-on-interior}:

    \begin{lem}\label{lem-modify-2-form-on-interior-closed}
    Suppose that $W \colon Y_0 \to Y_1$ is a cobordism of 3-manifolds,
    $\frS \subset \Spin^c(W)$ is a set of $\Spin^c$ structures as above,
    $\omega$ is a closed 2-form  on $W$, and $\eta$ is a 1-form that vanishes on a
    neighborhood of $Y_0$ and $Y_1$. If $[\omega]$ vanishes on $Y_0 \cup Y_1$, then
    \[
    F^\circ_{W, \frS; \omega} \dot{=} F_{W, \frS; \omega + d\eta}.
    \]
    If $[\omega]$ is non-vanishing on $Y_0$ and $Y_1$, then the above equation holds
    when restricted to fixed $\Spin^c$ structures on $Y_0$ and $Y_1$.
    \end{lem}

    \begin{proof}
      This can be shown similarly to Lemma~\ref{lem-modify-2-form-on-interior};
      see Section~\ref{sec:change-2-formW}.
    \end{proof}

    \begin{lem}\label{lem:perturbed-coincide}
    Let $W$ be a cobordism from $Y_0$ to $Y_1$,
    and $\omega$ a closed 2-form on $W$ that vanishes on $\d W$.
    Furthermore, let $\frS \subset \Spin^c(W)$ be a set of $\Spin^c$ structures.
    If $\circ \in  \{-,\infty\}$, we also assume there are only finitely
    many $\frs \in \frS$ for which $F^\circ_{W, \frs} \neq 0$.
    If $\frs_0 \in \Spin^c(W)$ is an arbitrary base $\Spin^c$ structure, then
    \begin{equation}\label{eqn:normalization}
    F_{W, \frS; \omega}^\circ \doteq
	\sum_{\frs \in \frS} z^{\langle i_*(\frs-\frs_0) \cup [\omega],
    [W, \d W] \rangle} \cdot F_{W,\frs}^\circ.
    \end{equation}
    \end{lem}

    We will prove
    Lemma~\ref{lem:perturbed-coincide} in Section~\ref{sec:normalization}.

	\begin{rem}
	As a consequence of Lemma~\ref{lem:perturbed-coincide},
	if $\omega$ is a closed 2-form on $W$ that vanishes on $\d W$, then $F_{W,\frs;\omega}^\circ\doteq F_{W,\frs}^\circ.$
	We note that it is natural to normalize the perturbed maps in this situation by defining
	\[
	F_{W,\frs; \omega}^\circ := z^{\langle\, i_*(\frs-\frs_0) \cup [\omega],
		[W,\d W] \,\rangle} \cdot  F_{W,\frs}^\circ,
	\]
	and
	\[
	F_{W; \omega}^\circ = \sum_{\frs \in \Spin^c(W)}  F_{W,\frs; \omega}^\circ =
	\sum_{\frs\in \Spin^c(W)} z^{\langle i_*(\frs-\frs_0) \cup [\omega], [W, \d W]
		\rangle}
	\cdot F_{W,\frs}^\circ,
	\]
	for $\circ\in \{\wedge,+\}$. For $\circ\in \{-,\infty\}$,
	we may take this convention in the
	case when $F_{W,\frs}^\circ$ is non-vanishing for only finitely many $\frs$.
	It is straightforward to see that this normalization convention
	is compatible with the composition law.
	\end{rem}

\section{Background on the Ozsv\'ath--Szab\'o mixed invariants}
	\label{sec:4-manifold-invariant}
		
	For a closed 4-manifold $X$ with $b_2^+(X)\ge 2$, Ozsv\'{a}th and Szab\'{o}
	defined a map
	\[
	\Phi_X \colon  \Spin^c(X)\to \bF_2.
	\]
	We write $\Phi_{X,\frs}$ for the value of $\Phi_X$ on $\frs$. The map $\Phi_X$ is
	referred to as the \emph{mixed invariant} of $X$,
	because it uses both $\HF^+$ and $\HF^-$.
	
	The map $\Phi_X$ is defined by picking a connected,
	codimension one submanifold
	$N \subset X$ that cuts $X$ into two pieces, $W_1$ and $W_2$, such that
	$b_2^+(W_i) > 0$, and such that the restriction map
	\[
	H^2(X) \to H^2(W_1) \oplus H^2(W_2)
	\]
	is an injection. Such a cut is called \emph{admissible}.
	If we view $W_1$ as a cobordism from $S^3$ to $N$, and $W_2$ as
	a cobordism from $N$ to $S^3$, the maps $F_{W_1,\frs|_{W_1}}^\infty$ and
	$F_{W_2,\frs|_{W_2}}^\infty$
	vanish \cite{OSTriangles}*{Lemma~8.2}. Consequently, $F_{W_1,\frs_1}^-$ may
	be factored to
	have codomain
	\[
	\HF^-_{\red}(N,\frs|_N):=\ker\left(\HF^-(N,\frs|_N)\to
	\HF^\infty(N,\frs|_N)\right),
	\]
	and $F_{W_2,\frs_2}^+$ may be factored to have domain
	\[
	\HF^+_{\red}(N,\frs|_N):=\coker\left(\HF^\infty(N,\frs_N)\to
	\HF^+(N,\frs|_N)\right).
	\]
	The boundary map $\delta$ in the long exact sequence~\eqref{eq:exactsequence}
    induces an isomorphism between
	$\HF^+_{\red}(N,\frs|_N)$ and $\HF^-_{\red}(N,\frs|_N)$.
	
	The invariant $\Phi_{X,\frs}$ is defined as the coefficient
	of the bottom-graded generator $\Theta_+$ of $\HF^+(S^3)$ in the expression
	\[
	\left(F_{W_2,\frs|_{W_2}}^+\circ \delta^{-1} \circ
	F_{W_1,\frs|_{W_1}}^-\right)(1),
	\]
	where $1$ denotes the top-graded generator of $\HF^-(S^3) \iso \bF_2[U]$.
	Ozsv\'ath and Szab\'o prove that this is independent of the admissible cut $N$.
	
	We now describe how to compute the mixed invariants using
	the perturbed cobordism maps.
    To do that, we will need the following two results:

    \begin{lem}\label{lem:cut}
    Let $X$ be a closed, oriented 4-manifold with $b_2^+(X) \ge 2$, and let $b \in H^2(X)$.
    Given an admissible cut $X = W_1 \cup_N W_2$, there is a closed 2-form $\omega$ on $X$ such that
    \begin{enumerate}
      \item $[\omega] = b \in H^2(X; \R)$, and
      \item $\omega|_N = 0$.
    \end{enumerate}
    \end{lem}

    \begin{proof}
        Choose $\varphi \in \Omega^2(X)$ such that $[\varphi] = b$.
        Since $N$ gives an admissible cut, the coboundary map $H^1(N) \to H^2(X)$ is zero.
		This is Poincar\'{e} dual to the inclusion $H_2(N) \to H_2(X)$,
		so this is trivial as well. Hence, the
		restriction map from $H^2(X;\R)$ to $H^2(N;\R)$ is trivial. In particular,
		$[\varphi|_N] = 0$ in $H^2(N;\R)$, and so there is a 1-form $\eta \in \Omega^1(N)$
        such that $\varphi|_N = d\eta$.

        Let $\nu(N)$ be a tubular neighborhood of $N$ in $X$,
        and write $p \colon \nu(N) \to N$ for the projection.
        Choose a smooth function $f$ on $X$ that is $0$ outside $\nu(N)$,
        and is $1$ on a neighborhood of $N$ contained in the interior of $\nu(N)$.
        We define
        \[
        \omega := \varphi - d(f \cdot p^*\eta).
        \]
        Then $\omega$ satisfies the required conditions.
    \end{proof}

    \begin{lem}\label{lem:finiteness}
        Let $X$ be a closed, oriented 4-manifold with $b_2^+(X) > 1$,
        and let $X = W_1 \cup_N W_2$ be an admissible cut.
        If $\omegas$ is a tuple of closed 2-forms on $X$ that vanish on $N$, then
        $F_{W_1, \frt; \omegas|_{W_1}}^-$ and $F_{W_2,\fru;\omegas|_{W_2}}^+$ are non-zero for only finitely
        many $\frt \in \Spin^c(W_1)$ and $\fru\in \Spin^c(W_2)$.
    \end{lem}

    \begin{proof}
        By Lemma~\ref{lem:perturbed-coincide},
        it suffices to show this for the unperturbed maps $F^-_{W_1,\frt}$ and $F^+_{W_2,\fru}$.
        Note that $F^-_{W_1,\frt}$ has image in $\HF^-_{\red}(N)$ for every $\frt \in \Spin^c(W_1)$.
        Let $d \in \N$ be such that $U^d \cdot \HF^-_{\red}(N) = \{0\}$.
        If $1$ is the generator of $\HF^-(S^3)$, then
        \[
        F^-_{W_1,\frt}(1) \not\in U^d \cdot \HF^-_{\red}(N) = \{0\}
        \]
        only for finitely many $\frt \in \Spin^c(W_1)$ by \cite[Theorem~3.3]{OSTriangles},
        and since $\HF^-_{\red}(N, \frs) \neq 0$ only for finitely many $\frs \in \Spin^c(N)$. The same argument works for $F_{W_2,\fru;\omegas|_{W_2}}^+$.
    \end{proof}

Recall from the introduction that, if $\ve{b}=(b_1,\dots, b_n)$ is a basis of $H^2(X;\R)$, we define
\[
\Phi_{X;\ve{b}}:= \sum_{\frs \in \Spin^c(X)} \Phi_{X,\frs} \cdot
	z_1^{\langle\, i_*(\frs-\frs_0) \cup b_1, [X] \,\rangle} \cdots
	z_n^{\langle\, i_*(\frs-\frs_0) \cup b_n, [X] \,\rangle},
\]
where $\frs_0\in \Spin^c(X)$ is a choice of base $\Spin^c$ structure.   If $H^2(X)$ is torsion-free,
then $\Phi_{X; \ve{b}}$ completely encodes the map $\frs \mapsto \Phi_{X,\frs}$.
We now give a slight reformulation of $\Phi_{X;\ve{b}}$,
which is well suited for proving Theorem~\ref{thm:1}:

\begin{prop}\label{prop:perturbed-mixed-invariant}
  Suppose $X$ is a closed, oriented 4-manifold with $b_2^+(X)>1$,
  and $N$ is an admissible cut, dividing $X$ into cobordisms
  $W_1$ and $W_2$. Suppose  $\ve{b}=(b_1,\dots, b_n)$ is an
  $n$-tuple of classes in $H^2(X;\R)$, represented
  by 2-forms $\omegas=(\omega_1,\dots, \omega_n)$ that
  vanish on $N$. Write $\omegas_1 = \omegas|_{W_1}$ and $\omegas_2 = \omegas|_{W_2}$.
  Then the maps $F_{W_2;\omegas_2}^+$ and $F_{W_1;\omegas_1}^-$ are well-defined,
  and satisfy
  \begin{equation}\label{eq:perturbed-mixed-invariant}
  \Phi_{X;\ve{b}}\doteq \left \langle \left(F_{W_2;\omegas_2}^+ \circ
  \delta^{-1} \circ F_{W_1;\omegas_1}^-\right)(1), \Theta_+ \right \rangle.
  \end{equation}
  \end{prop}

\begin{proof}
Well-definedness of $F_{W_1;\omegas_1}^-$ and
$F_{W_2;\omegas_2}^+$ follows from Lemma~\ref{lem:finiteness},
so we focus on equation~\eqref{eq:perturbed-mixed-invariant}.

Let $\frs_0$ be a fixed element of $\Spin^c(X)$, and let $\frt_0 = \frs_0|_{W_1}$
and $\fru_0 = \frs_0|_{W_2}$.
Since $\omegas_1$ and $\omegas_2$ vanish on $N$, we apply a straightforward
adaptation of Lemma~\ref{lem:perturbed-coincide} from the single- to the multi-variable setting 
to obtain
\begin{equation}
\begin{split}
F_{W_1;\omegas_1}^-&\doteq \sum_{\frt \in \Spin^c(W_1)}
 z_1^{\langle i_*(\frt-\frt_0)\cup [\omega_1], [W_1, \d W_1] \rangle}
 \cdots z_n^{\langle i_*(\frt-\frt_0)\cup [\omega_n], [W_1, \d W_1] \rangle}
  \cdot F_{W_1,\frt}^-, \quad \text{and}\\
F_{W_2;\omegas_2}^+& \doteq
\sum_{\fru \in \Spin^c(W_2)} z_1^{\langle i_*(\fru-\frs_0)\cup [\omega_1], [W_2, \d W_2] \rangle}
\cdots z_n^{\langle i_*(\fru-\fru_0)\cup [\omega_n], [W_2, \d W_2] \rangle} \cdot F_{W_2,\fru}^+.
\end{split}
\label{eq:expand-perturbed-plus-minus-maps}
\end{equation}
Equation~\eqref{eq:perturbed-mixed-invariant} is obtained by
inserting equation~\eqref{eq:expand-perturbed-plus-minus-maps}
into the right-hand side of equation~\eqref{eq:perturbed-mixed-invariant},
and using the fact that, if $\frs\in \Spin^c(X)$ restricts to
 $\frt\in \Spin^c(W_1)$ and $\fru\in \Spin^c(W_2)$, then
\[
\langle i_*(\frt-\frt_0)\cup [\omega_i], [W_1,\d W_1] \rangle
+\langle i_*(\fru-\fru_0)\cup [\omega_i], [W_2,\d W_2] \rangle
=\langle i_*(\frs-\frs_0)\cup [\omega_i], [X] \rangle. \qedhere
\]
\end{proof}

\begin{rem} In light of Proposition~\ref{prop:perturbed-mixed-invariant},
it is natural to view $\Phi_{X;\ve{b}}$ as a perturbed version of the mixed invariant.
\end{rem}

	\section{Fintushel--Stern knot surgery and concordance surgery}
	\label{sec:proof-of-concordance-surgery-formula}
	
	Fintushel and Stern~\cite{FSKnotSurgery} described an operation
	on a 4-manifold $X$ called \emph{knot surgery}. Given a knot $K$ in $S^3$ and an embedded torus $T$ in $X$
	with
	zero self-intersection, we define the 4-manifold
	\[
	X_0 := X \setminus N(T)
	\]
	with boundary $\bT^3$. A neighborhood of $T$ can be
	identified with $T \times D^2$. We pick any orientation-preserving
	diffeomorphism $\phi \colon \d(T \times D^2) \to S^1 \times \d N(K)$ such that
	$\phi_*([ \{p\} \times \d D^2]) = [\{q\} \times \ell_K]$, where $\ell_K$ is a
	Seifert longitude on $\d N(K)$, while $p \in T$
	and $q \in S^1$. We let
	\[
	X_K := X_0 \cup_{\phi} (S^1 \times (S^3 \setminus N(K)))
	\]
	be the result of knot surgery on $X$ using $K$ and $T$.
	Note that there is some ambiguity in the choice of $\phi$, so we write $X_K$
	for any 4-manifold constructed in this way. It is straightforward to see that
	$H^*(X_K)$ and $H^*(X)$ are canonically isomorphic.
	
	Fintushel and Stern described a generalization of this operation called
	\emph{concordance surgery}; see Akbulut~\cite{Akbulut}.
	Let $K$ be a knot in a homology 3-sphere $Y$ (note that Akbulut only considered
	$Y = S^3$).
	Given a self-concordance $\cC = (I \times Y, A)$ from $(Y, K)$ to itself, we can
	construct a 4-manifold $X_{\cC}$, as follows. We take the annulus $A$, and glue its ends
	together to form a 2-torus $T_{\cC}$ embedded in $S^1 \times Y$.
    The quotient map $I \times Y \to S^1 \times Y$ is given by $(t,y) \mapsto (e^{2\pi i t}, y)$
    for $t \in I$ and $Y \in Y$.
    After removing a neighborhood of $T_{\cC}$, we get a 4-manifold $W_{\cC}$ with
	boundary $\bT^3$. We pick any orientation-preserving diffeomorphism
	$\phi \colon \d X_0 \to \d N(T_{\cC})$ that sends
	$[\{p\}\times \d D^2 ]$ to $[\{1\} \times \ell_K]$.
	We write $X_{\cC}$ for any manifold constructed as the union
	\[
	X_{\cC} := X_0 \cup_{\phi} W_{\cC}.
	\]
	It is easy to see that $H^*(X_{\cC})$ and $H^*(X)$ are canonically isomorphic.
	
	If $\cC = (I \times Y, A)$ is a self-concordance of the knot $K$ in $Y$,
	and $a$ is a pair of parallel arcs on $A$ connecting the two components of $\d
	A$, then
	there is an induced map on knot Floer homology
	\[
	\hat{F}_{\cC,a} \colon \HFKh(Y, K) \to \HFKh(Y, K),
	\]
	described by the first author~\cite{JCob}.
	The map $\hat{F}_{\cC,a}$ preserves the Alexander and Maslov gradings
	according to Marengon and the first
	author~\cite{JMComputeCobordismMaps}*{Theorem~5.18},
	and is non-vanishing when $Y = S^3$ by~\cite{JMConcordance}*{Theorem~1.2}.
	
	Note that the group $\HFKh(Y, K)$ only becomes natural once we
	choose a pair $P$ of
	basepoints on $K$, which we suppress from the notation. We require $\d a$ to be
	disjoint from~$P$, and also to link $\d a$.
	We define $\Lef_z(\cC)$ to be the polynomial
	\[
	\Lef_z(\cC) := \sum_{i \in \Z} \Lef \left(\hat{F}_{\cC,a}|_{\HFKh(Y, K,i)}
	\colon
	\HFKh(Y, K, i) \to \HFKh(Y, K, i) \right) \cdot z^i
	\]
	for any pair of parallel arcs $a$ connecting
	the two boundary components of $\cC$.
	Although the map $\hat{F}_{\cC,a}$ depends on the arcs $a$, we have the
	following:
	
	\begin{lem}
		The graded Lefschetz number of $\hat{F}_{\cC,a}$
		is independent of the choice of arcs $a$.
	\end{lem}
	
	\begin{proof}
		Changing the arcs $a$ by a proper isotopy that does not cross the basepoints
		$P$ does not change the cobordism map $\hat{F}_{\cC,a}$. Hence, it suffices to
		show that the Lefschetz number is unchanged by applying a Dehn twist to $a$
		along one of
		the boundary components of the annulus $A$. The action of a Dehn twist
		on $\hat{\HFK}(Y, K)$ was computed by Sarkar~\cite{SarkarMovingBasepoints}
		when $Y = S^3$,
		and by the second author~\cite{ZemQuasi}*{Theorem~B} for a null-homologous
		knot in a general
		3-manifold~$Y$. If $r_*$ denotes the action of a single Dehn twist, then
		\[
		r_*=\id+\Phi\Psi,
		\]
		where $\Phi$ and $\Psi$ are two endomorphisms of $\HFKh(Y, K)$ that satisfy
		\[
		\Phi^2 = \Psi^2 = 0 \text{, } \Phi \Psi = \Psi\Phi.
		\]
		
		Since a Dehn twist on an annulus may be pulled to either boundary component,
		it follows that, if $a'$ differs from $a$ by
		a single Dehn twist along one end of the annulus, then
		\[
		\hat{F}_{\cC,a'} = \hat{F}_{\cC,a} \circ (\id + \Phi \Psi)=(\id+\Phi\Psi)\circ
		\hat{F}_{\cC,a}.
		\]
		Consequently, the map
		$\hat{F}_{\cC,a} \circ (\Phi\Psi)$ is nilpotent, so has Lefschetz number 0 in
		each Alexander grading.
	\end{proof}
	
\begin{lem}\label{lem:symmetry}
  The graded Lefschetz number $\Lef_z(\cC)$ is symmetric with respect to
  the conjugation $z \mapsto z^{-1}$.
\end{lem}
	
\begin{proof}
The proof follows easily from the conjugation symmetry of the knot Floer
homology groups  \cite{OSKnots}*{Proposition~3.10}, as well as the
corresponding symmetry of the knot cobordism maps
\cite{ZemkeConnectedSums}*{Theorem~1.3}.
\end{proof}
	
	If $X$ is a closed, oriented 4-manifold with a smoothly embedded 2-torus $T$
	such that $[T] \neq 0 \in H_2(X;\R)$, then we can pick a basis
    $\ve{b} = (b_1,\dots, b_n)$ of $H^2(X;\R)$ such that
	\begin{equation}
	\langle\, b_1, [T] \,\rangle = 1 \text{ and } \langle\, b_i, [T] \,\rangle = 0 \text{ for } i > 1.
	\label{eq:basisof2-forms}
	\end{equation}
	This induces a basis of $H^2(X_\cC;\R)$ that we also denote by $\ve{b}$.
	We restate our main theorem.
	
	\begin{customthm}{\ref{thm:1}}\label{thm:akbulutconcordanceformula}
		Let $X$ be a closed, oriented 4-manifold such that $b_2^+(X) \ge 2$.
		Suppose that $T$ is a smoothly embedded 2-torus in $X$ with trivial
		self-intersection,
		such that $[T] \neq 0 \in H_2(X;\R)$. Furthermore, let $\ve{b} =
		(b_1,\dots, b_n)$ be a basis of $H^2(X; \R)$ satisfying equation~\eqref{eq:basisof2-forms}.
		If $\cC$ is a self-concordance of $(Y,K)$, where $Y$ is a homology 3-sphere,
		then
		\[
		\Phi_{X_{\cC}; \ve{b}} = \Lef_{z_1}(\cC) \cdot \Phi_{X; \ve{b}}.
		\]
	\end{customthm}
	
	In order to prove Theorem~\ref{thm:akbulutconcordanceformula}, we need to
	perform several computations.
	Let $\cC$ be a self-concordance of a knot~$K$ in the homology 3-sphere~$Y$.
	On the torus $T_{\cC} \subset S^1 \times Y$, we
	pick a pair of dividing curves, each intersecting
	$\{1\} \times K$ exactly once. Such dividing curves
	are determined up to Dehn twists about $\{1\}\times K$.
	The dividing set specifies an
	isotopically unique, positive, $S^1$-invariant
	contact structure~$\xi_{\cC}$ on~$\bT^3 = -\d N(T_{\cC})$,
	by the work of Lutz \cite{LutzS1InvariantContactStructures}.
	Note that this contact structure is positive with respect
	to the boundary orientation from $W_{\cC}$.
	
	\begin{prop}\label{prop:computehatmapsusingsuturedmaps}
		Let $\omega_{\cC}$ be a closed 2-form on the 4-manifold $W_{\cC}$,
		Poincar\'e dual to $\{1\} \times \cS_K$, where $\cS_K$ is a Seifert surface
		for the knot $K$.
		If we view $W_{\cC}$ as a cobordism from $-\bT^3$ to $\emptyset$,
        and write $\tau_\cC = \omega_\cC|_{\partial W_\cC}$, then
		\[
		\hat{F}_{W_{\cC}; \omega_{\cC}}\left(\hat{c}(\xi_{\cC}; \tau_{\cC})\right) \doteq
		\Lef_z(\cC),
		\]
		as an element of $\HFh(\emptyset; \Lambda) \cong \Lambda$, where
        $\hat{c}\left(\xi_{\cC}; \tau_{\cC}\right) \in \HFh(-\bT^3; \Lambda_{\tau_{\cC}})$
        is the contact class of $\xi_{\cC}$ twisted by $\tau_{\cC}$.
	\end{prop}
	
	\begin{proof}
		We consider the sutured manifold cobordism $\cW_{\cC} := (W_{\cC}, \bT^3, [\xi_{\cC}])$
		from the empty sutured manifold to itself. In Section~\ref{sec:cob-construct}, we define the sutured cobordism
		map as the composition of the contact gluing map for gluing
		$(\bT^3,\xi_{\cC})$ to the empty sutured manifold and perturbed by $\tau_{\cC}$,
        followed by 4-dimensional 1-, 2-, and 3-handle maps.
        The composition of the handle maps is the perturbed cobordism map $\hat{F}_{W_{\cC}; \omega_{\cC}}$
        induced by the cobordism $W_{\cC}$ from $\bT^3$ to $\emptyset$, as defined by Ozsv\'ath and Szab\'o~\cite{OSTriangles}. 
        Since $\bT^3$ is a closed 3-manifold, 
        the gluing map sends the generator of $\SFH(\emptyset; \Lambda) \cong \Lambda$ to
        the perturbed contact element $\hat{c}(\xi_{\cC}; \tau_{\cC})$.
        Consequently, the perturbed sutured cobordism map $F_{\cW_{\cC}; \omega_{\cC}}$ satisfies
		\[
		F_{\cW_{\cC}; \omega_{\cC}}(1) \doteq \hat{F}_{W_{\cC}; \omega_{\cC}} 
        \left(\hat{c}(\xi_{\cC}; \tau_{\cC})\right).
		\]

		Let us write $Y(K)$ for the sutured manifold obtained by adding
		two meridional sutures to $Y \setminus N(K)$.
		We decompose $\cW_{\cC}$ as
        \[
        \Cap_{Y(K)} \circ \textsf{Id}_{Y(K) \sqcup -Y(K)} \circ (\cW(\cC, a) \sqcup \textsf{Id}_{-Y(K)}) \circ \Cup_{Y(K)},
        \]
        where
		\begin{itemize}
			\item $\Cup_{Y(K)}$ is the cotrace cobordism from $\emptyset$ to $Y(K) \sqcup -Y(K)$,
			\item $\cW(\cC, a)$ is the sutured manifold cobordism from $Y(K)$ to itself
			complementary to the decorated concordance $(\cC, a)$, and $\textsf{Id}_{-Y(K)}$ is the identity
            cobordism of $-Y(K)$,
			\item $\textsf{Id}_{Y(K) \sqcup -Y(K)}$ is the identity cobordism of $Y(K) \sqcup -Y(K)$, and
			\item $\Cap_{Y(K)}$ is the trace cobordism from $Y(K) \sqcup -Y(K)$ to $\emptyset$.
		\end{itemize}
        Since $\cW_{\cC}$ is a sutured cobordism from $\emptyset$ to $\emptyset$, it
        follows from  Lemma~\ref{lem-modify-2-form-on-interior} that replacing
        $\omega_{\cC}$ with $\omega_{\cC} + d\eta$ for a 1-form $\eta$  only changes
        $F_{\cW_{\cC};\omega_{\cC}}(1)$ by an overall factor of $z^x$. Hence, we may
        assume that	the 2-form $\omega_{\cC}$ restricts trivially to $\Cup_{Y(K)}$,
        $\cW(\cC, a) \sqcup \textsf{Id}_{-Y(K)}$, and
		$\Cap_{Y(K)}$. Its restriction $\omega'$ to $\textsf{Id}_{Y(K) \sqcup -Y(K)}$
		is Poincar\'{e}-Lefschetz dual to $\{\tfrac{1}{2}\}\times \cS_K$,
		for a Seifert surface $\cS_K \subset Y(K)$.
		
		By Lemma~\ref{lem:Alexander-grading}, and since
        $\textsf{Id}_{Y(K) \sqcup -Y(K)}$ is a disjoint union of two product cobordisms, we have
		\[
		F_{\textsf{Id}_{Y(K) \sqcup -Y(K)}; \omega'}(\xs \otimes \ys) = z^{-A(\ve{x})} \cdot (\x \otimes \y),
		\]
        up to an overall factor of $z^x$ for some $x \in \R$.
		By \cite{JuhaszZemkeContactHandles}*{Theorem~1.1}, we know that $\Cup_{Y(K)}$ and
		$\Cap_{Y(K)}$ induce the canonical cotrace and trace maps, respectively.
		It follows that
		\[
        (F_{\Cap_{Y(K)};0} \circ F_{\textsf{Id}_{Y(K) \sqcup -Y(K)}; \omega'}
        \circ  F_{\cW(\cC, a) \sqcup \textsf{Id}_{-Y(K)};0} \circ F_{\Cup_{Y(K)};0})(1)
        \]
        is the graded Lefschetz number $\Lef_{z^{-1}}(\hat{F}_{\cC,a})$.
        By Lemma~\ref{lem:symmetry}, this coincides with the graded Lefschetz number
        $\Lef_{z}(\hat{F}_{\cC,a})$, completing the proof.
	\end{proof}
	
	The special case of the unknot $U$ and the trivial concordance $(I \times S^3, I \times U)$ is important.
	In this case, the dividing set on the torus $S^1 \times U \subset S^1 \times
	S^3$ determines an $S^1$-invariant, positive contact
	structure $\xi_0$ on $\bT^3 = -\d N(S^1 \times U)$. Consider the 4-manifold
	\[
	W_0 = S^1 \times (S^3 \setminus N(U)) \iso S^1 \times S^1 \times D^2.
	\]
	
	\begin{cor}\label{cor:computehatmapsusingsuturedmaps}
		Let $\omega_0$ be a closed 2-form on the 4-manifold $W_0$,
		such that $[\omega_0]$ is Poincar\'e dual to $\{(1,1)\} \times D^2$.
		If we view $W_0$ as a cobordism from $-\bT^3$ to $\emptyset$, and write $\tau_0 = \omega_0|_{\partial W_0}$, then
		\[
		\hat{F}_{W_0; \omega_0}\left(\hat{c}(\xi_0; \tau_0)\right) \doteq 1,
		\]
		as an element of $\HFh(\emptyset; \Lambda) \cong \Lambda$.
	\end{cor}
	
	A choice of dividing sets on $S^1 \times U$ in $S^1\times S^3$ and $T_{\cC}$ in
	$S^1 \times Y$ induces a diffeomorphism between $S^1 \times U$ and $T_{\cC}$
	that maps $\{1\} \times U$ to $\{1\} \times K$, well-defined up to
	isotopy. We can extend this diffeomorphism to a $D^2$-bundle map from
	$(S^1 \times U) \times D^2$ to $T_{\cC} \times D^2$. We write $\bT^3$ for
	both $-\partial N(S^1 \times U)$ and $-\partial N(T_{\cC})$, identified via
	the restriction of such a diffeomorphism. Furthermore, the contact structures
	$\xi_{0}$ and $\xi_{\cC}$ are identified by this diffeomorphism, and hence we
	will write $\xi$ for both. Similarly, the 2-forms $\tau_0 = \omega_0|_{\bT^3}$
    and $\tau_\cC = \omega_{\cC}|_{\bT^3}$ are identified,
	so we write $\tau \in \Omega^2(\bT^3)$ for both.
	
Note that $\Spin^c(W_0)\iso \Spin^c(W_{\cC})\iso \Z$. We write $\frt_k\in
\Spin^c(W_0)$ for the $\Spin^c$ structure with
\[
 c_1(\frt_k)= 2k\cdot \PD [\{1\}\times \cS_U],
\]	
where $\cS_U$ is a Seifert surface for $U$ in $S^3\setminus N(U)$, and we are
using Poincar\'{e} duality
\[
 H_2(W_0, \d W_0) \iso H^2(W_0).
\]
Similarly, we write $\frt_k'\in \Spin^c(W_{\cC})$ for the $\Spin^c$ structure
satisfying $c_1(\frt_k')=2k\cdot \PD[\{1\}\times \cS_K]$, where $\cS_K$ is a
Seifert surface for $K$ in $Y\setminus N(K)$.
	
	\begin{cor}\label{cor:mapsonplus1}
		As maps from $\HF^+(-\bT^3;\Lambda_{\tau})$ to $\HF^+(\emptyset; \Lambda) \cong
		\Lambda$, we have
		\[
		F^+_{W_{\cC}, \frt_0'; \omega_{\cC}} \doteq
		\Lef_{z}(\cC) \cdot F^+_{W_0, \frt_0; \omega_0}.
		\]
        Furthermore,
		$F^+_{W_0, \frt_k; \omega_0}$ and $F^+_{W_{\cC}, \frt_k'; \omega_{\cC}}$
		vanish for every $k \in \Z \setminus \{0\}$.
	\end{cor}
	
	\begin{proof}
		The contact element
		\[
		c^+(\xi; \tau) \in \HF^+(-\bT^3; \Lambda_{\tau})
		\]
		was defined by Ozsv\'ath and Szab\'o~\cite{OSContactStructures}
		as the image of $\hat{c}(\xi; \tau)$ under the natural map
		\[
		\iota_* \colon \HFh(-\bT^3;\Lambda_\tau) \to \HF^+(-\bT^3;\Lambda_{\tau}).
		\]
		Since $\iota_*$ commutes with the perturbed cobordism maps for $W_0$ and
		$W_{\cC}$ on $\HFh$ and $\HF^+$,
		we have
		\[
		F_{W_{\cC}; \omega_{\cC}}^+\left(c^+(\xi; \tau)\right) \dot{=} \Lef_z(\cC)
		\]
		by Proposition~\ref{prop:computehatmapsusingsuturedmaps}, and
		\[
		F_{W_0; \omega_0}^+\left(c^+(\xi; \tau)\right) \dot{=} 1
		\]
		by Corollary~\ref{cor:computehatmapsusingsuturedmaps}. Hence
		$c^+(\xi; \tau) \neq 0$, and
		\begin{equation}
		F_{W_{\cC}; \omega_{\cC}}^+\left(c^+(\xi, \tau)\right) \doteq
		\Lef_z(\cC)\cdot F_{W_0; \omega_0}^+\left(c^+(\xi, \tau)\right).
		\label{eq:equalityonc+(xi)}
		\end{equation}
		
		Next, we use the well-known fact that, if $\tau$ is any non-vanishing, closed
		2-form on $-\bT^3$, then
		\[
		\HF^+(-\bT^3;\Lambda_{\tau})\iso \Lambda,
		\]
		and $\HF^+(-\bT^3;\Lambda_{\tau})$ is supported in the torsion
		$\Spin^c$ structure on $-\bT^3$; see Ai and
		Peters~\cite{AiPetersTwisted}*{Theorem~1.3}, Jabuka and Mark~\cite[Theorem~10.1]{Jabuka-Mark},
		Lekili~\cite{LekiliBrokenFibrations}*{Theorem~14}, and
		Wu~\cite{WuPerturbedFibered}.
		It follows that $F_{W_{\cC}; \omega_{\cC}}^+$ and $F_{W_{0}; \omega_0}^+$,
		whose domains are thus rank~1 over $\Lambda$, must be constant multiples of each other.
		Equation~\eqref{eq:equalityonc+(xi)} and the fact that $c^+(\xi; \tau) \neq 0$
		now establish that the ratio is $\Lef_z(\cC)$, up to an overall factor of $z^x$.		
		
		Finally, the maps in the $\Spin^c$
		structures $\frt_k$ and $\frt_k'$ for $k \in \Z \setminus \{0\}$
		vanish because they have trivial domain. In particular,
		\[
		F_{W_{\cC}; \omega_{\cC}}^+ = F_{W_{\cC},\frt_0'; \omega_{\cC}}^+
		\text{ and } F_{W_0; \omega_0}^+ = F_{W_0,\frt_0; \omega_0}^+,
		\]
		completing the proof.
	\end{proof}
	
	\begin{cor}\label{cor:mapsonplus2}
		If $\omegas = (\omega_1, \dots, \omega_n)$ is a collection of closed 2-forms on
		$X$ satisfying
        \[
        \int_T \omega_1 = 1 \text{ and } \int_T \omega_i = 0 \text{ for } i > 1,
        \]
		and $\omegas' = (\omega_1', \dots, \omega_n')$ is the induced collection on
		$X_{\cC}$ under the canonical isomorphism $H^2(X_{\cC};\R) \iso H^2(X;\R)$, then
		\[
		F_{W_{\cC},\frt_0'; \omegas'|_{W_{\cC}}}^+ \dot{=} \Lef_{z_1}(\cC)
		\cdot F^+_{W_0,\frt_0; \omegas|_{W_0}},
		\]
		and both maps vanish for all other $\Spin^c$ structures.
	\end{cor}
	
	\begin{proof}
		Let the 1-variable Novikov ring $\Lambda$ act on the $n$-variable Novikov ring $\Lambda_n$
		in the variables $z_1,\dots,z_n$ via multiplication by the first variable.
        We write $\Lambda_{\omegas}$ for $\Lambda_n$ viewed as a module over $\bF_2[H_2(M)]$ via 
        the formula $e^a \cdot z_i^x = z_i^{x + \int_a \omega_i}$ for $a \in H_2(M)$ and $x \in \R$.
		Since the classes $[\omega_2], \dots, [\omega_n]$ vanish on $W_0$ and $[\omega_2'], \dots,
		[\omega_n']$ vanish on $W_{\cC}$, arguing as in the proof of
		Lemma~\ref{lem:cut}, we may assume the 2-forms $\omega_2,\dots, \omega_n$ and
		$\omega_2',\dots, \omega_n'$ have been chosen to vanish on $W_0$ and $W_{\cC}$.
		Hence, we obtain a canonical isomorphism
		\[
		\HF^+(-\bT^3;\Lambda_{\omegas|_{-\bT^3}})\iso \HF^+(-\bT^3;\Lambda_\tau)\otimes_{\Lambda} \Lambda_n.
		\]
		Immediately from the definitions, we obtain that, with respect to this decomposition,
		\[
		F_{W_0, \frt_k;
		\omegas|_{W_0}}^+=F_{W_0, \frt_k;
		\omega_1|_{W_0}}^+\otimes \id_{\Lambda_n},
		\]
		and similarly for $F_{W_{\cC}, \frt_k'; \omegas'|_{W_{\cC}}}^+$.
		The main  result now follows from Corollary~\ref{cor:mapsonplus1}.
		\end{proof}
			
	We can now prove Theorem~\ref{thm:akbulutconcordanceformula}.
	
	\begin{proof}[Proof of Theorem~\ref{thm:akbulutconcordanceformula}]
		As before, let $X_0 = X \setminus N(T)$.
		Since $b_2^+(X) \ge 2$, by analyzing the Mayer--Vietoris sequence for $X = X_0
		\cup N(T)$,
		it is easy to see that $b_2^+(X_0) \ge 1$. Hence, there is a surface~$Q$
		of positive self-intersection in the complement of~$T$.
		Let $N$ denote the boundary of a tubular neighborhood of~$Q$.
		The manifold $N$ is an admissible cut of $X$ by \cite[Example~8.4]{OSTriangles}.
        We write $W_1 = N(Q)$ and $W_2 = X \setminus \Int(N(Q))$.

        By Lemma~\ref{lem:cut}, there are 2-forms $\omegas = (\omega_1, \dots, \omega_n)$
        such that $[\omega_i] = b_i$ and $\omega_i|_N = 0$. Furthermore, we can arrange
        that $\omega_1|_{N(T)} = \omega_0$ and $\omega_i|_{N(T)} = 0$ for $i > 1$.
        We let $\omegas' = (\omega_1', \dots, \omega_n')$ be an $n$-tuple of forms on $X_{\cC}$
        such that $\omega_1'|_{N(T)} = \omega_{\cC}$ and $\omega_i'|_{N(T)} = 0$ for $i > 1$,
        while $\omega_i'|_{X_0} = \omega_i|_{X_0}$ for $i \in \{1,\dots, n\}$.	

        By Proposition~\ref{prop:perturbed-mixed-invariant},
        \[
        \Phi_{X; \ve{b}} =
        \left\langle\, \left(F_{W_2; \omegas|_{W_2}}^+ \circ \delta^{-1}
        \circ F_{W_1; \omegas|_{W_1}}^-\right)(1), \Theta_+ \,\right\rangle.
        \]
        We now apply the composition law, Proposition~\ref{prop:composition},
        to the splitting $W_2 = W_0 \cup_{\T^3} W'$, where $W_0 = N(T)$ and
        $W' = W_2 \setminus \Int(N(T))$, to obtain that
        \[
        F_{W_2; \omegas|_{W_2}}^+ \doteq F_{W_0; \omegas|_{W_0}}^+ \circ F_{W'; \omegas|_{W'}}^+.
        \]
        Similarly, if $W_2' := W_\cC \cup_{\T^3} W'$, then we have
        \[
        \Phi_{X_\cC; \ve{b}} =
        \left\langle\, \left(F_{W_2'; \omegas'|_{W_2'}}^+ \circ \delta^{-1}
        \circ F_{W_1; \omegas'|_{W_1}}^-\right)(1), \Theta_+ \,\right\rangle,
        \]
        where
        \[
        F_{W_2'; \omegas'|_{W_2'}}^+ \doteq F_{W_\cC; \omegas'|_{W_\cC}}^+ \circ F_{W'; \omegas'|_{W'}}^+.
        \]
        By construction of $\omegas'$, we have $\omegas'|_{W_1} = \omegas|_{W_1}$ and
        $\omegas'|_{W'} = \omegas|_{W'}$. Hence, it follows from Corollary~\ref{cor:mapsonplus2} that
        \begin{equation}
        \Phi_{X_\cC; \ve{b}} \doteq \Lef_{z_1}(\cC) \cdot \Phi_{X; \ve{b}}. \label{eq:almost-concordance-formula}
        \end{equation}
        Equality in equation~\eqref{eq:almost-concordance-formula}
         can be established using the conjugation symmetry of the Ozsv\'{a}th--Szab\'{o}
          4-manifolds invariants \cite{OSTriangles}*{Theorem~3.6}.
        \end{proof}

\subsection{Concordance surgery and diffeomorphism types of 4-manifolds}
\label{sec:diffeomorphism-types}

As an application of Theorem~\ref{thm:1},  we prove
Corollary~\ref{cor:not-diffeomorphic}, which states that $X$ and $X_{\cC}$
are not diffeomorphic if $\Phi_{X;\ve{b}}\neq 0$ and $\Lef_{z}(\cC)\neq 1$:

	\begin{proof}[Proof of Corollary~\ref{cor:not-diffeomorphic}]
    Choose a basis $\ve{b} = (b_1,\dots, b_n)$ of $H_2(X;\R)$ that is induced by a basis
	of $H^2(X)/\Tors$. In this situation, the invariant $\Phi_{X;\ve{b}}$
	takes values in the integral group ring $\bF[\Z^n]$. It is convenient to use the group ring notation
	\[
	e^{(a_1,\dots, a_n)}:=z_1^{a_1}\cdots z_n^{a_n},	
	\]
	where $(a_1, \dots, a_n) \in \Z^n$. If $\ve{b}=(b_1,\dots, b_n)$
    is an $n$-tuple of cohomology classes, we abbreviate
	\[
	\langle i_*(\frs-\frs_0)\cup \ve{b},[X]\rangle :=
    (\langle i_*(\frs-\frs_0)\cup b_1,[X]\rangle, \dots, \langle i_*(\frs-\frs_0) \cup b_n,[X]\rangle).
	\]
	
    Performing a change of basis to Theorem~\ref{thm:1}, we obtain
	\begin{equation}\label{eq:slightly-generalized-formula}
    \Phi_{X_{\cC};\ve{b}} = \Lef_{e^{\langle \ve{b},[T] \rangle}}(\cC) \cdot \Phi_{X;\ve{b}}.	
	\end{equation}

On the other hand, if $\phi\colon X_{\cC}\to X$ were an orientation preserving diffeomorphism, then
\begin{equation}
\Phi_{X,\frs} = \Phi_{X_{\cC},\phi^*(\frs)}
\label{eq:diffeo-invariance}
\end{equation}
 for all $\frs$. Hence
\begin{equation}
    \begin{split}
    \Phi_{X;\ve{b}}&:=\sum_{\frs\in \Spin^c(X)}  e^{\langle i_*(\frs-\frs_0)\cup \ve{b},[X]\rangle} \cdot \Phi_{X,\frs}\\
    &=\sum_{\frs\in \Spin^c(X)} e^{\langle \phi^*i_*(\frs-\frs_0)\cup \phi^*(\ve{b}),[X_{\cC}]\rangle} \cdot \Phi_{X_{\cC},\phi^*(\frs)}\\
    &=\sum_{\frs\in \Spin^c(X_{\cC})} e^{\langle i_*(\frs-\phi^*(\frs_0))\cup \phi^*(\ve{b}),[X_{\cC}]\rangle} \cdot \Phi_{X_{\cC},\frs}\\
    &\doteq \sum_{\frs\in \Spin^c(X_{\cC})} e^{\langle i_*(\frs-\frs_0)\cup \phi^*(\ve{b}),[X_{\cC}]\rangle} \cdot \Phi_{X_{\cC},\frs}\\
    &= e^{M(\phi^*)^t} \cdot \sum_{\frs\in \Spin^c(X_{\cC})} e^{\langle i_*(\frs-\frs_0)\cup \ve{b},[X_{\cC}]\rangle} \cdot \Phi_{X_{\cC},\frs}\\
    &= e^{M(\phi^*)^t} \cdot \Phi_{X_{\cC};\ve{b}}.
    \end{split}
    \label{eq:effect-of-diffeo}
\end{equation}
Here, $M(\phi^*)$ denotes the element of $\GL_n(\Z)$ induced by $\phi^*$ after
identifying $H^2(X)/\Tors$ and $H^2(X_{\cC})/\Tors$ with $\Z^n$
via the basis $\ve{b}$, and $M(\phi^*)^t$ denotes its transpose.
Also, we are writing $e^{M(\phi^*)^t}$ for the endomorphism of
$\bF[\Z^n]$ given by $e^{M(\phi^*)^t} e^{\ve{a}}=e^{M(\phi^*)^t\cdot \ve{a}}$,
where we view $\ve{a}$ as a column vector.

Equation~\eqref{eq:effect-of-diffeo} is justified as follows.
The first equality is a definition.
The second equality follows from equation~\eqref{eq:diffeo-invariance}, and the naturality of cohomology.
The third equality follows from rearranging the sum.
The fourth equality follows since $\Phi_{X;\ve{b}}$ is
invariant, up to overall multiplication by a monomial,
of the choice of base $\Spin^c$ structure $\frs_0$.
The fifth equality can be computed directly, and the final equality again holds by definition.

The ring $\bF[\Z^n]$ is a UFD, since it is the localization of the polynomial
ring $\bF[z_1,\dots, z_n]$ at monomials. Furthermore, the units are
exactly the monomials. The map $e^{M(\phi^*)^t}$ preserves the number of
irreducible factors since $e^{M(\phi^*)^t}(f\cdot g)=(e^{M(\phi^*)^t} \cdot
f)(e^{M(\phi^*)^t}\cdot g)$, the map $e^{M(\phi^*)^t}$ sends monomials to
monomials, and $e^{M(\phi^*)^t}$ is invertible.
	
In particular, if $\Lef_{z}(\cC)\neq 1$ and $\Phi_{X;\ve{b}}\neq 0$,
equation~\eqref{eq:slightly-generalized-formula} implies that
$\Phi_{X_{\cC};\ve{b}}$ has more irreducible factors than $\Phi_{X;\ve{b}}$,
while equation~\eqref{eq:effect-of-diffeo} implies they have the same number,
a contradiction.
\end{proof}

\section{Naturality of perturbed sutured Floer homology}
\label{sec:naturality}
		
This section is devoted to defining transition maps on perturbed sutured Floer
homology for naturality, and proving Theorem~\ref{thm:naturality-perturbed}.

\subsection{Changing the 2-form}
\label{sec:change-2-form}
	
We first describe transition maps for changing the 2-form by a boundary. Unlike
the transition maps for changing the Heegaard diagrams,
we usually do not want to view sutured Floer homology as a transitive system
over closed 2-forms which represent the same cohomology class. Nonetheless, the
transition maps for changing the 2-form are convenient to define.

Let $\cH$ be an admissible diagram of the balanced sutured manifold $(M,\gamma)$,
and let $\omega$ and $\omega'$ be closed cohomologous 2-forms on $M$.
Suppose $\eta$ is a 1-form such that $d\eta=\omega'-\omega$. Then we may
	define a chain isomorphism
	\[
	\Psi_{\omega\to \omega';\eta}\colon \CF_J(\cH;\Lambda_{\omega})\to
	\CF_J(\cH;\Lambda_{\omega'})
	\]
	via the formula
	\[
	\Psi_{\omega\to \omega';\eta}(z^{x} \cdot \xs)=z^{x+\int_{\gamma_{\xs}} \eta} \cdot \xs,
	\]
    where we obtain $\gamma_{\xs}$ by connecting $\xs$ to the centers of the
    disks $D_{\sca}$
    and $D_{\scb}$ along radii. We orient $\gamma_{\xs}$ from $D_{\sca}$ to $D_{\scb}$.
	The map $\Psi_{\omega\to \omega';\eta}$ is a chain map by Stokes' theorem,
	and is an isomorphism since $\Psi_{\omega' \to \omega; -\eta}$ is its inverse.
	
	\begin{lem}\label{lem:independence-rep-omega-spinc}
		When restricted to a single $\Spin^c$ structure, the map
		$\Psi_{\omega\to\omega';\eta}$
		is independent of the 1-form $\eta$ satisfying $d \eta=\omega'-\omega$, up to
		an overall factor of $z^x$.
	\end{lem}

	\begin{proof}
        It is sufficient to show that, if $\eta$ is a closed 1-form,
		then $\Psi_{\omega\to \omega;\eta}$ is equal to
		overall multiplication by $z^x$
		for some $x\in \R$,  when restricted to a single $\Spin^c$ structure.
		Hence, it is
		sufficient to show that, if $\underline{\frs}(\xs)=\underline{\frs}(\ys)$ and
		$d\eta=0$, then
		\[
		\int_{\gamma_{\xs}} \eta=\int_{\gamma_{\ys}} \eta.
		\]
		The condition that $\underline{\frs}(\xs)=\underline{\frs}(\ys)$ is equivalent
		to the
		condition that the integral 1-cycle $\gamma_{\xs} -\gamma_{\ys}$ is $\d S$,
		for some integral 2-chain $S$. By Stokes' theorem,
		\[
		\int_{\gamma_{\xs} -\gamma_{\ys}} \eta=\int_{S} d \eta=0,
		\]
		completing the proof.
	\end{proof}
	
	In general, the map $\Psi_{\omega\to \omega';\eta}$
	is not independent of $\eta$
	when working with multiple $\Spin^c$ structures at once,
	even if $[\omega]=[\omega']=0$; see
	Remark~\ref{rem:failure-transitive-system-over-2-forms}.

	\subsection{Change of almost complex structure maps}
	\label{sec:change-of-almost-complex-structure}	
	
	Suppose $\cH$ is an admissible diagram of $(M,\g)$. If $J$ and $J'$ are
	two
	cylindrical almost complex structures on $\Sigma\times I\times \R$,
	there is a standard Floer theoretic construction that gives a transition map
	from $\CF_{J}(\cH;\Lambda_{\omega})$ to $\CF_{J'}(\cH;\Lambda_{\omega})$;
    see Lipshitz~\cite[Section~9]{LipshitzCylindrical}.
	Pick a generic almost complex structure $\tilde{J}$ on
	$\Sigma\times I\times \R$ such that
	\[
	\tilde{J}=J\quad  \text{on}\quad \Sigma\times I\times (-\infty,a]
	\]
	and
	\[
	\tilde{J}=J'\quad \text{on} \quad \Sigma\times I\times [b,\infty),
	\]
	where $a\ll 0$ and $b\gg 0$. Define
	\[
	\Psi_{J\to J'}\colon \CF_J(\cH;\Lambda_{\omega})\to
	\CF_{J'}(\cH;\Lambda_{\omega})
	\]
	via the formula
	\[
	\Psi_{J\to J'}(z^x\cdot \xs)=\sum_{\ys\in \bT_{\sca}\cap \bT_{\scb}}
	\sum_{\substack{\phi\in \pi_2(\xs,\ys)\\
			\mu(\phi)=0}} \left( |\cM_{\tilde{J}}(\phi)| \mod 2\right) \cdot
	z^{x+A_{\omega}(\phi)}\cdot \ys.
	\]	

	\begin{lem}
		The map $\Psi_{J\to J'}$ is a chain map, and is independent of $\tilde{J}$, up
		to chain homotopy.
	\end{lem}

	\begin{proof} The claim that $\Psi_{J\to J'}$ is a chain map is proven by
		counting the ends of the moduli spaces of index 1, $\tilde{J}$-holomorphic
		curves. The claim that $\Psi_{J\to J'}$
		is independent of $\tilde{J}$ is proven
		by taking two generic choices $\tilde{J}_0$ and $\tilde{J}_1$, and connecting
		them via a path $(\tilde{J}_z)_{t\in I}$. A chain homotopy between the map
		which counts $\tilde{J}_0$-holomorphic curves and the map which counts
		$\tilde{J}_1$-holomorphic curves is given by counting index $-1$ curves that
		are $\tilde{J}_t$-holomorphic for some $t\in I$.
	\end{proof}	
	
	\subsection{Perturbed stabilization maps}
	\label{sec:stabilization}
	Suppose that $\cH=(\Sigma,\as,\bs)$ is an admissible diagram of $(M,\gamma)$, and
	$\cH'=(\Sigma',\as\cup \{\alpha'\}, \bs\cup \{\beta'\})$ is a stabilization of
	$\cH$; i.e., there is a 3-ball $B$ in $\Int (M)$ such that
	\begin{enumerate}
		\item $B\cap \Sigma$ is a disk and $B\cap \Sigma'$ is a punctured 2-torus
		that contains the curves $\alpha'$ and $\beta'$,
		and is disjoint from $\as\cup \bs$,
		\item $\Sigma \setminus B = \Sigma' \setminus B$, and
		\item $\alpha'$ and $\beta'$ intersect transversely at a single point $c$.
	\end{enumerate}
	
	The stabilization map
	\[
	\sigma \colon \CF(\cH)\to \CF(\cH')
	\]
	is given by $\sigma(\xs)=\xs\times c.$ According to \cite{OSDisks}*{Theorem~10.2},
    for a sufficiently stretched almost complex structure, the map $\sigma$ is a chain map.
    See Lipshitz~\cite[Section~12]{LipshitzCylindrical} for the corresponding result in the 
    cylindrical reformulation. We define the perturbed stabilization map
	\[
	\sigma\colon \CF(\cH;\Lambda_{\omega})\to \CF(\cH';\Lambda_{\omega})
	\]
	via the formula $\sigma(z^x\cdot \xs)=z^x \cdot (\xs\times c).$

	\begin{lem}\label{lem:perturbed-stabilization-chain-map}
		For a sufficiently stretched almost complex structure,
		the perturbed stabilization map $\sigma\colon \CF(\cH;\Lambda_{\omega})\to
		\CF(\cH';\Lambda_{\omega})$ is a chain map.
	\end{lem}

	\begin{proof}
	 If $\phi\in \pi_2(\xs,\ys)$
	is a class on $\cH$, there is a unique class $\phi'\in \pi_2(\xs\times
	c,\ys\times c)$ whose domain agrees with $\phi$ on $\Sigma\setminus B$. The
	class $\phi'$ has the same Maslov index as $\phi$. Ozsv\'{a}th and Szab\'{o}
	showed that, if $\mu(\phi)=1$, and if the almost complex structure on $\Sigma'$
	is sufficiently stretched, then
	\begin{equation}\label{eq:stabilization}
	|\cM(\phi)/\R| \equiv |\cM(\phi')/\R| \mod{2}.
	\end{equation}

		We  note that the 2-chain $\tilde{\cD}(\phi')$ only differs from
		$\tilde{\cD}(\phi)$ in the 3-ball $B$. Furthermore, there is an
		integral 3-chain $C_3$ (a sum of solid tori) such that
		\[
		\tilde{\cD}(\phi')+\d C_3=\tilde{\cD}(\phi).
		\]
		Hence $A_{\omega}(\phi')=A_{\omega}(\phi)$, so
equation~\eqref{eq:stabilization} implies that
		 $\sigma$ is a chain map on the perturbed complex.
	\end{proof}
	
	\subsection{Perturbed isotopy maps}
	\label{sec:perturbed-diffeos}
	
	Suppose that $(\phi_t)_{t \in I}$ is an isotopy of $M$, satisfying
	 $\phi_0=\id_M$. For convenience, let us assume that $\phi_t$ is constant
	for $t$ in a neighborhood of $\d I$. If
	$\cH = (\S,\as,\bs)$ is an admissible diagram for $(M,\gamma)$, write
	$\cH'$ for the diagram obtained by pushing forward $\Sigma$  along
	$\phi_1$. Let $J$ be a cylindrical almost complex structure on $\Sigma\times
	I \times \R$, and let $J'$ denote its pushforward along $\phi_1$.
    Given a choice of compressing disks $D_{\sca}$ and $D_{\scb}$ for $\cH$,
    we use $\phi_1(D_{\sca})$ and $\phi_1(D_{\scb})$ for $\cH'$.
	
	If $\xs \in \T_{\sca} \cap \T_{\scb}$ is an
	intersection point on $\cH$, let $\gamma_{\xs}$
	denote the 1-chain obtained by coning the
	points of $\xs$ into $U_{\a}$ and $U_{\b}$,
	and let $\Gamma_{\xs,\phi_t}$ denote the 2-chain in $M$ obtained by sweeping out
	$\gamma_{\xs}$ under $\phi_t$.
	We define
	\[
	(\phi_t)_*\colon \CF_J(\cH;\Lambda_\omega)\to \CF_{J'}(\cH';\Lambda_{\omega})
	\]
	via the formula
	\[
	z^x \cdot \xs\mapsto z^{x+\int_{\Gamma_{\xs,\phi_t}} \omega} \cdot \phi_1(\xs).
	\]
	Stokes' theorem can be used to show that $(\phi_t)_*$ is a chain map.
	We define the transition map for the isotopy $(\phi_t)_{t \in I}$ from $\cH$ to
	its image $\cH'$ to be $(\phi_t)_*$.
	
	\begin{rem}\label{rem:compressing-disks-transition-map}
    As a special case of the above construction,
	when $\phi_t$ fixes the Heegaard surface pointwise for all $t$,
	the map $(\phi_t)_*$ induces a map for
	transitioning between collections of compressing
	disks that are related by an ambient isotopy fixing $\Sigma$ pointwise.
	A similar construction gives a map for transitioning
	between collections of compressing disks that are instead only
	isotopic as maps from $D^2$ into $Y$, relative to $\d D^2$.
	The construction also adapts to give a transition map for changing the
	choice of radial foliation on the disks.
	\end{rem}

	The map $(\phi_t)_*$ depends only on $\phi_1$, in the following sense:
	
	\begin{lem}\label{lem:no-monodromy-spinc}
		Suppose that $(\phi_t)_{t\in I}$ and $(\psi_t)_{t\in I}$
		are two isotopies of $(M,\gamma)$, such that
		$\phi_0 = \psi_0 = \id_{(M,\gamma)}$, and $\phi_1 = \psi_1$. Then
		$(\phi_t)_* \doteq (\psi_t)_*$ on each $\Spin^c$
		structure. If $[\omega]=0$, then $(\phi_t)_* \doteq (\psi_t)_*$
		on all of $\CF_J(\cH;\Lambda_\omega)$.
	\end{lem}

    \begin{proof}
		Suppose that $\xs$ and $\ys$ are two intersection points that
		represent the same $\Spin^c$ structure. This is
		equivalent to the condition that $\gamma_{\xs}  - \gamma_{\ys} = \d S$ for
		some integral 2-chain $S$ in $M$. The isotopies $\phi_t$
		and $\psi_t$ applied to $S$ sweep
		out 3-chains $C_{\phi_t}$ and $C_{\psi_t}$. We have
		\begin{equation}
		\d C_{\phi_t} = \Gamma_{\xs,\phi_t} - \Gamma_{\ys,\phi_t} - S + \phi_1(S),
        \label{eq:boundary-S-phi-t}
		\end{equation}
		and a similar formula holds for $\d C_{\psi_t}$. Integrating $d \omega = 0$ on
		$C_{\phi_t}$ and $C_{\psi_t}$, and using equation~\eqref{eq:boundary-S-phi-t}
		and Stokes' theorem, we obtain
        \begin{equation}
		\int_{\Gamma_{\xs,\phi_t}}\omega
		-\int_{\Gamma_{\ys,\phi_t}}\omega=\int_{\Gamma_{\xs,\psi_t}}\omega-\int_{\Gamma_{\ys,\psi_t}}\omega.
		\label{eq:integrals-independent-path}
		\end{equation}
		Equation~\eqref{eq:integrals-independent-path} implies that
		$(\psi_t)_*$ and $(\phi_t)_*$ differ only by an overall factor of $z^x$,
        when restricted to a single $\Spin^c$ structure.

        Suppose now that $[\omega]=0$, and let $\xs$
        and $\ys$ be any two intersection points.
        Since $\gamma_{\xs} - \gamma_{\ys}$ is a 1-cycle,
		$(\Gamma_{\xs,\phi_t} - \Gamma_{\ys,\phi_t}) -
		(\Gamma_{\xs,\psi_t} -\Gamma_{\ys,\psi_t})$
		is a 2-cycle, so $\omega$ integrates to zero over it,
        and equation~\eqref{eq:integrals-independent-path} follows.		
	\end{proof}
		
	Let $\phi$ be an automorphism of $(M,\g)$.
    If $\cH = (\S,\as,\bs)$ is an admissible diagram of $(M,\g)$ with
    a cylindrical almost complex structure $J$ on $\S \times I \times \R$,
    and $\cH' = \phi(\cH)$ and $J' = \phi_*(J)$ are their pushforwards,
    then there is a tautological chain isomorphism
	\[
	\phi_*^{\mathrm{taut}} \colon \CF_J(\cH; \Lambda_{\omega})\to
	\CF_{J'}(\cH'; \Lambda_{\phi_*(\omega)}),
	\]
	obtained by sending $z^x\cdot \xs$ to $z^x\cdot \phi(\xs)$.
	If $\phi_*(\omega)=\omega$, we have the following relation between the
	tautological map and the map from naturality:
	
	\begin{lem}\label{lem:tautological-map-naturality}
        If $(\phi_t)_{t\in I}$ is an isotopy of
		$(M,\g)$ such that $\phi_0=\id$ and $(\phi_1)_*(\omega) = \omega$, then
		\[
		(\phi_t)_* \doteq (\phi_1)^{\mathrm{taut}}_*
		\]
		on each $\Spin^c$ structure.
	\end{lem}

    \begin{proof}
        By definition, $(\phi_t)_*(z^x\cdot \xs)=
		z^{x+\int_{\Gamma_{\xs,\phi_t}} \omega} \cdot \xs$, where
		$\Gamma_{\xs,\phi_t}$ is the 2-chain swept out by
		$\gamma_{\xs}$ under $\phi_t$.
		Hence, it is sufficient to show that,
		if $\xs$ and $\ys$ represent the same $\Spin^c$ structure, then
		\[
		\int_{\Gamma_{\xs},\phi_t} \omega=\int_{\Gamma_{\ys},\phi_t} \omega.
		\]
        As in the proof of Lemma~\ref{lem:no-monodromy-spinc},
        write $S$ for a 2-chain such that $\d S=\gamma_{\xs}-\gamma_{\ys}$.
        By equation~\eqref{eq:boundary-S-phi-t},
        and since $d \omega = 0$, we have
        \[
        \int_{\Gamma_{\xs,\phi_t}} \omega - \int_{\Gamma_{\ys,\phi_t}}
        \omega = \int_S \omega
        - \int_{\phi_1(S)} \omega.
        \]
        Since $(\phi_1)_*(\omega) = \omega$, we have
        $\int_{\phi_1(S)} \omega =
        \int_{\phi_1(S)} (\phi_1)_* (\omega) = \int_S \omega$,
        and the result follows.
    \end{proof}

\subsection{Monodromy}

	In this section, we give several examples which illustrate
	the existence of monodromy around loops
	of Heegaard diagrams.
	
	\begin{example}\label{rem:projective-monodromy}
		Suppose $D_{\sca,t}$ for $t \in I$ is a path of compressing disks
		that moves just one of the compressing disks
		$D_i$. 	 Further, assume that the center of $D_i$
		traces out a small loop in $U_{\sca}$ that bounds a disk $D_0$.
		Following Remark~\ref{rem:compressing-disks-transition-map},
		by modifying the transition maps for isotopies,
		the path $D_{\sca,t}$ induces a transition map.
		Write $\gamma_{\xs,t}$ for the 1-chain obtained by coning $\xs$
		using $D_{\sca,t}$,
		and write $\Gamma_{\xs}$ for the 2-chain
		swept out by $\gamma_{\xs,t}$ for $t \in I$. Then
		$\Gamma_{\xs}\cup D_0$ is a closed 2-chain, which is a boundary since
		$H_2(U_{\sca})=\{0\}$. Hence, the monodromy of the transition maps
		around the loop $D_{\sca,t}$ is overall
		multiplication by
		\[
		z^{\int_{\Gamma_{\xs}} \omega}=z^{-\int_{D_0} \omega},
		\]
		which may be non-zero.
	\end{example}	
	
	We now show that the perturbed isotopy maps
	can have projectively non-trivial monodromy
	over loops of Heegaard diagrams if we consider multiple $\Spin^c$ structures
	simultaneously.
	
	\begin{lem}\label{lem:monodromy}
        Suppose that $\cH$ is an admissible diagram for
		$(M,\gamma)$ and $(\phi_t)_{t\in I}$ is an isotopy of $M$ such that
		$\phi_0=\phi_1=\id_{(M,\gamma)}$. Let
		\[
		f\colon H_1(M)\to H_2(M)
		\]
		denote the composition
        $H_1(M)\to H_2(M\times S^1)\to H_2(M)$,
		where the first map is obtained via the cross product
		with the fundamental class of $S^1$, and the
		second map is induced by $\phi_t$.
		If $\underline{\frs}_0 \in \Spin^c(M,\g)$ is a fixed $\Spin^c$
		structure, then the isotopy map (summed over all $\Spin^c$ structures)
		\[
		(\phi_t)_*\colon \CF(\cH; \Lambda_{\omega})\to \CF(\cH; \Lambda_{\omega})
		\]
		is projectively equivalent to the map
		\[
		\xs\mapsto z^{\int_{f (\PD[\underline{\frs}(\xs)-\underline{\frs}_0])}
			\omega}\cdot \xs.
		\]
	\end{lem}
	
	\begin{proof}
        As in the proof of Lemma~\ref{lem:no-monodromy-spinc}, let
        $\Gamma_{\xs,\phi_t}$ denote the 2-chain obtained by
        sweeping out $\gamma_{\xs}$ under $\phi_t$.
		Let $\xs_0$ be some fixed intersection point on $\cH$,
        and let $\underline{\frs}_0=\underline{\frs}(\xs_0)$. If $\xs$ is an
		arbitrary intersection point, then
		\[
		\PD[\underline{\frs}(\xs)-\underline{\frs}_0]=\gamma_{\xs}-\gamma_{\xs_0}
		\]
		by \cite{JDisks}*{Lemma~4.7}.
		The claim now follows from the computation
		\begin{align*}
		\int_{\Gamma_{\xs,\phi_t}} \omega-\int_{\Gamma_{\xs_0,\phi_t}}\omega &=
		\int_{\Gamma_{\xs,\phi_t}-\Gamma_{\xs_0,\phi_t}}\omega\\
		&=\int_{f(\gamma_{\xs}-\gamma_{\xs_0})} \omega\\
		&=\int_{f(\PD[\underline{\frs}(\xs)-\underline{\frs}_0])} \omega. \qedhere
		\end{align*}
	\end{proof}
	
	\begin{example}\label{ex:T^3}
		Let $D\subset \bT^2$ be a closed disk, and set $M=(\bT^2\setminus \Int
		(D))\times S^1$. Let the sutures $\g\subset \d M$ be the images of two points
		in $\d D$
		under the action of $S^1$. The $S^1$-action induces a loop $\phi_t$ of
		automorphisms of $(M, \gamma)$
		based at $\id_{(M,\g)}$. The map $f$ is non-zero in this case, and hence
		$(\phi_t)_*$
        is projectively non-trivial when considered over the whole chain complex
		by Lemma~\ref{lem:monodromy}.
	\end{example}

	\subsection{Perturbed triangle maps}
	\label{sec:perturbed-triangle-maps}
	Suppose $(\Sigma,\as,\bs)$ is an admissible diagram for $(M,\gamma)$,
    and $\as'$ is obtained from $\as$ by a sequence of handleslides and
	isotopies.
    Suppose further that $(\Sigma,\as',\as,\bs)$ is admissible.
	Then there is an unperturbed holomorphic triangle map
	\[
	F_{\sca',\sca,\scb} \colon \CF(\Sigma,\as',\as) \otimes \CF(\Sigma,\as,\bs)\to
	\CF(\Sigma,\as',\bs).
	\]
	
	Pick compressing disks $D_{\sca'}$, $D_{\sca}$, and $D_{\scb}$ for $\as'$,
	$\as$,
	and $\bs$, respectively. Note that, since $U_{\sca}=U_{\sca'}$, the disks
	$D_{\sca}$ and $D_{\sca'}$ are compressing disks for the same handlebody.
	If $\psi \in \pi_2(\xs,\ys,\zs)$ is a homology class
	of triangles, we may cone
	the domain of $\psi$ along the compressing disks to obtain a 2-chain
	$\tilde{\cD}(\psi)$ in $M$.
    By integrating $\omega$ over $\tilde{\cD}(\psi)$, we obtain a real number
	$A_\omega(\psi)$. Hence, we obtain a perturbed version of the triangle map
	\[
	F_{\sca',\sca,\scb;\omega} \colon
	\CF(\Sigma,\as',\as; \Lambda_{\omega|_{U_{\sca}}}) \otimes
	\CF(\Sigma,\as,\bs; \Lambda_{\omega})\to \CF(\Sigma,\as',\bs;
	\Lambda_{\omega}).
	\]
	Some care is required in interpreting $\CF(\Sigma,\as',\as; \Lambda_{\omega|_{U_{\sca}}})$, as its definition
	differs slightly from the other two complexes. If $\xs,$ $\ys\in \bT_{\sca'}\cap \bT_{\sca}$
	and $\phi\in \pi_2(\xs,\ys)$, we cone the class $\phi$ in $U_{\sca}$, using the compressing disks
	$D_{\sca}$ and $D_{\sca'}$. We define $A_{\omega}(\phi)$
	as the integral of $\omega$ over this 2-chain in $U_{\sca}$.
	
	Since $H^2(U_{\a})=0$, we may write $\omega|_{U_{\a}}=d\eta$, for some 1-form
	$\eta\in \Omega^1(U_{\a})$. There is a chain isomorphism
	\[
	\Psi_{0\to \omega|_{U_{\a}};\eta}\colon \CF(\Sigma,\as',\as)\otimes \Lambda\to
	\CF(\Sigma,\as',\as;\Lambda_{\omega|_{U_{\sca}}}),	
	\]
	whose construction is analogous to the one in Section~\ref{sec:change-2-form}.
	The complex $\CF(\Sigma,\as',\as)$ contains a cycle $\Theta_{\sca',\sca}$ whose
	homology class is the top-graded generator of $\SFH(\Sigma,\as',\as)$. The
	cycle
	$\Theta_{\sca',\sca}$ is unique up to adding a boundary.
	We define
	\begin{equation}
	\Theta_{\sca',\sca}^\omega := \Psi_{0\to
	\omega|_{U_{\sca}};\eta}(\Theta_{\sca',\sca}\otimes
	1_\Lambda)\in \CF(\Sigma,\as',\as,\Lambda_{\omega|_{U_{\sca}}}).
	\label{eq:top-degree-generator}
	\end{equation}
	A simple modification of Lemma~\ref{lem:independence-rep-omega-spinc} implies that
	$[\Theta^{\omega}_{\sca',\sca}]$ is independent of $\eta$, up to overall
	multiplication by $z^x$.
			
	If the triple $(\Sigma,\as',\as,\bs)$ is admissible, then the transition map
	\[
	\Psi_{\sca\to \sca'}^{\scb} \colon \CF(\Sigma,\as,\bs; \Lambda_{\omega})\to
	\CF(\Sigma,\as',\bs; \Lambda_{\omega})
	\]
	is defined via the formula
	\[
	\Psi_{\sca\to \sca'}^{\scb}(-) :=
	F_{\sca',\sca,\scb;\omega}(\Theta_{\sca',\sca}^\omega,-).
	\]
	
	If $(\Sigma,\as',\as,\bs)$ is not admissible, we define $\Psi_{\sca\to
	\sca'}^{\scb}$ by picking a collection $\as''$ such that the triples
	$(\Sigma,\as',\as'',\bs)$ and $(\Sigma,\as'',\as,\bs)$ are both admissible, and
	setting $\Psi_{\sca\to \sca'}^{\scb}$ to be the composition
	of the triangle maps for  $(\Sigma,\as',\as'',\bs)$ and
	$(\Sigma,\as'',\as,\bs)$.
	A similar construction works for changes of the beta-curves.	
	
	If $(\Sigma,\as,\bs)$ and $(\Sigma,\as',\bs')$ are two admissible diagrams,
	then we define a transition map
	\begin{equation}
	\Psi_{\sca\to \sca'}^{\scb\to \scb'}:=\Psi_{\sca\to\sca'}^{\scb'}\circ
	\Psi_{\sca}^{\scb\to \scb'}.\label{eq:def-transition-alpha-and-beta}
	\end{equation}
	As in the unperturbed setting, the right-hand
	side of equation~\eqref{eq:def-transition-alpha-and-beta}
	is chain homotopic to $\Psi_{\a'}^{\b\to \b'}\circ \Psi_{\a\to \a'}^{\b}$.
	A chain homotopy may be constructed by
	counting holomorphic quadrilaterals. More generally,
	an associativity argument gives the following:

    \begin{prop}\label{prop:associativity}
		The transition map $\Psi_{\sca\to \sca'}^{\scb\to \scb'}$ is well-defined up
		to
		chain homotopy and overall multiplication by $z^x$. Furthermore,
		\[
		\Psi^{\scb'\to \scb''}_{\sca'\to \sca''}\circ \Psi_{\sca\to \sca'}^{\scb\to
			\scb'}\dotsim
		\Psi_{\sca\to \sca''}^{\scb\to \scb''}.
		\]
	\end{prop}
	
	\subsection{Compatibility of the triangle and isotopy maps}

	We now address compatibility of the maps induced by isotopies with the maps
	induced by counting holomorphic triangles.
	
	Let $(\Sigma,\as,\bs)$ be an admissible diagram, and $(\as_t)_{t\in I}$ a
	small Hamiltonian isotopy with $\as_0=\as$, which extends smoothly over $t\in
	\R$
    and is constant outside $I$.
	Then there is a continuation map
	\[
	\Gamma_{\sca_t,J;\omega}\colon \CF_J(\Sigma,\as_0,\bs;\Lambda_{\omega})
	\to \CF_J(\Sigma,\as_1,\bs;\Lambda_{\omega})
	\]
	that counts index-0 $J$-holomorphic curves with boundary on the cylinders
	\[
	C_{\sca_t}:=\{(p,0,t): p \in \as_t, t\in \R\}\quad \text{and} \quad
	C_{\scb}:=\{(p,1,t): p \in \bs, t\in \R\},
	\]
	weighted by their $\omega$-area.
	The cylinder  $C_{\scb}$ is Lagrangian for the product symplectic form, while
	$C_{\sca_t}$ is Lagrangian with respect to a symplectic form that has been
	deformed slightly near $\Sigma\times \{0\}\times \R$; see
	\cite{LOTSpectralSequenceII}*{Equation~(3.25)}.
	Finiteness of the counts contributing to $\Gamma_{\sca_t, J;\omega}$ follows from the work of
	Ozsv\'{a}th and Szab\'{o} \cite{OSDisks}*{Lemma~7.4}, using the admissibility
	assumption on $(\Sigma,\as,\bs)$.
	
	Compatibility of the triangle and continuation maps is given by the following
	lemma, adapted from the work of Lipshitz~\cite{LipshitzCylindrical}*{Section~11}:

	\begin{lem}\label{lem:continuation=triangle}
        Suppose that $(\Sigma,\as,\bs)$ is
		an admissible diagram for $(M,\gamma)$, and $\as'$ is a obtained from $\as$ by a
		small Hamiltonian isotopy $\as_t$ (for some symplectic form on $\Sigma$), such that
		$|\alpha_i'\cap \alpha_j|=2\delta_{ij}$, where $\delta_{ij}$ denotes the
		Kronecker delta. Let $J$ denote a cylindrical almost complex structure on
		$\Sigma\times
		I\times \R$, and let $\Gamma_{\sca_t,J;\omega}\colon \CF_J(\Sigma,\as,\bs;\Lambda_{\omega})\to
		\CF_J(\Sigma,\as',\bs;\Lambda_{\omega})$ denote the continuation map. Then
		\[
		\Gamma_{\sca_t,J;\omega}(-)\dotsim
		F_{\sca',\sca,\scb;\omega}(\Theta_{\sca',\sca}^\omega,-).
		\]
	\end{lem}

	\begin{proof}
        The proof is an adaptation of the proof of the result in the
		unperturbed setting~\cite{LipshitzCylindrical}*{Proposition~11.4}. Lipshitz's
		proof considers the moduli space of holomorphic \emph{monogons}
		associated to the isotopy $\as_t$, which are maps from a
		Riemann surface $S$ to $\Sigma\times [0,\infty)\times \R$, which have
		punctures asymptotic to an intersection point
		$\xs\in \bT_{\sca'}\cap \bT_{\sca}$, and have boundary mapping to the cylinder
		\[
		C_{\sca_t}:=\{(p,0,t):p\in \as_t,t\in \R\}.
		\]		

		Following Lipshitz's proof, a deformation of the almost complex structure on
		$\Sigma\times I\times \R$ gives a chain homotopy between
		$\Gamma_{\sca_t,J;\omega}$
		and the composition
		\[
		F_{\sca',\sca,\scb; \omega}(M_{\sca_t;\omega}(1), -),
		\]
		where $M_{\sca_t;\omega}$ is a map from $\Lambda$ to
		$\CF(\Sigma,\as',\as;\Lambda_{\omega|_{U_{\sca}}})$
		that sums over the count of index 0
		monogons at all intersection points $\xs\in \bT_{\sca'}\cap \bT_{\sca}$. If $\xs\in
		\bT_{\sca'}\cap \bT_{\sca}$ is an intersection
		point and $\phi\in \pi_2(\xs)$ is a
		class of monogons, then $\phi$ may be coned along a family of compressing
		disks
		$D_{\sca_t}$ to obtain a 2-chain $\tilde{\cD}(\phi)$,
		on which we may integrate
		$\omega$. According to \cite{LipshitzCylindrical}*{Lemma~11.8}, there are no
		index 0 classes $\phi\in \pi_2(\xs)$ with holomorphic representatives unless
		$\xs=\Theta_{\sca',\sca}$. Furthermore, a model computation involving a
		stabilized
		diagram of $S^3$ can be used to show that $M_{\sca_t;\omega}(1)=z^x\cdot
		\Theta_{\sca',\sca}^{\omega}$ for some $x\in \R$. We refer the reader to
		\cite{LipshitzCylindrical}*{Proposition~11.4} for more details on the model
		computation.
	\end{proof}
	
	Next, we consider a diffeomorphism $\phi\colon \Sigma\to \Sigma$, which is near
	$\id_{\Sigma}$, and is the time 1 flow of a Hamiltonian vector field
	for some symplectic form on $\Sigma$. Write $\phi_t$ for the time $t$
	flow of this Hamiltonian vector field. In particular, $\phi_1 = \phi$.
	By extending $\phi_t$ to an isotopy of $M$, we obtain an isotopy map
	$(\phi_t)_*$ on the perturbed
	Floer homology, as in Section~\ref{sec:perturbed-diffeos}.
	
	\begin{prop}\label{prop:tautological=triangle}
      Suppose $(\Sigma,\as,\bs)$
	  is an admissible diagram for a sutured manifold $(M,\g)$, which
	  is equipped with a closed 2-form $\omega$,  and $\phi_t\colon \Sigma \to \Sigma$
	  is the flow of a Hamiltonian vector field (for some symplectic form on $\Sigma$),
	  as above. Write $\as_t = \phi_t(\as)$ and $\bs_t = \phi_t(\bs)$.
      Then the perturbed isotopy map $(\phi_t)_*$ satisfies
	  \[
	  (\phi_t)_*\dotsim \Psi_{J\to \phi_* (J)} \circ \Psi_{\sca\to\sca_1}^{\scb \to \scb_1}.
	  \]
	\end{prop}	

	\begin{proof}
		The first step is to interpret the isotopy map $(\phi_t)_*$
		as a continuation map. Consider the two cylinders
		$C_{\sca_t}$ and $C_{\scb_t}$, where
		$\as_t$ and $\bs_t$ are the images of $\as$
		and $\bs$ under $\phi_t$. Let $\tilde{J}$
		denote the almost complex structure
		on $\Sigma\times I\times \R$ obtained by pushing forward
		a generic cylindrical almost complex structure $J$ along the map
		$\Phi(x,s,t) = (\phi_t(x),s,t)$.
        For $\phi_t$ sufficiently small, $\tilde{J}$
        will be tamed by a product symplectic form, and
		achieve transversality at index zero holomorphic curves with boundary on
		$C_{\sca_t}$ and $C_{\scb_t}$.
		Hence, if $\Gamma_{\sca_t,\scb_t,\tilde{J};\omega}$ denotes
		the map that counts index zero $\tilde{J}$-holomorphic curves  with boundary
		on $C_{\sca_t}$ and $C_{\scb_t}$, we have
		\begin{equation}
		\Gamma_{\sca_t,\scb_t,\tilde{J};\omega}(\xs)=(\phi_t)_*(\xs).
		\label{eq:tautological=continuation}
		\end{equation}	
		
		We now consider a 1-parameter family of cylinders $C_{\sca_t^\tau}$, $C_{\scb_t^\tau}$,
        and almost complex structures $\tilde{J}^\tau$ for $\tau\in [0,\infty)$,
		as follows. The cylinder $C_{\sca_t^\tau}$ is obtained by translating $C_{\sca_t}$
		downward in the $\R$-direction by $\tau$ units. The cylinder $C_{\scb_t^\tau}$ coincides
		with $C_{\scb_t}$ for all $\tau$. The almost complex structure $\tilde{J}^\tau$ is
		obtained by translating $\tilde{J}$ upward in the $\R$-direction by $\tau$ units.
		
		A chain homotopy $H$ is defined by
		counting index $-1$, $\tilde{J}^\tau$-holomorphic curves
		with boundary on $C_{\sca_t^\tau}$ and $C_{\scb_t^\tau}$ for $\tau\in [0,\infty)$,
		weighted by their $\omega$-area.
		Applying Gromov compactness to the moduli space
        of index 0, $\tilde{J}^\tau$-holomorphic curves
        with boundary on $C_{\sca_t^\tau}$ and $C_{\scb_t^\tau}$
        for $\tau \in [0,\infty)$, we obtain that
		\begin{equation}
		\Gamma_{\sca_t,\scb_t,\tilde{J};\omega} +
		\Psi_{J\to \phi_*(J)} \circ \Gamma_{\scb_t,J} \circ
		\Gamma_{\sca_t,J} = \d \circ H + H \circ \d.
        \label{eq:tautological=triple-composition}
		\end{equation}
        Indeed, at $\tau = 0$, we obtain $\Gamma_{\sca_t,\scb_t,\tilde{J}}$.
        At $\tau \to \infty$, we obtain
        $\Psi_{J\to \phi_*(J)} \circ \Gamma_{\scb_t,J} \circ \Gamma_{\sca_t,J}$.
        The only other way a curve may break is for a
        family to split into an index $-1$ curve, giving $H$,
        and an index 1 curve, giving $\d$.
		Combining equations~\eqref{eq:tautological=continuation}
        and~\eqref{eq:tautological=triple-composition} with
		Lemma~\ref{lem:continuation=triangle}, the result follows.
	\end{proof}

	\subsection{Proof of naturality}
	\label{sec:proof-naturality}
	
	We now prove Theorem~\ref{thm:naturality-perturbed}, naturality of the
	perturbed invariants:
	
	\begin{proof}[Proof of Theorem~\ref{thm:naturality-perturbed}]
		Our proof follows the framework of \cite{JTNaturality}. Suppose that
		$(M,\gamma)$ is a balanced sutured manifold with a closed 2-form $\omega$. We define a
		directed graph $\cG_{(M,\gamma)}$, as follows. The vertices of $\cG_{(M,\gamma)}$
		consist of \emph{isotopy diagrams} of $(M,\gamma)$; i.e., tuples $(\Sigma,A,B)$
		consisting of an embedded Heegaard surface $\Sigma$, and isotopy classes $A$ and $B$ of
		attaching curves. If $\cH = (\Sigma,\as,\bs)$ is a Heegaard diagram, we write $[\cH]$ for
		the induced isotopy diagram.
				
		If $H_1$ and $H_2$ are two isotopy diagrams, we define the set of edges in
		$\cG_{(M,\gamma)}$ connecting $H_1$ and $H_2$ to be the union
		\begin{equation}
		\cG_{(M,\gamma)}(H_1,H_2):=\cG_{\alpha}(H_1,H_2)\cup \cG_{\beta}(H_1,H_2)\cup
		\cG_{\stab}(H_1,H_2)\cup \cG^0_{\diff}(H_1,H_2),
        \label{eq:G-M-gamma-def}
		\end{equation}
		as follows. The set $\cG_{\alpha}(H_1,H_2)$ consists of a single arrow if
		$H_1$ and $H_2$ share the same Heegaard surface, have isotopic beta-curves, and
		have alpha-curves that are related by a sequence of handleslides and
		isotopies. The set $\cG_{\alpha}(H_1,H_2)$ is empty otherwise. The set
		$\cG_{\beta}(H_1,H_2)$ is defined similarly. The set $\cG_{\stab}(H_1,H_2)$
		has	a single arrow if $H_1$ and $H_2$ are related by a stabilization or destabilization,
        and is empty otherwise. Finally, $\cG^0_{\diff}(H_1,H_2)$ is the set of all automorphisms
		of	$(M,\gamma)$ which move $H_1$ to $H_2$, and are isotopic to the identity of
		$(M,\gamma)$.
		Write $\cG_{\alpha}$ for the union over all pairs $(H_1,H_2)$ of
		$\cG_{\alpha}(H_1,H_2)$, and define $\cG_{\beta}$, $\cG_{\stab}$, and
		$\cG_{\diff}^0$ similarly.
		
		If $H$ is an isotopy diagram, write $\SFH(H;\Lambda_\omega)$ for the
		projective transitive system of $\Lambda$-modules,
		indexed by pairs $(\cH,J)$, where $\cH = (\S,\as,\bs)$
		is an admissible Heegaard diagram with $[\cH]=H$, and $J$ is a generic almost
		complex structure on $\S \times I \times \R$. The transition maps may be constructed using the
		holomorphic triangle maps, as in Section~\ref{sec:perturbed-triangle-maps},
		as well as change of almost complex structure maps from
		Section~\ref{sec:change-of-almost-complex-structure}.
		Propositions~\ref{prop:associativity} and ~\ref{prop:tautological=triangle}
		imply that this gives a projective transitive system of
		$\Lambda$-modules.		
		
		We consider the following cycles in $\cG_{(M,\g)}$:
		\begin{enumerate}[ref=L-\arabic*, label=(L-\arabic*)]
			\item\label{loop-1} A loop formed by a stabilization followed by a
			destabilization.
			\item\label{loop-2} A rectangular subgraph
			\[
			\begin{tikzcd} H_1\arrow[r, "e"]\arrow[d,"f"] & H_2\arrow[d, "g"]\\
			H_3\arrow[r, "h"]& H_4
			\end{tikzcd}
			\]
			of $\cG_{(M,\gamma)}$, where one of the following hold:
			\begin{enumerate}[ref=R-\arabic*, label=(R-\arabic*)]
				\item Both $e$, $h \in \cG_\alpha$ and $f$, $g\in \cG_{\beta}$.
				\item Either $e$, $h \in \cG_{\alpha}$, or $e$, $h \in \cG_{\beta}$.
				Furthermore,
				$f$, $g \in \cG_{\stab}$.
				\item Either $e$, $h \in \cG_{\alpha}$, or $e$, $h\in \cG_{\beta}$.
				Furthermore,
				$f$, $g \in \cG_{\diff}^0$.
				\item The edges $e$, $f$, $g$, $h$ are all in $\cG_{\stab}$. Furthermore,
				$e$ and $h$
				correspond to stabilizing in a 3-ball $B$, while $f$ and $g$ correspond to
				stabilizing in a 3-ball $B'$, and $B \cap B'=\emptyset$.
				\item Both $e$, $h \in \cG_{\stab}$, while $f$, $g \in \cG_{\diff}^0$.
				Furthermore,
				$f$ and $g$ may be induced by the same diffeomorphism $\phi$ of
				$(M,\gamma)$,
				and the stabilization 3-ball for $e$ is pushed forward to the stabilization
				3-ball
				for $h$ by $\phi$.
			\end{enumerate}
			\item\label{loop-3} A loop formed by an edge in $\cG_{\diff}^0(H,H)$.
			\item\label{loop-4} A simple handleswap loop; see Figure~\ref{fig:handleswap},
			which is \cite{JTNaturality}*{Figure~4}, for a depiction.
		\end{enumerate}
		
        \begin{figure}
        \includegraphics{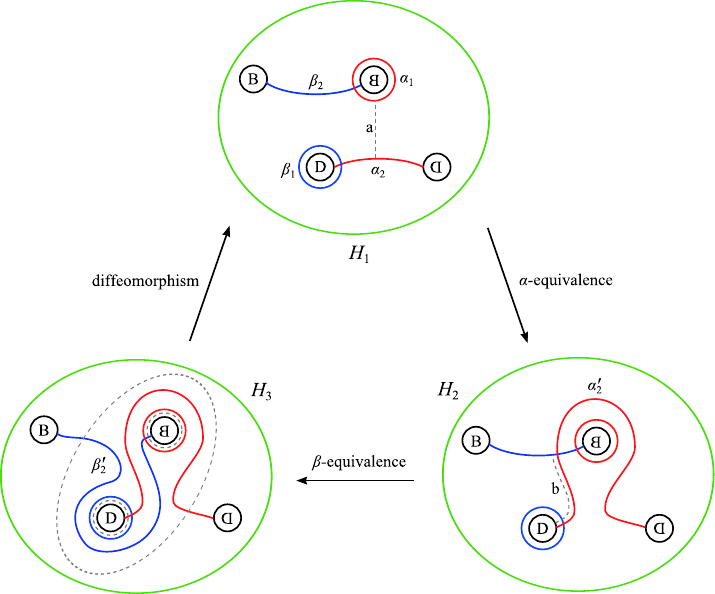}
        \caption{A simple handleswap, which is a loop of diagrams
        consisting of an $\alpha$-handleslide, a $\beta$-handleslide, and a diffeomorphism.
        The green curve is the boundary of the punctured genus two surface $P$ that is obtained
        by identifying the circles marked with corresponding letters (namely, $B$ and $D$).
        We draw the $\alpha$-curves in red and the $\beta$-curves in blue.}
        \label{fig:handleswap}
        \end{figure}
			
		Commutativity of the transition maps along the loops \eqref{loop-1}--\eqref{loop-4}
        correspond to the axioms for a \emph{strong Heegaard invariant} \cite{JTNaturality}*{Definition~2.32}.
		According to \cite{JTNaturality}*{Theorem~2.38}, it suffices to
		prove that the perturbed transition maps have no monodromy around loops~\eqref{loop-1}--\eqref{loop-4}.
		
		As in Remark~\ref{rem:system-of-systems}, to define a projectively transitive
		system indexed by all pairs $(\cH,J)$, it is sufficient to define a morphism
		of transitive systems for each edge of $\cG_{(M,\gamma)}$, and show that there is
		only projective monodromy around loops \eqref{loop-1}--\eqref{loop-4}.
				
		We define chain maps for edges in $\cG_{\alpha}(H_1,H_2)$ and
		$\cG_{\beta}(H_1,H_2)$  to be triangle maps, as described in
		Section~\ref{sec:perturbed-triangle-maps}. Chain maps for stabilizations are
		described in Section~\ref{sec:stabilization}. Maps for edges in
		$\cG^0_{\diff}(H_1,H_2)$ are defined in Section~\ref{sec:perturbed-diffeos}.
		It is straightforward to see that these chain maps induce morphisms of transitive systems
		between the transitive systems associated to each isotopy diagram.
		
		The main subtlety compared to the unperturbed setting is that the map
		associated to a diffeomorphism $\phi$ in $\cG^0_{\diff}(H_1,H_2)$ is defined
		with an auxiliary choice of an isotopy $\phi_t$ connecting $\phi$ to
		$\id_{(M,\gamma)}$. The induced map $\phi$ is only well-defined as a
		projective map when restricted to each $\Spin^c$ structure by
		Lemma~\ref{lem:no-monodromy-spinc}, or when $[\omega]=0$. See
		Remark~\ref{lem:monodromy} for an example illustrating the subtlety.	
		
		We now verify that the monodromy around Loops~\eqref{loop-1}--\eqref{loop-4}
		is of projective type. The monodromy around loops of type~\eqref{loop-1} is
		clearly trivial. Similarly to the unperturbed setting,
		associativity of the holomorphic
		triangle maps,  Proposition~\ref{prop:associativity}, implies that loops
		of type~\eqref{loop-2} induce projectively trivial monodromy. Loops of
		type~\eqref{loop-3} induce projectively trivial monodromy by
		Lemma~\ref{lem:no-monodromy-spinc} and Proposition~\ref{prop:tautological=triangle},
        when restricted to individual $\Spin^c$ structures, or when $[\omega]=0$.
        The main claim follows once we verify that there is only
		projective monodromy around simple handleswap loops \eqref{loop-4},
		which is verified in Lemma~\ref{lem:perturbed-handleswap} below.
	\end{proof}	
	
	\begin{lem}\label{lem:perturbed-handleswap}
        Suppose $(M,\g)$ is a balanced sutured manifold, with a closed
        2-form $\omega$, and $\ufrs \in \Spin^c(M,\g)$. Suppose further that
		\[
		\begin{tikzcd}[row sep=small]
		\cH_1\arrow[dr,"e_\alpha"]\\
		&\cH_2\arrow[dl,"e_{\beta}"]\\
		\cH_3\arrow[uu,"\phi_1"]
		\end{tikzcd}
		\]
		is a simple handleswap loop, where $\cH_1$, $\cH_2$, and $\cH_3$ are admissible diagrams
        of $(M,\g)$, and $e_\alpha \in \cG_\alpha$, $e_{\beta} \in \cG_\beta$,
        and $(\phi_t)_{t\in I}$ is an isotopy with $\phi_0 = \id_{(M,\gamma)}$. Then
		\[
		(\phi_t)_*\circ \Psi_{e_{\beta}} \circ \Psi_{e_{\alpha}} \dotsim
        \id_{\CF(\cH_1,\ufrs;\Lambda_{\omega})}.
		\]
		The same statement holds for the total complex $\CF(\cH_1;\Lambda_{\omega})$ if $[\omega]=0$.
	\end{lem}

	\begin{proof}
        By definition, the diagrams $\cH_1$, $\cH_2$, and $\cH_3$ are all
		2-fold stabilizations of a fixed diagram $\cH=(\Sigma,\as,\bs)$. If $i\in
		\{1,2,3\}$,
		write $\cH_{i}'=(\Sigma_0,\as_i',\bs_i',p_0)$ for the
		genus 2 portion of $\cH_i$ in the handleswap region. With this notation, we
		think of
		$\cH_i$ as $\cH\# \cH_{i}'$, where the connected sum
		is taken at $p_0\in \Sigma_0$ and a point $p\in \Sigma$. The diagrams $\cH_i'$
		are all genus 2 diagrams for
		$S^3$. Note that
		\[
		\bs_2'=\bs_1'\quad \text{\and} \quad \as_3'=\as_2'.
		\]		
		
		The map $\Psi_{e_{\alpha}}$ may be computed as the composition of a triangle map
		for an alpha-handleslide, followed by a continuation map to move the
		alpha-curves on $\cH$ back to their original position. Similarly, the map
		$\Psi_{e_{\beta}}$ may be computed as the composition of the
		triangle map for a beta-handleslide, followed by a
		continuation map to move the beta-curves
		 on $\cH$ back to their original position.
		The
		map $(\phi_t)_*$ is the isotopy map
		described in Section~\ref{sec:perturbed-diffeos}.
		
		For a sufficiently stretched almost complex structure $J(T)$ along the connected sum tube of $\S \# \S_0$,
        the proof of stabilization invariance implies
		that the unperturbed complex for $\cH_i$ decomposes as a tensor product:
		\begin{equation}
		\CF_{J(T)}(\cH_i)\iso \CF_{J}(\cH)\otimes_{\bF_2}
		\langle \ve{c}_i \rangle,\label{eq:unperturbed-tensor-product}
		\end{equation}
		where $\{\ve{c}_i\}=\bT_{\sca_i'}\cap \bT_{\scb_i'}$,
		and $\langle \ve{c}_i \rangle$
		denotes the 1-dimensional vector space over $\bF_2$, generated by $\ve{c}_i$,
        for $i \in \{1,2,3\}$.
		
		In the unperturbed setting, handleswap invariance
		\cite{JTNaturality}*{Theorem~9.30}
		is proven by showing
		\begin{equation}
		\Psi_{e_{\alpha}} = \left(\Gamma_{\sca_{t}, J}\circ \Psi_{\sca\to
		\sca^H}^{\scb} \right) \otimes (\ve{c}_1\mapsto \ve{c}_2)
        \label{eq:unperturbed-tensor-product-maps}
		\end{equation}
		with respect to the chain isomorphism of
		equation~\eqref{eq:unperturbed-tensor-product}, where
		$\as^H$ is a small Hamiltonian translate of $\as$, and $\as_{t}$
		is a Hamiltonian isotopy moving $\as^H$ back to $\as$. 	
		A similar tensor product description holds for the unperturbed version of
		$\Psi_{e_{\beta}}$.	
	
		For the perturbed versions, an extension of
		Lemma~\ref{lem:perturbed-stabilization-chain-map}
		to genus 2 stabilizations gives an analog of equation~\eqref{eq:unperturbed-tensor-product}
		to the perturbed setting, namely:
		\begin{equation}
		\CF_{J(T)}(\cH_i,\ufrs; \Lambda_{\omega})\iso \CF_{J}(\cH,\ufrs;
		\Lambda_\omega)\otimes_{\bF_2} \langle
		\ve{c}_i\rangle.\label{eq:perturbed-tensor-product}
		\end{equation}
		We now show that a similar tensor product decomposition as in
		equation~\eqref{eq:unperturbed-tensor-product-maps}  holds for the
		perturbed versions of $\Psi_{e_{\sca}}$ and $\Psi_{e_{\scb}}$.
		
        Firstly, if $\psi\# \psi_0$ is a class of triangles on
		$(\Sigma\# \Sigma_0,\as^H\cup \as_2',\as\cup \as_1', \bs\cup \bs_1')$, then
		\begin{equation}
		A_\omega(\psi\#
		\psi_0)=A_\omega(\psi)+A_{\omega}(\psi_0).\label{eq:omega-areas-handleswap}	
		\end{equation}
		According to the proof of \cite{JTNaturality}*{Proposition~9.31}, for a
		sufficiently stretched almost complex structure, all index 0 triangles $\psi\#
		\psi_0$ that are counted by $\Psi_{\sca\cup \sca_1'\to \sca^H\cup
		\sca_2'}^{\scb\cup\scb_1'}$ have $\mu(\psi)=0$. Furthermore, if $\mu(\psi)=0$,
		then
		\begin{equation}
		|\cM(\psi)|\equiv \sum_{\substack{\psi_0\in
		\pi_2(\Theta_{\sca_2',\sca_1'},\ve{c}_1,\ve{c}_2)\\ n_{p_0}(\psi_0)=n_p(\psi)}}
		|\cM(\psi\# \psi_0)| \mod{2}.\label{eq:congurent-counts-handleswap}
		\end{equation}
		
	Next, we claim that $A_{\omega}(\psi_0)$ is independent of the triangle
	class $\psi_0\in \pi_2(\Theta_{\sca_2',\sca_1'},\ve{c}_1,\ve{c}_2)$. This is
	established by observing that any two classes in
	$\pi_2(\Theta_{\sca_2',\sca_1'},\ve{c}_1,\ve{c}_2)$ differ by a sum of doubly
	periodic domains. Doubly periodic domains on $\cH_i'$ cone to closed 2-chains in
	$C_2(S^3)$, and hence do not affect the $\omega$-area, so $A_{\omega}(\psi_0)$
	is independent of the triangle class.  A similar claim holds for triangles in
	$\pi_2(\ve{c}_2,\Theta_{\scb_1',\scb'_3},\ve{c}_3)$.
		
	Combining equations~\eqref{eq:omega-areas-handleswap},
	\eqref{eq:congurent-counts-handleswap}, and the independence of
	$A_{\omega}(\psi_0)$ from $\psi_0$,
	we obtain that the perturbed transition maps satisfy
	\begin{equation}
	\begin{split}
	&(\phi_t)_*\circ \Psi_{e_{\scb}}\circ \Psi_{e_{\sca}}\\ \doteq &(\phi_t)_*
	\circ \left( \left(\Gamma_{\scb_{t}} \circ \Psi_{\sca}^{\scb\to \scb^H} \right)\otimes
    (\ve{c}_2\mapsto  \ve{c}_3)\right)\circ \left( \left(\Gamma_{\sca_{t}}\circ
	\Psi_{\sca\to \sca^H}^{\scb} \right)\otimes ( \ve{c}_1\mapsto  \ve{c}_2)\right),
	\end{split}
	\label{eq:expanded-handle-swap-loop}
	\end{equation}	
	with respect to the tensor product decomposition from
	equation~\eqref{eq:perturbed-tensor-product}.

	Since the isotopy $\phi_t$ is supported in the 3-ball of the handleswap, it
	follows that
	\begin{equation}
	(\phi_t)_*\doteq \id_{\CF(\cH,\ufrs;\Lambda_{\omega})} \otimes ( \ve{c}_3\mapsto
	\ve{c}_1).
	\label{eq:expanded-diffeo-map}
	\end{equation}
	Furthermore, by Lemma~\ref{lem:continuation=triangle},
	\begin{equation}
	\Gamma_{\scb_{t}}\circ \Psi_{\sca}^{\scb\to \scb^H}\dotsim
    \id_{\CF(\cH,\ufrs;\Lambda_\omega)}\quad \text{and}\quad \Gamma_{\sca_{t}}\circ
    \Psi_{\sca\to \sca^H}^{\scb}\dotsim\id_{\CF(\cH,\ufrs;\Lambda_\omega)}.
    \label{eq:handleswap-continuation=triangle}
	\end{equation}
    Combining equations~\eqref{eq:expanded-handle-swap-loop},
    \eqref{eq:expanded-diffeo-map}, and~\eqref{eq:handleswap-continuation=triangle}
    yields the main statement.
	\end{proof}

	\section{Perturbed sutured cobordism maps}	
    \label{sec:cob-construct}	
	
    In this section, we define the perturbed sutured cobordism maps,
    and prove that they are well-defined in Proposition~\ref{prop:perturbedcobordismmapswelldefined}.
    Furthermore, we prove the composition law, Proposition~\ref{prop:composition-law},
    the effect of changing the 2-form on the cobordism, Lemma~\ref{lem-modify-2-form-on-interior},
    and finally compare the perturbed and unperturbed maps when the 2-form vanishes on the
    boundary in Lemma~\ref{lem:perturbed-coincide}.

	\subsection{The perturbed contact gluing map}
	
	We now describe a perturbed version of the Honda--Kazez--Mati\'c contact
	gluing map~\cite{HKMTQFT}.
	Suppose $(M,\g)$ is a sutured submanifold of $(M',\g')$ (i.e., $M$ is a submanifold with boundary
    of $M'$ such that $M \subset \Int(M')$), $\omega$ and $\omega'$
	are closed 2-forms on $M$ and $M'$, respectively, such that
	$\omega = \omega'|_M$, and $\xi$ is a co-oriented contact structure on $M' \setminus \Int(M)$.
	Let $\underline{\frs}$ be a $\Spin^c$ structure on $M$ represented by a
	non-vanishing vector field $v$, and let $\underline{\frs}'$ be the $\Spin^c$
	structure
	on $M'$ obtained by gluing $v$ to $\xi^\perp$. We will define a gluing map
	\[
	\Phi_{\xi;\omega} \colon \SFH(-M,-\g,\underline{\frs};\Lambda_{\omega})\to
	\SFH(-M',-\g',\underline{\frs}';\Lambda_{\omega'}),
	\]
	by adapting the construction from the unperturbed setting. Our description
	will use the reformulation of the gluing map given in
	\cite{JuhaszZemkeContactHandles} using contact handles.
	See \cite{JuhaszZemkeContactHandles}*{Definition~3.11} for background on
	contact handles in this setting.
	
\begin{rem}
We require $M'$ to have no closed components, though we allow $M'\setminus \Int(M)$
to have what Honda, Kazez, and Mati\'{c} refer to as \emph{isolated components},
which are components of $M'\setminus \Int(M)$ that are disjoint from $\d M'$. These
are permitted since the construction from \cite{JuhaszZemkeContactHandles} had a
contact 3-handle map, which was not present in \cite{HKMTQFT}.
\end{rem}	
	
	On Heegaard diagrams, adding a contact 0-handle has the effect of adding a disk
	$D$
	to the Heegaard surface, with no alpha or beta-curves. The contact 0-handle map
	is the canonical chain isomorphism between $\CF(\Sigma,\as,\bs)$ and
	$\CF(\Sigma\sqcup D,\as,\bs)$. This extends to the perturbed setting via the
	formula
	\[
	\Phi_{\xi;\omega}(z^x\cdot \xs)=z^x\cdot \xs,
	\]
	 for any closed 2-form on the 0-handle.
	
	Adding a contact 1-handle has the effect of attaching a band to the boundary of the
	Heegaard surface. The contact 1-handle map is the canonical chain
	isomorphism between $\CF(\Sigma,\as,\bs)$ and $\CF(\Sigma\cup B,\as,\bs)$, which
	extends to a map on the perturbed complexes with no complications.
	
	The contact 2-handle map is slightly more involved. The effect on diagrams
	is to add a band and a pair of new curves, $\alpha$ and $\beta$, which have a
	single intersection point $c$ in the band.  See
	\cite{JuhaszZemkeContactHandles}*{Figure~3.11} for the precise configuration.
	The contact 2-handle map is defined via the formula
	\[
	\Phi_{\xi;\omega}(z^x\cdot \xs)=z^x\cdot \xs\times c.
	\]
    To see that this is a chain map
	on the perturbed complexes, note that  all disks counted by
	$\d(\ve{x} \times c)$ have homology class of the form $\phi \#
	e_{c}$, where $\phi \in \pi_2(\xs,\ys)$ is a homology class, and
	$e_{c}$ is the constant class at $c$. However, $A_{\omega'}(\phi \#
	e_{c}) = A_{\omega}(\phi)$. Hence, the contact 2-handle map is a chain map
    on the perturbed complexes.
	
	Finally, a contact 3-handle is attached along
    a boundary component $S^2\subset \d M$ which is a 2-sphere
	with a single suture $s$. Then pick a diagram $(\Sigma,\as\cup \{\alpha_0\},\bs\cup
	\{\beta_0\})$, where $\alpha_0$ and $\beta_0$ are parallel
	to the boundary component of
	$\Sigma$ corresponding to $s$, and intersect each other in a pair of points.
	The contact 3-handle map is obtained by filling $s\subset \d \Sigma$ with a
	disk $D$, and setting
	\[
	\Phi_{\xi;\omega}(z^x \cdot \xs\times \theta)=\begin{cases} z^x \cdot \xs&
	\text{ if } \theta=\theta^-,\\
	0& \text{ if } \theta=\theta^+,
	\end{cases}	
	\]
	where $\{\theta^+,\theta^-\}=\alpha_0\cap \beta_0$, with relative grading $\mu(\theta^+, \theta^-) = 1$
	induced by the Maslov index on $(-\Sigma,\as\cup \{\alpha_0\}, \bs\cup
	\{\beta_0\})$. (The formula is the same as the 4-dimensional 3-handle map).
	Note that the contact 3-handle map is only defined if $\d M$ has at least one other
	boundary component. Furthermore, we must either choose $(\Sigma,\as\cup
	\{\alpha_0\}, \bs\cup \{\beta_0\})$ so that $\alpha_0$ and $\beta_0$ are
	adjacent to another component of $\d \Sigma$, or we must stretch the almost
	complex structure along a circle bounding $\alpha_0$ and $\beta_0$. We focus on
	the case when $(\Sigma,\as\cup \{\alpha_0\}, \bs\cup \{\beta_0\})$ has been
	chosen so that $\d \Sigma$ has an additional boundary component adjacent to
	$\alpha_0$ and $\beta_0$. (The more general case requires using a holomorphic
	degeneration argument \cite{OSLinks}*{Proposition~6.5}, but follows similarly.)
	In this situation, an
	index 1 class on $(-\Sigma,\as\cup \{\alpha_0\}, \bs\cup \{\beta_0\})$
	with holomorphic representatives has one of the following forms:
	\begin{itemize}
		\item $\phi\#e_\theta$, where $\phi$ is an index 1 class on $(-\Sigma\cup D,
		\as,\bs)$, with zero multiplicity on $D$, and $e_\theta$ is the constant class at $\theta\in \alpha_0\cap
		\beta_0$.
		\item $e_{\xs}\# \phi_0$, where $\phi_0$ is one of the two bigons between
	$\alpha_0$ and $\beta_0$.
	\end{itemize}
	To see that  the contact 3-handle map is a chain map, it suffices to show that
	the two
	bigons have the same $\omega$-area. The difference of the bigons is a periodic
	domain, which cones to a 2-sphere bounding the $S^2$ boundary
	component of $\d M$ which is filled in by the 3-handle. Since $\omega$ extends
	over the contact 3-handle, $\omega$
	must integrate to zero on this 2-sphere, and hence have equal area on
	the cones of the two bigons.
	
	As in the unperturbed case,
	the composition of the contact handle maps for a canceling pair of contact $i$
	and $i+1$
	handles coincides with the transition map from naturality (up to an overall
	factor of $z^x$);
	see \cite{JuhaszZemkeContactHandles}*{Figures~3.13, 3.14}.
	By following our contact handle proof of invariance of the contact gluing map
	in the unperturbed case \cite{JuhaszZemkeContactHandles}*{Theorem~3.14}, it
	follows that the perturbed contact gluing map is well-defined up to an overall
	factor of $z^x$, when restricted to each $\Spin^c$ structure on $(M,\g)$.
	Furthermore, if $[\omega']=0$, then the gluing map is well-defined on all
	$\Spin^c$ structures, up to an overall factor of $z^x$.

\subsection{Perturbed maps for cylinders}
    \label{sec:cylinders}
	
	We now define the 4-dimensional cobordism maps for $W=I\times M$,
	 equipped with a closed 2-form $\omega$.
	
	Recall that a sutured manifold cobordism is called
	 \emph{special} if it is a product along the
    boundary, with an $I$-invariant contact structure compatible with the
    dividing sets; see \cite[Definition~5.1]{JCob}.	
	Suppose that $\cW=(W,Z,[\xi])\colon (M_0,\g_0)\to (M_1,\g_1)$
	is a special cobordism which is equipped with a Morse function $f$
	with no critical points, and let $v$ be a gradient-like vector field for $f$.
	
	To define the map for $\cW$, we first pick an admissible diagram
	$\cH_0=(\Sigma_0,\as_0,\bs_0)$ for $(M_0,\g_0)$.
	The flow of $v$ induces a diffeomorphism between $M_0$ and $M_1$,
	and we write $\cH_1=(\Sigma_1,\as_1,\bs_1)$ for the push-forward of $\cH_0$
	under this diffeomorphism.
	If $\xs\in \bT_{\sca_0}\cap \bT_{\scb_0}$, we write $v_*(\ve{x})\in
	\bT_{\sca_1}\cap \bT_{\scb_1}$ for the corresponding intersection point.
	Write $\Gamma_{\xs}$ for the 2-chain traced out by the flow of $v$ applied to
	$\gamma_{\xs}\subset M_0$.
	
	We define the perturbed cylinder map
	\[
	F_{\cW;\omega,(f,v)}\colon \CF(\cH_0;\Lambda_{\omega|_{M_0}})\to
	\CF(\cH_1;\Lambda_{\omega|_{M_1}})	
	\]
	via the formula
	\begin{equation}
	F_{\cW;\omega,(f,v)}(z^x\cdot \xs)=z^{x+\int_{\Gamma_{\xs}} \omega}\cdot
	v_*(\xs).\label{eq:perturbed-map-for-cylinders}
	\end{equation}
	As in Remark~\ref{rem:1-map-determines-morphism}, for a choice of diagram
	$\cH_0$ of $(M_0,\g_0)$ and $\ufrs\in \Spin^c(M_0,\g_0)$, equation~\eqref{eq:perturbed-map-for-cylinders} gives a
	morphism of transitive systems from $\CF(M_0,\g_0, \ufrs;\Lambda_{\omega|_{M_0}})$ to
	$\CF(M_1,\g_1, v_*(\ufrs);\Lambda_{\omega|_{M_1}})$.

	\begin{lem}\label{lem:cylinder-maps-well-defined}
	Suppose that $\cW=(W,Z,[\xi])$ is a special cobordism with a Morse function $f$
	with no critical points and gradient-like vector field $v$.
	\begin{enumerate}
	\item \label{lem:no-crit-points-1} The map $F_{\cW,\omega;(f,v)}$ is a chain
	map.
	\item \label{lem:no-crit-points-2} The induced morphism of transitive systems
	is independent of the choice of Heegaard diagram $\cH_0$ for $(M_0,\g_0)$.
	\item \label{lem:no-crit-points-3} The induced morphism of transitive systems
	is independent of $v$.
	\end{enumerate}
	\end{lem}
	\begin{proof} Claim~\eqref{lem:no-crit-points-1}, that $F_{\cW,\omega;(f,v)}$
	is a chain map, follows from Stokes' theorem.
	
We now consider claim~\eqref{lem:no-crit-points-2}, that the morphism induced
by $F_{\cW,\omega;(f,v)}$ is independent of $\cH_0$. This amounts to showing
that the maps $F_{\cW,\omega;(f,v)}$ commute with the transition maps for
changing diagrams, up to an overall factor of $z^x$. We focus on the case when
we have two diagrams for $(M_0,\g_0)$ that are related by a single
beta-handleslide or isotopy. We leave verification of
claim~\eqref{lem:no-crit-points-2} for other Heegaard moves to the reader.
	
Suppose that $(\Sigma_0,\as_0,\bs_0,\bs_0')$ is an admissible Heegaard triple
for a beta-handleslide or isotopy in $(M_0,\g_0)$. Set
$\cH_0=(\Sigma_0,\as_0,\bs_0)$ and
$\cH_0'=(\Sigma_0,\as_0,\bs'_0)$.
Let $\cH_1$ and $\cH_1'$ denote their images
in $M_1$ under the flow of $v$.

It is sufficient to consider the claim when the
top-graded generator of $\SFH(\Sigma_0,\bs_0,\bs_0')$ is represented by a
single intersection point $\Theta_{\scb_0,\scb_0'}\in \bT_{\scb_0}\cap
\bT_{\scb_0'}$, since a general beta-isotopy or handleslide may
be decomposed into a sequence of beta-isotopies and handleslides which each satisfy this condition.
		
Let $\psi\in \pi_2(\xs,\Theta_{\scb_0,\scb_0'},\zs)$ be a homology class of
triangles, where $\xs \in \bT_{\sca_0} \cap \bT_{\scb_0}$ and $\zs \in \bT_{\sca_0}
\cap \bT_{\scb_0'}$. Let $\Phi\colon I\times M_0\to W$ denote the flow of $v/v(f)$.
Let $C_3\subset W$ be the 3-chain $\Phi(I\times \tilde{\cD}(\psi))$, where
$\tilde{\cD}(\psi) \subset M_0$ is the 2-chain constructed
		in Section~\ref{sec:perturbed-triangle-maps}. Since
		\begin{equation}
		\d C_3=\Phi\left(\{1\}\times \tilde{\cD}(\psi)\right) -\Phi\left(\{0\}\times
		\tilde{\cD}(\psi)\right) +
		\Gamma_{\zs} -\Gamma_{\xs} -\Gamma_{\Theta_{\scb_0,\scb'_0}},
		\label{eq:3-chain-difference-triangle-maps}
		\end{equation}
		it follows that $\omega$ evaluates trivially on the sum of the 2-chains on the
		right-hand side of
		equation~\eqref{eq:3-chain-difference-triangle-maps}.	The quantities
		$\int_{\Phi(\{0\} \times \tilde{\cD}(\psi))} \omega$ and $\int_{\Phi(\{1\}
		\times \tilde{\cD}(\psi))} \omega$ are the area contributions of
		$\Psi_{\cH_0\to \cH_0'}(\xs)$ and $\Psi_{\cH_1\to \cH_1'}(v_*(\xs))$,
		respectively. The
		quantity $\int_{\Gamma_{\zs}} \omega$ is the area contribution of
        $F_{\cW;\omega, (f,v)}(\zs)$, and $\int_{\Gamma_{\xs}} \omega$ is the
		area contribution of $F_{\cW;\omega,(f,v)}(\xs)$. Hence
		\[
		F_{\cW;\omega,(f,v)}(\Psi_{\cH_0\to\cH_0'}(\xs))
		=z^{-\int_{\Gamma_{\Theta_{\scb_0,\scb'_0}}}\omega}\cdot
		\Psi_{\cH_1\to \cH_1'}(F_{\cW;\omega,(f,v)}(\xs)).
		\]
		Since $\int_{\Gamma_{\Theta_{\scb_0,\scb_0'}}}\omega$ is independent of $\xs$ and
		$\zs$, the result follows.	
	
	We now consider claim~\eqref{lem:no-crit-points-3}, independence from the
gradient-like vector field. Any two $v$ may be connected by a 1-parameter family
$(v_t)_{t\in I}$. As before, let $\cH_0=(\Sigma_0,\as_0,\bs_0)$ denote a diagram
for $(M_0,\g_0)$. For $t\in I$, let $\Phi_t\colon I\times M_0\to W$ denote the
flow of $v_t/v_t(f)$.
	
	Write $\phi_t\colon M_1\to M_1$ for the diffeomorphism $(\Phi_t\circ
	\Phi_0^{-1})|_{M_1}$.
	Claim~\eqref{lem:no-crit-points-3} amounts to showing
	\begin{equation}
	F_{\cW;\omega,(f,v_1)}\doteq (\phi_t)_* \circ F_{\cW;\omega,(f,v_0)},
	\label{eq:perturbed-isotopy-cylinder}
	\end{equation}
	where $(\phi_t)_*$ denotes the isotopy map from
	Section~\ref{sec:perturbed-diffeos}.

	Let $\Gamma_{\xs,t}$ denote the 2-chain
	$\Phi_t(I\times \gamma_{\xs})\subset W$, and let $\Gamma_{\xs}'\subset M_1$
	denote the 2-chain swept out by $\Phi_t(\{1\}\times \gamma_{\xs})$ as $t$ ranges
	over $I$. Equation~\eqref{eq:perturbed-isotopy-cylinder} amounts to showing that
		\begin{equation}
		\int_{\Gamma_{\xs,1}}\omega -\int_{\Gamma_{\xs,0}}\omega-\int_{\Gamma'_{\xs}}
		\omega\label{eq:isotopy-gradient-integral}
		\end{equation}
		is independent of $\xs$.
		
		Write $\hat{\Phi}\colon I\times I\times M_0\to W$ for the map
		$\hat{\Phi}(t,s,x)=\Phi_t(s,x)$.
		Let $C_3$ be the 3-chain defined  by applying $\hat{\Phi}$ to $I\times I\times
		\gamma_{\xs}$.
		Equation~\eqref{eq:isotopy-gradient-integral} is equal to
		$\int_{\d (I\times I)\times \gamma_{\xs}} \hat{\Phi}^*(\omega)$. Since
		$\int_{C_3} d \omega=0$, Stokes' theorem implies that
		equation~\eqref{eq:isotopy-gradient-integral} is equal to
		$\int_{I\times I\times \d \gamma_{\xs}} \hat{\Phi}^*(\omega).$ Since $\d
		\gamma_{\xs}$ is independent of $\xs$, it follows that the quantity in
		equation~\eqref{eq:isotopy-gradient-integral} is also independent of $\xs$,
		completing the proof.
	\end{proof}
	
    We are now ready to prove Lemma~\ref{lem:Alexander-grading}.

    \begin{proof}[Proof of Lemma~\ref{lem:Alexander-grading}]
        By construction, the map $F_{I\times Y(K);\omega_{\cS_K}}$ sends
		$z^x\cdot \xs$ to $z^{x+\int_{\Gamma_{\xs}} \omega_{S_K}}\cdot \xs$, where
		$\Gamma_{\xs}=I\times \gamma_{\xs}$. It is sufficient to show that
		\begin{equation}
		\int_{\Gamma_{\xs}-\Gamma_{\ys}}
		\omega_{\cS_K} =
		-A(\xs,\ys),\label{eq:desired-Alexander-grading}
		\end{equation}
		where $A(\xs,\ys)$ is the relative Alexander grading.
		
		Since $\omega_{\cS_K}$ is the Poincar\'{e}--Lefschetz dual of $\{\tfrac{1}{2}\}\times \cS_K$, we have
		\begin{equation*}
		\int_{\Gamma_{\xs}-\Gamma_{\ys}}
		\omega_{\cS_K} =\# (\gamma_{\xs}-\gamma_{\ys})\cap \cS_K.
		\end{equation*}
		If $\phi\in \pi_2(\xs,\ys)$ is a class of disks, then, by definition,
		\[
		A(\xs,\ys)=n_{z}(\phi)-n_w(\phi).		
		\]
		On the other hand,
		\[
		\d \tilde{\cD}(\phi)=\gamma_{\ys}-\gamma_{\xs}.		
		\]
		Using the Leibniz rule for intersections, we have
		\begin{equation}
		\begin{split}
		\# (\gamma_{\xs}-\gamma_{\ys})\cap \cS_{K}&=-\# (\d \tilde{\cD}(\phi) \cap \cS_K)\\
		&=-\#(\tilde{\cD}(\phi)\cap \d \cS_K).
		\end{split}
		\label{eq:integral=alexander-grading-1}
		\end{equation}
		Since $\d \cS_K=K$, equation~\eqref{eq:integral=alexander-grading-1} gives
		\[
	   \#(\gamma_{\xs}-\gamma_{\ys})\cap \cS_K=-\#( \tilde{\cD}(\phi)\cap K),		
		\]
		which is $-(n_z(\phi)-n_w(\phi))=-A(\xs,\ys)$, because, by convention, $K$ intersects $\Sigma$
		positively at $z$ and negatively at~$w$.
	\end{proof}	
	
	\begin{rem}\label{rem:failure-transitive-system-over-2-forms}
        In Lemma~\ref{lem:independence-rep-omega-spinc},  we described a
		transition map $\Psi_{\omega\to \omega';\eta}$
		for changing between cohomologous closed 2-forms $\omega$ and $\omega'$ when
		$d\eta=\omega'-\omega$,
		though the map was only independent of $\eta$ when restricted to a fixed
        $\Spin^c$ structure.
		Lemma~\ref{lem:Alexander-grading} is a perfect example of why this is important.
		The 2-form $\omega_{S_K}$ is a boundary on $I\times Y(K)$. Write $\omega_{\cS_K} = d \eta$,
		and write $\eta_i := \eta|_{\{i\}\times Y(K)}$. Note that $\omega_{\cS_K}$
        restricts trivially to $\{i\}\times Y(K)$ for $i\in \{0,1\}$.
        An easy Stokes' theorem argument shows that the diagram
		\begin{equation}
		\begin{tikzcd}[column sep=1.8 cm]\SFH(Y(K);\Lambda_{0})\arrow[d, "F_{I\times
		Y(K); \omega_{S_K}}"] \arrow[r,"\Psi_{0\to d \eta_0;\eta_0}"]&
		\SFH(Y(K);\Lambda_{d
		\eta_0})\arrow[d, "F_{I\times Y(K);0}"]\\
		\SFH(Y(K);\Lambda_{0})& \SFH(Y(K);\Lambda_{d \eta_1})\arrow[l,"\Psi_{d
		\eta_1\to 0;-\eta_1}"]
		\end{tikzcd}
		\label{eq:erroneous-conclusion}
		\end{equation}
		commutes up to an overall factor of $z^x$. Hence $F_{I\times
        Y(K);\omega_{\cS_K}}\doteq \Psi_{0\to 0; \eta_0-\eta_1}$, but this does not
        imply that $F_{I\times Y(K);\omega_{S_K}}\doteq \id,$ since
        Lemma~\ref{lem:independence-rep-omega-spinc} only applies if we restrict to a
        single $\Spin^c$ structure.
	\end{rem}

	\subsection{Perturbed 1-handle and 3-handle maps}
    \label{sec:1-handles}
	
	We now describe the cobordism maps for 1-handles and 3-handles.
    We focus on 1-handles, since the 3-handle maps are algebraically dual.
	
	Suppose that	
	\[
	\cW_1=(W_1,Z_1,[\xi_1])\colon (M_0,\g_0)\to (M_1,\g_1)
	\]
	 is a special cobordism with a Morse function $f$ that
	has a single index 1 critical point $p_0$. Let $v$ be a gradient-like vector
    field for $f$. We use $f$ and $v$ as auxiliary data to construct the cobordism map
    for  $\cW_1$.

	The stable manifold of $v$ at $p_0$ intersects $M_0$ in two points, $p_1$ and
	$p_2$. Let $\cH_0=(\Sigma_0,\as_0,\bs_0)$ be an admissible diagram for $(M_0,\g_0)$, such that
	$p_1$, $p_2\in \Sigma_0 \setminus (\as_0 \cup \bs_0)$. Let $D_1$ and $D_2$ be two small
	disks in $\Sigma_0$, centered at $p_1$ and $p_2$. The flow of $v$ induces an
	embedding of $\Sigma_0 \setminus (D_1 \cup D_2)$ into $M_1$.
	
	A Heegaard diagram $(\Sigma_1,\as_1,\bs_1)$ for $(M_1,\g_1)$
	is constructed as follows. The surface $\Sigma_1$ is obtained by connecting
    the boundary components of
	the image of $\Sigma_0 \setminus (D_1\cup D_2)$ under the flow of $v$ with an
    annulus in the 1-handle
	region. The attaching curves $\as_1$ and $\bs_1$ are given by $\as_1\cup \{\a\}$
	and $\bs_1\cup \{\b\}$,
	where $\a$ and $\b$ are contained in the 1-handle annulus, intersect
	transversely,  are
	homologically essential therein, and satisfy $|\a\cap\b|=2$. Write $\a\cap
	\b=\{\theta^+,\theta^-\}$, where $\theta^+$ has the larger relative Maslov grading.
	
	If $\xs\in \bT_{\sca_0}\cap \bT_{\scb_0}$, write $v_*(\xs)$ for the corresponding
	tuple of points on $\Sigma_1$. A set of compressing disks in $M_0$ may be
	pushed forward under the flow of $v$. By adding two disks in the 1-handle region,
	we naturally obtain a set of compressing disks in $(M_1,\g_1)$.
	If $\xs\in \bT_{\sca_0}\cap
	\bT_{\scb_0}$,
	write $\Gamma_{\xs}\subset W_1$ for the 2-chain traced out by applying the flow of
	$v$ to $\gamma_{\xs}\subset M_0$. We define the perturbed 1-handle map $F_{\cW_1;\omega, (f,v)}$ as
	\[
	F_{\cW_1;\omega, (f,v)}(z^x\cdot \xs):=z^{x+\int_{\Gamma_{\xs}} \omega}\cdot
	v_*(\xs)\times \theta^+.
	\]
	
\begin{lem} \label{lem:1-handles}
Suppose that $\cW_1=(W_1,Z_1,[\xi_1])\colon (M_0,\g_0)\to (M_1,\g_1)$ is a
special cobordism and $(f,v)$ is a Morse function and gradient-like vector field
on $W_1$ with a single index 1 critical point.
\begin{enumerate}
\item\label{lem:1-handle-claim-1} For an almost complex structure sufficiently
stretched on the two boundary components of the 1-handle annulus, the map
$F_{\cW_1;\omega,(f,v)}$ is a chain map.
\item \label{lem:1-handle-claim-2}The morphism of transitive systems induced by
$F_{\cW_1;\omega,(f,v)}$ is independent of the Heegaard diagram for
$(M_0,\g_0)$.
\item \label{lem:1-handle-claim-3}The morphism of transitive systems induced by
$F_{\cW_1;\omega,(f,v)}$ is independent of $v$.
\end{enumerate}
\end{lem}
	
	\begin{proof}
	    The proof of claim~\eqref{lem:1-handle-claim-1}, that $F_{\cW_1;\omega,(f,v)}$
        is a chain map, relies on the same holomorphic degeneration
        argument used in the unperturbed setting. See
        \cite{OSTriangles}*{Section~4.3} for the original proof, as well as
		\cite{JCob}*{Section~7}, or \cite{ZemGraphTQFT}*{Section~8} for versions of the proof
		in several related contexts.
		In the perturbed setting, one must also check that
		the cones of the two bigons in the 1-handle region are assigned the same $\omega$-area.
		Note that the difference between these two bigon classes is a periodic
		domain, which cones off to a 2-sphere $S$ that is homotopic to the belt sphere
		of the 4-dimensional 1-handle. Since $\omega$ is defined on all of $W$ (in particular,
        on the co-core of the 1-handle), we must have $\int_{S}\omega=0$.

        To prove claim~\eqref{lem:1-handle-claim-2}, that the morphism of
        transitive systems induced by $F_{\cW_1,\omega; (f,v)}$ is independent of the
        Heegaard diagram $\cH_0$, one repeats the standard proof of the well-definedness
        of the 1-handle maps \cite{OSTriangles}*{Theorem~4.10}, while keeping track of
        areas as in the proof of Lemma~\ref{lem:cylinder-maps-well-defined}.

        Claim~\eqref{lem:1-handle-claim-3}, independence from $v$, is proven as follows.
		Suppose that $(v_t)_{t\in I}$ is a path of gradient-like vector fields.
		We can pick an isotopy $\phi_t$ of $M_0$, and an admissible diagram
		$(\Sigma_0,\as_0,\bs_0)$ for $(M_0,\g_0)$ such that the stable manifold of the
	    critical point of $f$ is contained in $\phi_t(\Sigma_0\setminus (\as_0\cup \bs_0))$
		for all $t$. We can choose an isotopy $\psi_t$ of $(M_1,\g_1)$ such that the
		image of $\phi_t(\Sigma_0)$ under the flow of $v_t$ coincides with $\psi_t(\Sigma_1)$
        outside the 1-handle region. Write $(\Sigma_0',\as_0',\bs_0')$ for the
        image of $(\Sigma_0,\as_0,\bs_0)$ under $\phi_1$, and write
        $(\Sigma_1',\as_1',\bs_1')$ for the image of $(\Sigma_1,\as_1,\bs_1)$ under $\psi_t$.
		
		It suffices to show that the following diagram commutes, up to overall
		multiplication by $z^x$:
		\begin{equation}
		\begin{tikzcd}\CF(\Sigma_0,\as_0,\bs_0;\Lambda_{\omega|_{M_0}})\arrow[r, "(\phi_t)_*"]\arrow[d,
		"F_{\cW_1;\omega,(f,v_0)}"] &
        \CF(\Sigma_0',\as_0',\bs_0';\Lambda_{\omega|_{M_0}})\arrow[d,
		"F_{\cW_1;\omega,(f,v_1)}"]\\
		\CF(\Sigma_1,\as_1,\bs_1;\Lambda_{\omega|_{M_1}})\arrow[r, "(\psi_t)_*"] & \CF(\Sigma_1',
		\as_1',\bs_1';\Lambda_{\omega|_{M_1}}).
		\end{tikzcd}
		\label{eq:commutative-diagram-1-handle-vector-field}
		\end{equation}

    We  define
    \[
    \hat{\Phi}\colon I\times I\times \gamma_{\xs}\to W_1,
    \]
    where $\hat{\Phi}(t,s, x)$ is the time $s$ flow of $\phi_t(x)$ under $v_t/v_t(f)$.
    Consider the 3-chain $C_3 = \Phi(I \times I \times \gamma_{\xs})$ in $W_1$.
    Then we have
    \begin{equation}
	\d C_3= \hat{\Phi}(\d (I\times I)\times \gamma_{\xs})+\hat{\Phi}(I\times
    I\times \d \gamma_{\xs}). \label{eq:1-handle-C3}
	\end{equation}
	
    Write $\gamma_{\theta^+}\subset M_1$ for 1-chain obtained by coning
    $\theta^+$ into the two handlebodies,
    and let $\Gamma_{\theta^+,\psi_t}\subset M_1$ denote the 2-chain
    swept out by the family $(\psi_t(\gamma_{\theta_+}))_{t\in I}$.
	By definition, the difference in area contributions from the two length 2 paths
    in equation~\eqref{eq:commutative-diagram-1-handle-vector-field} is
	\begin{equation}
    \int_{\d(I\times I)\times \gamma_{\xs}} \hat{\Phi}^*
    (\omega)+\int_{\Gamma_{\theta^+,\psi_t}}
    \omega.\label{eq:1-handle-diff-across-paths}
	\end{equation}
    Applying Stokes' theorem to equation~\eqref{eq:1-handle-C3}, we see that
    equation~\eqref{eq:1-handle-diff-across-paths} is equal to
    \[
    -\int_{I\times I\times \d \gamma_{\xs}}
    \hat{\Phi}^*(\omega)+\int_{\Gamma_{\theta^+,\psi_t}} \omega,
    \]
    which is independent of $\xs$. It follows that
    equation~\eqref{eq:commutative-diagram-1-handle-vector-field} commutes up to an
    overall factor of $z^x$, completing the proof.
\end{proof}
		
    The perturbed 3-handle maps are dual to the 1-handle maps.
    We leave the details of the definition to the reader.

	\subsection{Perturbed 2-handle maps}
	\label{sec:2-handle}
	Suppose that
	\[
	\cW_2=(W_2,Z_2,[\xi_2])\colon (M_0,\g_0)\to (M_1,\g_1)
	\]
	 is a special cobordism equipped
	with a Morse function $f$ and gradient-like vector field $v$ such that
	$f$ has only index 2 critical points, and the stable and unstable manifolds
	of $v$ are transverse.
		
	Let $\bS_1\subset M_0$ denote the intersection of the stable manifolds of $(f,v)$
	and $M_0$. Let $(\Sigma,\as,\bs,\bs')$ be a Heegaard triple subordinate to a
	bouquet for $\bS_1$; see \cite{JCob}*{Definition~6.3}. Let
    \[
    \cW_{\sca,\scb,\scb'} = (W_{\sca,\scb,\scb'}, Z_{\sca,\scb,\scb'},
    [\xi_{\sca,\scb,\scb'}])
    \]
    be the associated sutured manifold cobordism, as described in
    \cite[Section~7]{JuhaszZemkeContactHandles}.
    The 4-manifold $W_{\sca,\scb,\scb'}$ is defined as follows. If $\Delta$ denotes a triangle
    with edges $e_{\a}$, $e_{\b}$, and $e_{\b'}$, then
    \[
	W_{\a,\b,\b'}:=(\Sigma\times \Delta)\cup
	\left(U_{\a}\times e_{\a}\cup U_{\b}\times e_{\b}\cup U_{\b'}\times e_{\b'}\right),
    \]
    where $U_{\sca}$, $U_{\scb}$, and $U_{\scb'}$ are the sutured compression bodies
    corresponding to $(\S, \as)$, $(\S, \bs)$, and $(\S, \bs')$, respectively.
    We view the 4-manifold $W_{\a,\b,\b'}$ as
    having three sutured manifold boundary components,
    $M_0$, $M_{\b,\b'}$, and $M_1$.

	From our choice of $(f,v)$, we obtain an embedding
	\[
	\Phi_{(f,v)}\colon W_{\sca,\scb,\scb'}\to W_2,
	\]
	which is well-defined up to isotopy, as follows.
	 Let $\{b_1,\dots, b_k\}\subset (0,1)$ be
	the critical values of $f$, and let $\epsilon>0$ be chosen such that
	$\epsilon<b_i<1-\epsilon$ for $i \in \{1, \dots, k\}$. Let $N(\Sigma)$ be a product neighborhood
    of $\Sigma$ in $M_0$. We can view $M_0$ as $
    U_{\sca} \cup N(\Sigma)\cup -U_{\scb}$.
    We can correspondingly view $W_{\sca,\scb,\scb'}$ as
	\[
	(N(\Sigma)\times I)\cup (U_{\sca}\times I)\cup (-U_{\scb}\times
	[0,\epsilon])\cup (-U_{\scb'}\times [1-\epsilon,1]).
	\]	
	The embedding $\Phi_{(f,v)}$ sends a point $(x,t)\in U_{\scb}\times
    [0,\epsilon]$ to the
	point $z\in W_2$ which is in the flow line of $v$ over $x\in U_{\scb}\subset M$
	and has $f(z)=t$. The embeddings on the other portions of $W_{\sca,\scb,\scb'}$ are
	defined similarly. See Figure~\ref{fig:triple} for a schematic. We note also
	that the boundary component
    $M_{\b,\b'}\subset \d \cW_{\a,\b,\b'}$ may be naturally filled in with a
     sutured manifold cobordism $\cW_{\b,\b'}$
    to obtain the sutured 2-handle cobordism $\cW_2$.
    A description of $\cW_{\b,\b'}$ may be found in
    \cite{JCob}*{Proposition~6.6} (see also \cite{JuhaszZemkeContactHandles}*{Section~8}).

        \begin{figure}[ht!]
        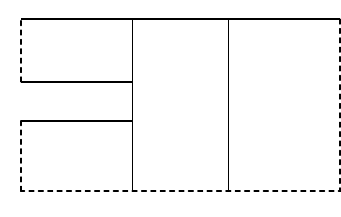
        \caption{The triple cobordism $W_{\a,\b,\b'}$.}
        \label{fig:triple}
        \end{figure}	
	
	A homology class $\psi\in \pi_2(\xs,\ys,\zs)$ on
	$(\Sigma,\as,\bs,\bs')$ induces a coned off singular 2-chain $\tilde{\cD}(\psi)$
	in $W_{\sca,\scb,\scb'}$, as follows.
	Firstly, the class $\psi$ induces a singular 2-chain $\cD_0(\psi)$ in $\Sigma\times \Delta$,
	which has boundary on
	$(\as\times e_{\sca})\cup (\bs\times e_{\scb})\cup (\bs'\times e_{\scb'})$,
	where $\d \Delta=e_{\sca}\cup e_{\scb}\cup e_{\scb'}$. The 2-chain $\cD_0(\psi)$ is
	determined, up to addition of a boundary,
	by the property its projection to $\Sigma$ is the domain of $\psi$, and that the projection onto
	$\Delta$ is degree $d$, where $d=|\as|=|\bs|=|\bs'|$.
	 We pick compressing
	disks $D_{\sca}$, $D_{\scb}$, and $D_{\scb'}$, and we let $c_{\sca}$, $c_{\scb}$, and $c_{\scb'}$
	denote the sets of center points of these compressing disks, respectively.
	We cone $\cD_0(\psi)$ into $U_{\sca}\times e_{\sca}$, $U_{\scb}\times e_{\scb}$,
	and $U_{\scb'}\times e_{\scb'}$ to obtain a 2-chain
	$\tilde{\cD}(\psi)$ in $W_{\sca,\scb,\scb'}$ that has boundary
	\[
	-\gamma_{\xs} -	\gamma_{\ys} + \gamma_{\zs}	+ c_{\sca} \times e_{\sca}
    + c_{\scb} \times e_{\scb} + c_{\scb'} \times e_{\scb'}.
	\]
	
	We define $A_\omega(\psi)$ to be the integral of
	$\Phi_{(f,v)}^* (\omega)$ over $\tilde{\cD}(\psi)$. We write $(M_{\scb, \scb'}, \gamma_{\scb, \scb'})$
    for the sutured manifold defined by the diagram $(\Sigma, \bs, \bs')$,
    and $\omega_{\scb, \scb'} = \omega|_{M_{\scb, \scb'}}$.
		
	By counting index 0 holomorphic triangles weighted with $z^{A_{\omega}(\psi)}$,
	we obtain a perturbed triangle map
	\begin{equation}
	F_{\sca,\scb,\scb';\omega}\colon
    \CF(\Sigma,\as,\bs;\Lambda_{\omega|_{M_0}})\otimes
	\CF(\Sigma,\bs,\bs';\Lambda_{\omega_{\scb,\scb'}})\to
	\CF(\Sigma,\as,\bs';\Lambda_{\omega|_{M_1}}).
	\label{eq:perturbed-2-handle-map}
	\end{equation}	
	Finally, the perturbed 2-handle map is given by the formula
	\begin{equation}\label{eq:perturbed-2-handle-map-2}
	F_{\cW_2;\omega, (f,v)}(z^x\cdot \xs) = z^x \cdot F_{\a,\b,\b';\omega}\left(\xs \otimes
	\Theta_{\scb,\scb'}^{\omega_{\scb,\scb'}}\right),
	\end{equation}
    where $\Theta_{\scb,\scb'}^{\omega_{\scb,\scb'}} \in
    \CF(\Sigma,\bs,\bs';\Lambda_{\omega_{\scb,\scb'}})$
    is defined analogously to equation~\eqref{eq:top-degree-generator}.

	The domain and codomain of $F_{\cW_2;\omega, (f,v)}$ do not form projective transitive
	systems unless we either restrict to a single $\Spin^c$ structure on $(M_0,\g_0)$
	and $(M_1,\g_1)$, or if $[\omega]|_{M_i}=0$ for $i\in \{0,1\}$.	However, if we
    fix $\ufrs_0\in \Spin^c(M_0,\g_0)$ and $\ufrs_1\in \Spin^c(M_1,\g_1)$, we obtain
    a morphism of projective transitive systems
	\[
	\pi_{\ufrs_1}\circ F_{\cW_2;\omega, (f,v)}\circ i_{\ufrs_0}\colon
    \CF(M_0,\g_0,\ufrs_0;\Lambda_{\omega|_{M_0}})\to
    \CF(M_1,\g_1;\ufrs_1,\Lambda_{\omega|_{M_1}}).
	\]
		
	\begin{lem}\label{lem:2-handles}
    Suppose that $\cW_2\colon (M_0,\g_0)\to
    (M_1,\g_1)$ is a special cobordism with a Morse function $f$ and gradient-like
    vector field $v$ with only index 2 critical points, which is Morse--Smale. Let
    $\bS_1$ denote the corresponding framed link in  $M_0$.
	\begin{enumerate}
	\item\label{lem:2-handles-claim-1} The morphism of transitive systems induced
    by $F_{\cW_2;\omega,(f,v)}$ is independent of the choice of bouquet for $\bS_1$,
    or the Heegaard triple subordinate to it.
	\item\label{lem:2-handles-claim-2} The morphism of transitive systems
    $F_{\cW_2;\omega,(f,v)}$ is independent of $v$.
	\end{enumerate}
	\end{lem}		

	\begin{proof}
    The proof of claim~\eqref{lem:2-handles-claim-1} is similar to
	the original proof given by Ozsv\'{a}th and Szab\'{o}
	\cite{OSTriangles}*{Proposition~4.6, Lemma~4.8}, and follows from associativity
	of the perturbed holomorphic triangle maps. See also \cite{JCob}*{Theorem~6.9}
	for a more detailed explanation of the argument in the sutured setting.
	
	Independence from $v$, claim~\eqref{lem:2-handles-claim-2}, is proven as
	follows. The space of gradient-like
	vector fields of $f$ is connected. Suppose $(v_t)_{t\in I}$ is a path of
	gradient-like vector fields. Let $\bS_1^t$ denote the intersection of the
	stable manifolds of $v_t$ with $M_0$. Generically, $v_t$ is
	Morse--Smale at all but finitely many $t$, at which time a handleslide amongst two of the
	components of $\bS_1^t$ occurs.
	
	We break $I$ into two types of subintervals: $[a,b]$, where $(f,v_t)$ is
	Morse--Smale for all $t\in [a,b]$; and $[t_0-\epsilon,t_0+\epsilon]$, where
	$\epsilon>0$ is small, and a handleslide occurs at $t_0$.
	
    For the first type of subinverval $[a,b]$, let $(\Sigma,\as,\bs,\bs')$ be subordinate
    to a bouquet for $\bS_1^a$. Let $(\phi_t)_{t\in [a,b]}$ be an isotopy of $M_0$,
    such that $\phi_a=\id_{M_0}$, and the diagram
    \[
    (\Sigma_t,\as_t,\bs_t,\bs'_t) := \phi_t(\Sigma,\as,\bs,\bs') \subset M_0
    \]
    is subordinate to $\bS_1^t$.

    Using the abbreviation $\Phi_t$ for $\Phi_{(f,v_t)}$, we obtain a family of
    embeddings $(\Phi_t)_{t\in [a,b]}$ of $W_{\a,\b,\b'}$ into $W_2$.
    Let $\psi_t\colon M_1\to M_1$ denote the map $(\Phi_t\circ
    \Phi_a^{-1})|_{M_1}$. We claim that the following diagram commutes up to an
    overall factor of $z^x$:
	\begin{equation}
	\begin{tikzcd}
	\CF(\Sigma_a,\as_a,\bs_a;\Lambda_{\omega|_{M_0}})\arrow[r,"(\phi_t)_*"]
	\arrow[d, "F_{\cW_2;\omega,(f,v_a)}"]
	& \CF(\Sigma_b,\as_b,\bs_b;\Lambda_{\omega|_{M_0}})
	\arrow[d,"F_{\cW_2;\omega, (f,v_b)}"]\\
	\CF(\Sigma_a,\as_a,\bs_b';\Lambda_{\omega|_{M_1}})\arrow[r,"(\psi_t)_*"]&
	\CF(\Sigma_b,\as_b,\bs_b';\Lambda_{\omega|_{M_1}}).
	\end{tikzcd}
	\label{eq:2-handle-isotopy-diagram}
	\end{equation}
	Suppose $\psi\in \pi_2(\xs,\Theta_{\scb,\scb'},\zs)$ is a homology class of
    triangles on $(\Sigma,\as,\bs,\bs')$, where $\xs \in \T_{\sca} \cap \T_{\scb}$
    and $\zs \in \T_{\sca} \cap \T_{\scb'}$.
    Write $\Gamma_{\xs,\phi_t}\subset M_0$ and $\Gamma_{\zs,\psi_t}\subset M_1$ for the
    2-chains swept out by $\gamma_{\xs}$ and $\gamma_{\zs}$ by $\phi_t$ and
    $\psi_t$ for $t \in [a,b]$, respectively. Commutativity of
    equation~\eqref{eq:2-handle-isotopy-diagram} up to an overall factor of $z^x$
    amounts to showing that the integral of $\omega$ over
	\begin{equation}\label{eqn:C3}
	\Phi_a(\tilde{\cD}(\psi))-\Phi_b(\tilde{\cD}(\psi))+\Gamma_{\zs,\psi_t}-\Gamma_{\xs,\phi_t}
	\end{equation}
	is independent of $\psi$, $\xs$, and $\zs$.
	
	The family $\Phi_t$ induces a map $\hat{\Phi}\colon [a,b]\times
    W_{\sca,\scb,\scb'}\to W_2$, and we let $C_3\subset W_2$ be the 3-chain
    $\hat{\Phi}([a,b]\times \tilde{\cD}(\psi))$.	
	Stokes' theorem applied to $\d C_3$ implies that the integral of $\omega$ over
    the 2-chain in equation~\eqref{eqn:C3} is equal to the integral of $\omega$ over
	\begin{equation}
	\Gamma_{\Theta_{\scb,\scb'},\Phi_t}+C_{\sca,\scb,\scb'},	
    \label{eq:2-handle-ind-of-xzpsi}
	\end{equation}
	where $\Gamma_{\Theta_{\scb,\scb'},\Phi_t}$ is the 2-chain
    $\hat{\Phi}([a,b]\times \gamma_{\Theta_{\scb,\scb'}})$, and $C_{\sca,\scb,\scb'}$
    is defined as follows. Let $c_{\sca}\subset U_{\sca}$ be the union of the
    centers of the alpha compressing disks, and let $e_{\sca}$ denote the alpha side
    of the triangle $\Delta$ used to build $W_{\sca,\scb,\scb'}$. Let $c_{\scb}$,
    $c_{\scb'}$, $e_{\scb}$, and $e_{\scb'}$ be defined similarly. Then
    $C_{\sca,\scb,\scb'}$ is the image under $\hat{\Phi}$ of $[a,b]\times
    (c_{\sca}\times e_{\sca}\cup c_{\scb}\times e_{\scb}\cup c_{\scb'}\times
    e_{\scb'})$.
    Since equation~\eqref{eq:2-handle-ind-of-xzpsi} is independent of $\xs$, $\zs$,
    and $\psi$, it follows that equation~\eqref{eq:2-handle-isotopy-diagram}
    commutes up to an overall factor of $z^x$.	
		
    Next, we consider the case when the subinterval of $I$ is of the form
    $[t_0-\epsilon,t_0+\epsilon]$, where a handleslide amongst the components of
    $\bS^t_1$ occurs at $t=t_0$. Adapting the proof of Ozsv\'{a}th
	and Szab\'{o} \cite{OSTriangles}*{Lemma~4.14}, we may pick a Heegaard triple
	$(\Sigma,\as,\bs,\bs')$ subordinate to a bouquet for
	$\bS_{1}^{t_0-\epsilon}$, such that there are attaching curves $\bar{\bs}$ and
	$\bar{\bs}'$ on $\Sigma$, where $\bar{\bs}$ is obtained from $\bs$
    and $\bar{\bs}'$ is obtained from $\bs'$ via a
	sequence of handleslides and isotopies, and
	$(\Sigma,\as,\bar{\bs},\bar{\bs}')$ is subordinate to a bouquet for
	$\bS_1^{t_0+\epsilon}$. The 4-manifold $W_{\sca,\scb,\scb'}$ is unchanged by
	isotopies and handleslides of the attaching curves.
	A straightforward associativity argument
	shows that the two morphisms constructed with the embedding $\Phi_{t_0-\epsilon}$
    and either of the triples $(\Sigma,\as,\bs,\bs')$ or $(\Sigma,\as,\bar{\bs},\bar{\bs}')$
    coincide. Similarly, the previous argument shows that the two morphisms computed using the triple
    $(\Sigma,\as,\bar{\bs},\bar{\bs}')$ and either of the embeddings $\Phi_{t_0-\epsilon}$ or
    $\Phi_{t_0+\epsilon}$ coincide, completing the proof.
	\end{proof}

	\subsection{Defining the \texorpdfstring{$\Spin^c$}{Spinc} restricted cobordism maps}
	\label{sec:def-spinc-restricted}
	
	In this section, we define the $\Spin^c$ restricted versions of the perturbed
    sutured cobordism maps.
	Suppose that
	\[
	\cW=(W,Z,[\xi])\colon (M_0,\g_0)\to (M_1,\g_1)
	\]
	 is a cobordism of
	sutured manifolds equipped
    with a closed 2-form $\omega$ on $W$. We remove a collection of tight
	3-balls from $Z$, adding them to $M_0$ or $M_1$, so that $M_0\cup Z$ has
    no closed components, and so that each component of $W$ intersects $M_0$ and
    $M_1$ non-trivially.

	We can decompose $\cW$ as $\cW^{s}\circ \cW^{\d}$, where $\cW^{\d}$ consists of
    $I\times (M_0\cup Z)$,
	viewed as a cobordism from $M_0$ to $M_0\cup Z$, and $\cW^s$ consists of $W$,
    viewed as a special cobordism from $M_0\cup Z$ to $M_1$.
	
	We choose a self-indexing Morse function $f$ on $\cW^s$, with no
    index 0 and 4 critical points, and a gradient-like vector field $v$ for $f$.
    The pair $(f, v)$ induces a decomposition
	\[
	 \cW^s = \cW_3 \circ \cW_2 \circ \cW_1,
	\]
	where $\cW_i=(W_i,Z_i,[\xi_i])$ is a special cobordism that contains the index
    $i$ critical points of $f$.

	Suppose $\ufrs_0\in \Spin^c(M_0, \gamma_0)$ and $\ufrs_1\in \Spin^c(M_1,\gamma_1)$. The
	$\Spin^c$ structure $\ufrs_1$ extends uniquely over $\cW_3$. Write
	$\underline{\fru}$ for its restriction to the incoming boundary of $\cW_3$. We
	define
	\begin{equation}
	\pi_{\ufrs_1}\circ F_{\cW;\omega} \circ
    i_{\ufrs_0}:=F_{\cW_3;\omega|_{W_3}}\circ
	\pi_{\underline{\fru}}\circ  F_{\cW_2;\omega|_{W_2}}\circ
	F_{\cW_1;\omega|_{W_1}}\circ \Phi_{\xi;\omega|_{M_0\cup Z}}\circ i_{\ufrs_0}
    \label{eq:spinc-restricted-cobordism-definition}
	\end{equation}
	where we have suppressed the dependence of the  map $F_{\cW_i;\omega|_{W_i}}$
    on the Morse function $f|_{W_i}$. There is no dependence on the gradient-like vector
    field $v|_{W_i}$ according to Lemmas~\ref{lem:cylinder-maps-well-defined}, ~\ref{lem:1-handles},
    and~\ref{lem:2-handles}.
	
    We now prove the $\Spin^c$ restricted perturbed cobordism maps are well-defined:	
	
    \begin{proof}[Proof of Part~\eqref{prop:well-def:restricted}
     of Proposition~\ref{prop:perturbedcobordismmapswelldefined}]
    The proof is similar to the proof
    of the corresponding claim in the unperturbed setting; see
    \cite{OSTriangles}*{Section~4.4} and
	\cite{JCob}*{Theorem~8.2}. Given two Morse functions $f_0$ and $f_1$ on $W$,
	viewed as a special cobordism from $M_0\cup Z$ to $M_1$,
    one may pick a generic path $(f_t)_{t\in I}$ of smooth functions
	that are Morse at all but finitely many $t$ and connect $f_0$ to $f_1$.
    Furthermore, using Cerf theory, one may assume that there are no index 0 or 4 critical
    points, and that critical
	points of index $i$ for $i \in \{2,3\}$ have values greater than the values of
    critical points of
	index less than $i$. Furthermore, at the finitely many $t$
	where $f_t$ fails to be Morse, an index 1/2 or 2/3 birth-death singularity occurs.
	
	If $f_t$ is Morse for every $t \in [a, b] \subset [0, 1]$, the decompositions of $\cW^s$
	as $\cW_1 \circ \cW_2 \circ \cW_3$ corresponding to $f_a$ and $f_b$ are isotopic, so
    adaptations of Lemmas~\ref{lem:1-handles} and \ref{lem:2-handles} show
	that the composition is unchanged, up to an overall factor of $z^x$.
	
	Invariance under index 1/2 birth-death follows from Ozsv\'{a}th and
	Szab\'{o}'s holomorphic triangle computation \cite{OSTriangles}*{Lemma~4.16},
	with extra attention paid to areas. Invariance under index 2/3 birth-deaths
	follows by the same argument.	
    \end{proof}

\subsection{Defining the total cobordism map}
\label{sec:def-general-cob}

In this section, we define the total perturbed cobordism map $F_{\cW;\omega}$,
when $[\omega]$ restricts trivially to $M_0$ and $M_1$. This addresses
Part~\eqref{prop:well-def:total} of
Proposition~\ref{prop:perturbedcobordismmapswelldefined}.

As a first step, if $[\omega]$ restricts trivially to $M_1$, and $\ufrs_0\in
\Spin^c(M_0,\g_0)$, we may define the partially $\Spin^c$ restricted  map
$F_{\cW;\omega}\circ i_{\ufrs_0}$  by omitting $\pi_{\underline{\fru}}$ from
equation~\eqref{eq:spinc-restricted-cobordism-definition}.

This strategy does not extend to the case when $[\omega]|_{M_0}=0$, since we
also need $[\omega]|_{M_0\cup Z}=0$ for the gluing map to be well-defined.
Instead, when $[\omega]$ restricts trivially to $M_0$ and $M_1$, we make an alternate
construction. Pick an open collar
neighborhood $N\subset W$ of $M_0$. Set
\[
\cN = (\overline{N},Z \cap \overline{N}, [\xi|_{\overline{N}}]),
\]
which we view as a sutured manifold cobordism from $(M_0,\g_0)\cup (-M_0,\g_0)$ to the empty set.
Let us write
\[
\tilde{\cW} = (W\setminus N, Z\setminus N, [\xi|_{Z\setminus N}]).
\]
We view $\tilde{\cW}$ as a cobordism from the empty set
to $(-M_0,\g_0)\cup (M_1,\g_1)$. See Figure~\ref{fig:triple}.

The previous case gives a map
\begin{equation}
F_{\tilde{\cW};\omega|_{W\setminus N}}\colon \Lambda \to
\SFH(-M_0,\g_0;\Lambda_{\omega|_{M_0}})\otimes
\SFH(M_1,\g_1;\Lambda_{\omega|_{M_1}}).
\label{eq:dualized}
\end{equation}
Implicitly, we are pre-composing with the map $i_{\ufrs_0}$,
where $\ufrs_0$ is the unique $\Spin^c$ structure on the empty set.
We define the total cobordism map $F_{\cW;\omega}$ via the formula
\begin{equation}
F_{\cW;\omega}:=\left( F_{\cN, \omega|_{\overline{N}}}\otimes \id_{\SFH(M_1)} \right)\circ
 \left(\id_{\SFH(M_0)}\otimes
 F_{\tilde{\cW};\omega|_{W\setminus N}} \right).
\label{eq:alternate-def-cobordism-map}
\end{equation}

\begin{figure}[ht!]
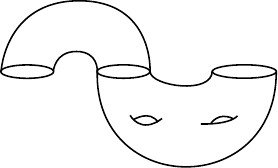
\caption{Decomposing $\cW$ into $\cN$ and $\tilde{\cW}$.}
\label{fig:triple}
\end{figure}

If $[\omega]$ restricts trivially to $M_0\cup Z$ and $M_1$,  then we may also
define the total perturbed cobordism map by removing the projections and
inclusions of $\Spin^c$ structures from
equation~\eqref{eq:spinc-restricted-cobordism-definition}. We claim
that this more direct construction coincides
with the construction given in equation~\eqref{eq:alternate-def-cobordism-map}.
To see this, we note that if $\cW=(W,Z,[\xi])$ is a sutured manifold
cobordism which decomposes as the composition of two cobordisms,
$\cW_1=(W_1,Z_1,[\xi_1])$ and
$\cW_2=(W_2,Z_2,[\xi_2])$, and $\omega$ is a 2-form such that
$[\omega|_{M_0\cup Z}] = 0$ and $[\omega|_{M_1}] = 0$,
then the original proof of the sutured cobordism
composition law
\cite[Theorem~11.3]{JCob} (see also \cite{OSTriangles}*{Theorem~3.4}), adapts to show that
\begin{equation}
F_{\cW;\omega}\doteq F_{\cW_2;\omega_2}\circ F_{\cW_1;\omega_1},
 \label{eq:composition-law-when-2-form-vanishes}
\end{equation}
where the maps in equation~\eqref{eq:composition-law-when-2-form-vanishes}
are defined using the construction in
equation~\eqref{eq:spinc-restricted-cobordism-definition}. When $[\omega]$ restricts trivially to $M_0\cup Z$ and $M_1$,  equation~\eqref{eq:alternate-def-cobordism-map} may be interpreted as a composition
satisfying these hypotheses, so the composition law of equation~\eqref{eq:composition-law-when-2-form-vanishes}
implies that equation~\eqref{eq:alternate-def-cobordism-map}
coincides with the construction obtained by removing the
$\Spin^c$ restrictions from equation~\eqref{eq:spinc-restricted-cobordism-definition}.

\subsection{The composition law}
\label{sec:composition-law}

We now sketch a proof of the composition law,
Proposition~\ref{prop:composition-law}.

\begin{proof}[Proof of Proposition~\ref{prop:composition-law}] We focus on
part~\eqref{prop:comp-law-total}, as part~\eqref{prop:comp-law-restricted}
follows from a simple modification.
Assume, as in the statement, that $\cW=(W,Z,[\xi])$ is a sutured cobordism from
$(M_0,\g_0)$ to $(M_2,\g_2)$, which decomposes into
$\cW_1=(W_1,Z_1,[\xi_1])$ and $\cW_2=(W_2,Z_2,[\xi_2])$ that meet along a
sutured manifold $(M_1,\g_1)$. We are interested in the case when
$[\omega]$ restricts trivially to $M_0$, $M_1$, and $M_2$.

As a first step, we claim that, via the same argument that gives
equation~\eqref{eq:composition-law-when-2-form-vanishes},
if $[\omega]$ restricts trivially to $M_1 \cup Z_2$ and $M_2$, then
\begin{equation}
F_{\cW;\omega} \circ i_{\ufrs}\doteq F_{\cW_2;\omega_2}\circ F_{\cW_1;\omega_1}\circ i_{\ufrs},
\label{eq:restricted-composition-law}
\end{equation}
where $\ufrs\in \Spin^c(M_0,\g_0)$, and the maps $F_{\cW;\omega} \circ i_{\ufrs}$,
$F_{\cW_2;\omega_2}$, and $F_{\cW_1;\omega_1}\circ i_{\ufrs}$ are defined
using the appropriate modification of
equation~\eqref{eq:spinc-restricted-cobordism-definition}.

We now claim that the restricted composition law stated in  equation~\eqref{eq:restricted-composition-law}
implies the full version of part~\eqref{prop:comp-law-total}
of  Proposition~\ref{prop:composition-law}. We recall
that the full version of Proposition~\ref{prop:composition-law}
involves the maps defined in equation~\eqref{eq:alternate-def-cobordism-map}.
Following the construction of Section~\ref{sec:def-general-cob}, we decompose
$\cW_1$ into sutured manifold cobordisms $\cN_1$ and $\tilde{\cW}_1$,
and we decompose $\cW_2$ into $\cN_2$ and $\tilde{\cW}_2$. We give $\cW$
the analogous decomposition into $\cN_1$ and $\tilde{\cW} := \tilde{\cW}_1 \cup \cN_2 \cup \tilde{\cW}_2$;
see Figure~\ref{fig:composition_law}.

\begin{figure}[ht!]
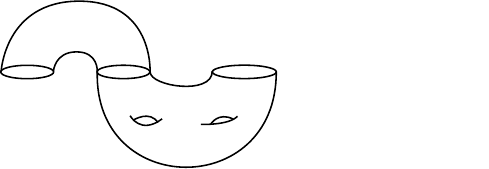
\caption{Decomposing $\cW = \cW_2 \circ \cW_1$ into $\cN_1$ and $\tilde{\cW} = \tilde{\cW}_1 \cup \cN_2 \cup \tilde{\cW}_2$.}
\label{fig:composition_law}
\end{figure}

Using the definition from equation~\eqref{eq:alternate-def-cobordism-map}, we have
\begin{equation}
\begin{split}
F_{\cW_2 ; \omega_2} \circ  F_{\cW_1 ; \omega_1}
 : =& \left( F_{ \cN_2 , \omega|_{\overline{N}_2}} \otimes \id_{\SFH(M_2)} \right)
 \circ \left( \id_{\SFH(M_1)} \otimes  F_{\tilde{\cW}_2 ; \omega|_{W_2\setminus N_2} }  \right)\\
\circ& \left( F_{\cN_1, \omega|_{\overline{N}_1}}\otimes \id_{\SFH(M_1)} \right)
\circ \left( \id_{\SFH(M_0)} \otimes  F_{ \tilde{ \cW }_1 ; \omega|_{ W_1 \setminus N_1 } } \right).
\end{split}
\label{eq:expanded-cobordism-map-turned-around}
\end{equation}
By commuting tensor factors,
we see that the right-hand side of
equation~\eqref{eq:expanded-cobordism-map-turned-around} coincides with
the composition of $F_{\cN_1;\omega|_{\overline{N}_1}}\otimes \id_{\SFH(M_2)}$ and
\begin{equation}
\left(\id_{\SFH(M_0)}\otimes \id_{ \SFH( - M_0 ) } \otimes  F_{ \cN_2 ; \omega|_{\overline{N}_2}}
\otimes \id_{ \SFH( M_2 ) } \right)
\circ \left(\id_{\SFH(M_0)}\otimes   F_{ \tilde{\cW}_1 ; \omega|_{W_1\setminus N_1}}
\otimes  F_{ \tilde{ \cW }_2 ; \omega|_{ W_2 \setminus N_2}} \right).
\label{eq:expanded--cobordism-composition}
\end{equation}
The hypotheses stated for the restricted version of the composition law from equation~\eqref{eq:restricted-composition-law} are satisfied for decomposing $\tilde{\cW}$ into the composition of $\mathsf{Id}_{-M_0} \sqcup \cN_2\sqcup \mathsf{Id}_{M_2}$ and $\tilde{\cW}_1\sqcup \tilde{\cW}_2$ (note that we are implicitly precomposing with the $i_{\ufrs_0}$, where $\ufrs_0$ is the unique $\Spin^c$ structure on the empty set). Hence equation~\eqref{eq:expanded--cobordism-composition} coincides
 with $\id_{\SFH(M_0)}\otimes F_{\tilde{\cW}; \omega|_{W\setminus N_1}}$.
 It follows that equation~\eqref{eq:expanded-cobordism-map-turned-around}
coincides with
\[
\left( F_{ \cN_1, \omega|_{N_1} } \otimes \id_{ \SFH( M_2 ) } \right)
\circ \left( \id_{ \SFH(M_0) }
\otimes  F_{ \tilde{ \cW } ; \omega|_{W\setminus N_1}}  \right)
\]
which is the definition of $F_{\cW;\omega}$ in equation~\eqref{eq:alternate-def-cobordism-map}.
\end{proof}

\subsection{Changing the 2-form on $W$}
\label{sec:change-2-formW}

We now prove Lemma~\ref{lem-modify-2-form-on-interior}.

\begin{proof}[Proof of Lemma~\ref{lem-modify-2-form-on-interior}]
We investigate equation~\eqref{eq:erroneous-conclusion} from
Remark~\ref{rem:failure-transitive-system-over-2-forms}.  Suppose that
$\cH_1,\dots, \cH_n$ is a sequence of sutured Heegaard diagrams such that
\begin{itemize}
\item $\cH_1$ is a diagram for $(M_0,\g_0)$ and $\cH_n$ is a diagram for $(M_1,\g_1)$,
\item $\cH_{i+1}$ is obtained from $\cH_{i}$ by either an elementary Heegaard
move, the contact gluing map, or
is the result of applying a 1-handle, 2-handle, or 3-handle map.
\end{itemize}

Consider the case when $\cH_{i}$ and $\cH_{i+1}$ are diagrams for the boundaries
of the 2-handle submanifold $\cW_2=(W_2,Z_2,[\xi_2])$ of $\cW$. Furthermore,
assume $\cH_{i}$ and $\cH_{i+1}$ are subdiagrams of a triple which is
subordinate to a bouquet for a framed link in the incoming boundary of $W_2$.
Write $\hat{\omega}_2$ for the restriction of $\omega$ to $\cW_2$. Write
$\omega_i$ and $\omega_{i+1}$ for the restrictions of $\omega$ to the manifolds
defined by $\cH_i$ and $\cH_{i+1}$, respectively.
Define $\hat{\eta}_{2}$, $\eta_i$, and $\eta_{i+1}$ similarly. An
argument using Stokes' theorem implies that the following diagram commutes up to
an overall factor of $z^x$:
\[
\begin{tikzcd}[column sep=2.3cm]
\CF(\cH_{i};\Lambda_{\omega_i})\arrow[d, "F_{\cW_2;\hat{\omega}_2}"]
\arrow[r,"\Psi_{\omega_i \to \omega_i+d \eta_i;\eta_i}"]&
	\CF(\cH_i;\Lambda_{\omega_i+d\eta_i})\arrow[d, "F_{\cW_2;
\hat{\omega}_2+d\hat{\eta}_2}"]\\
\CF(\cH_{i+1},\Lambda_{\omega_i})\arrow[r,"\Psi_{\omega_{i+1}+d
\eta_{i+1};\eta_{i+1}}"]
& \CF(\cH_{i+1};\Lambda_{\omega_{i+1}+d\eta_{i+1}}).
\end{tikzcd}
\]

In an analogous manner, we may relate  $\cH_{i}$ and $\cH_{i+1}$ by a similar
commutative square when $\cH_{i+1}$ is obtained from $\cH_i$ by an elementary
Heegaard move, or a 1-handle or 3-handle attachment. Stacking the $n-1$
projectively commutative squares, we obtain that the square
\[
\begin{tikzcd}[column sep=2.3cm]
\CF(\cH_1;\Lambda_{\omega_1})\arrow[d, "F_{\cW;\omega}"]
\arrow[r,"\Psi_{\omega_1\to \omega_1+d \eta_1;\eta_1}"]&
	\CF(\cH_1;\Lambda_{\omega_1+d\eta_1})\arrow[d, "F_{\cW; \omega+d\eta}"]\\
\CF(\cH_n,\Lambda_{\omega_n})\arrow[r,"\Psi_{\omega_n\to \omega_n+d \eta_n; \eta_n}"]
& \CF(\cH_{n};\Lambda_{ \omega_n+d\eta_{n}})
\end{tikzcd}
\]
commutes, up to an overall factor of $z^x$. Since $\eta|_{M_0} = \eta_1 = 0$ and
$\eta|_{M_1} = \eta_n = 0$, the maps $\Psi_{\omega_1\to \omega_1+d\eta_1;\eta_1}$ and
$\Psi_{\omega_n\to \omega_n+d\eta_n;\eta_n}$ are the identity, completing the
proof.
\end{proof}

\subsection{Perturbed and unperturbed cobordism maps}
\label{sec:normalization}

We are finally ready to prove Lemma~\ref{lem:perturbed-coincide}.

\begin{proof}[Proof of Lemma~\ref{lem:perturbed-coincide}]
    Let us write $W = W_1 \cup W_2 \cup W_3$, where $W_i$ is the $i$-handle part of $W$.
    Let $(\S, \as, \bs, \bs', w)$ be a triple subordinate to a bouquet for the 2-handles of $W$,
    and write $W_{\sca, \scb, \scb'}$ for the corresponding portion of $W_2$. In particular,
    $W_0 := W_2 \setminus \Int(W_{\sca,\scb,\scb'})$ is a boundary connected sum of copies of $S^1 \times D^3$.
    As $H^2(W_1, Y_0; \R) = 0$ and $H^2(W_3, Y_1; \R) = 0$, the restriction maps $H^2(W_1; \R) \to H^2(Y_0; \R)$
    and $H^2(W_3; \R) \to H^2(Y_1; \R)$ are both injective. Furthermore, $H^2(W_0; \R) = 0$.
    Hence, since $\omega|_{\d W} = 0$, we have $[\omega|_{W_0}] = 0$, $[\omega|_{W_1}] = 0$, and $[\omega|_{W_3}] = 0$.
    So there is a 1-form $\eta$ on $W$ such that $\eta|_{\d W} = 0$,
    and $\omega - d\eta$ vanishes on $W \setminus \Int(W_{\sca,\scb,\scb'})$;
    compare the proof of Lemma~\ref{lem:cut}.
    By Lemma~\ref{lem-modify-2-form-on-interior-closed}, we have
    \[
    F^\circ_{W, \frS; \omega} \dot{=} F^\circ_{W, \frS; \omega - d\eta}.
    \]
    Hence, we may assume that $\omega$ vanishes on $W_0$, $W_1$, and $W_3$.
    With this assumption, the maps $F^\circ_{W_1, \frS|_{W_1}; \omega|_{W_1}}$
    and $F^\circ_{W_3, \frS|_{W_3}; \omega|_{W_3}}$ are unperturbed.
    Furthermore,
    \[
    \langle i_*(\frs - \frs_0) \cup [\omega], [W, \d W] \rangle =
    \langle i_*(\frs|_{W_2} - \frs_0|_{W_2}) \cup [\omega|_{W_2}], [W_2, \d W_2] \rangle.
    \]
    So, without loss of generality, we can assume that $W = W_2$.

    Let $\xs$, $\xs' \in \T_{\sca} \cap \T_{\scb}$ and $\ys$, $\ys' \in \T_{\sca} \cap \T_{\scb'}$.
    Furthermore, let $\psi \in \pi_2(\xs, \ys, \Theta_{\scb,\scb'})$ and
    $\psi' \in \pi_2(\xs', \ys', \Theta_{\scb,\scb'})$
    be homology classes of triangles, where $\Theta_{\scb,\scb'} \in \T_{\scb} \cap \T_{\scb'}$.
    Note that
    \[
    \HF^\circ(\S,\bs,\bs'; \Lambda_{\omega|_{\d W_0}}) = \HF^\circ(\S,\bs,\bs')\otimes \Lambda,
    \]
    since $\omega|_{\d W_0} = 0$.
    Then, the coned-off domain $\tilde{\cD}(\psi) - \tilde{\cD}(\psi')$ represents the Poincar\'e dual of
    $\frs_w(\psi) - \frs_w(\psi') \in H^2(W_2)$. Hence
    \[
    A_\omega(\psi) - A_\omega(\psi') = \int_{\tilde{\cD}(\psi)} \omega - \int_{\tilde{\cD}(\psi')} \omega =
    \langle\, i_*(\frs_w(\psi) - \frs_w(\psi')) \cup [\omega], [W,\d W] \,\rangle,
    \]
    and equation~\eqref{eqn:normalization} follows.
\end{proof}

\bibliographystyle{custom}
\bibliography{biblio}
	
\end{document}

%% file: triple-cobordism.pdf_tex
\begingroup%
  \makeatletter%
  \providecommand\color[2][]{%
    \errmessage{(Inkscape) Color is used for the text in Inkscape, but the package 'color.sty' is not loaded}%
    \renewcommand\color[2][]{}%
  }%
  \providecommand\transparent[1]{%
    \errmessage{(Inkscape) Transparency is used (non-zero) for the text in Inkscape, but the package 'transparent.sty' is not loaded}%
    \renewcommand\transparent[1]{}%
  }%
  \providecommand\rotatebox[2]{#2}%
  \newcommand*\fsize{\dimexpr\f@size pt\relax}%
  \newcommand*\lineheight[1]{\fontsize{\fsize}{#1\fsize}\selectfont}%
  \ifx\svgwidth\undefined%
    \setlength{\unitlength}{163.61725613bp}%
    \ifx\svgscale\undefined%
      \relax%
    \else%
      \setlength{\unitlength}{\unitlength * \real{\svgscale}}%
    \fi%
  \else%
    \setlength{\unitlength}{\svgwidth}%
  \fi%
  \global\let\svgwidth\undefined%
  \global\let\svgscale\undefined%
  \makeatother%
  \begin{picture}(1,0.63375906)%
    \lineheight{1}%
    \setlength\tabcolsep{0pt}%
    \put(0,0){\includegraphics[width=\unitlength,page=1]{triple-cobordism.pdf}}%
    \put(0.74111036,0.14018833){\color[rgb]{0,0,0}\makebox(0,0)[lt]{\lineheight{1.25}\smash{\begin{tabular}[t]{l}$U_{\a}\times I$\end{tabular}}}}%
    \put(0.22418492,0.14042721){\color[rgb]{0,0,0}\makebox(0,0)[t]{\lineheight{1.25}\smash{\begin{tabular}[t]{c} $U_{\b}\times [0,\epsilon]$\end{tabular}}}}%
    \put(0.22349933,0.50526113){\color[rgb]{0,0,0}\makebox(0,0)[t]{\lineheight{1.25}\smash{\begin{tabular}[t]{c}$ U_{\b'}\times$\\$[1-\epsilon,1]$\end{tabular}}}}%
    \put(0.53037813,0.19748691){\color[rgb]{0,0,0}\makebox(0,0)[t]{\lineheight{1.25}\smash{\begin{tabular}[t]{c}$ N(\Sigma)\times I$\end{tabular}}}}%
    \put(0.50799332,0.00430931){\color[rgb]{0,0,0}\makebox(0,0)[t]{\lineheight{1.25}\smash{\begin{tabular}[t]{c}$M_0$\end{tabular}}}}%
    \put(0.50799332,0.6149461){\color[rgb]{0,0,0}\makebox(0,0)[t]{\lineheight{1.25}\smash{\begin{tabular}[t]{c}$M_1$\end{tabular}}}}%
    \put(0.33587619,0.31909443){\color[rgb]{0,0,0}\makebox(0,0)[rt]{\lineheight{1.25}\smash{\begin{tabular}[t]{r}$M_{\b,\b'}$\end{tabular}}}}%
    \put(0,0){\includegraphics[width=\unitlength,page=2]{triple-cobordism.pdf}}%
  \end{picture}%
\endgroup%

%% file: decompose-W.pdf_tex
\begingroup%
  \makeatletter%
  \providecommand\color[2][]{%
    \errmessage{(Inkscape) Color is used for the text in Inkscape, but the package 'color.sty' is not loaded}%
    \renewcommand\color[2][]{}%
  }%
  \providecommand\transparent[1]{%
    \errmessage{(Inkscape) Transparency is used (non-zero) for the text in Inkscape, but the package 'transparent.sty' is not loaded}%
    \renewcommand\transparent[1]{}%
  }%
  \providecommand\rotatebox[2]{#2}%
  \newcommand*\fsize{\dimexpr\f@size pt\relax}%
  \newcommand*\lineheight[1]{\fontsize{\fsize}{#1\fsize}\selectfont}%
  \ifx\svgwidth\undefined%
    \setlength{\unitlength}{133.03657806bp}%
    \ifx\svgscale\undefined%
      \relax%
    \else%
      \setlength{\unitlength}{\unitlength * \real{\svgscale}}%
    \fi%
  \else%
    \setlength{\unitlength}{\svgwidth}%
  \fi%
  \global\let\svgwidth\undefined%
  \global\let\svgscale\undefined%
  \makeatother%
  \begin{picture}(1,0.60488769)%
    \lineheight{1}%
    \setlength\tabcolsep{0pt}%
    \put(0,0){\includegraphics[width=\unitlength,page=1]{decompose-W.pdf}}%
    \put(0.65717459,0.07352821){\color[rgb]{0,0,0}\makebox(0,0)[t]{\lineheight{1.25}\smash{\begin{tabular}[t]{c}$\tilde{\cW}$\end{tabular}}}}%
    \put(0.21908183,0.47073952){\color[rgb]{0,0,0}\makebox(0,0)[lt]{\lineheight{1.25}\smash{\begin{tabular}[t]{l}$\cN$\end{tabular}}}}%
    \put(0.08705049,0.25314278){\color[rgb]{0,0,0}\makebox(0,0)[t]{\lineheight{1.25}\smash{\begin{tabular}[t]{c}$M_0$\end{tabular}}}}%
    \put(0.46382383,0.25314278){\color[rgb]{0,0,0}\makebox(0,0)[t]{\lineheight{1.25}\smash{\begin{tabular}[t]{c}$-M_0$\end{tabular}}}}%
    \put(0.88420568,0.3938601){\color[rgb]{0,0,0}\makebox(0,0)[t]{\lineheight{1.25}\smash{\begin{tabular}[t]{c}$M_1$\end{tabular}}}}%
  \end{picture}%
\endgroup%

%% file: composition_law.pdf_tex
\begingroup%
  \makeatletter%
  \providecommand\color[2][]{%
    \errmessage{(Inkscape) Color is used for the text in Inkscape, but the package 'color.sty' is not loaded}%
    \renewcommand\color[2][]{}%
  }%
  \providecommand\transparent[1]{%
    \errmessage{(Inkscape) Transparency is used (non-zero) for the text in Inkscape, but the package 'transparent.sty' is not loaded}%
    \renewcommand\transparent[1]{}%
  }%
  \providecommand\rotatebox[2]{#2}%
  \newcommand*\fsize{\dimexpr\f@size pt\relax}%
  \newcommand*\lineheight[1]{\fontsize{\fsize}{#1\fsize}\selectfont}%
  \ifx\svgwidth\undefined%
    \setlength{\unitlength}{234.37228845bp}%
    \ifx\svgscale\undefined%
      \relax%
    \else%
      \setlength{\unitlength}{\unitlength * \real{\svgscale}}%
    \fi%
  \else%
    \setlength{\unitlength}{\svgwidth}%
  \fi%
  \global\let\svgwidth\undefined%
  \global\let\svgscale\undefined%
  \makeatother%
  \begin{picture}(1,0.34411636)%
    \lineheight{1}%
    \setlength\tabcolsep{0pt}%
    \put(0,0){\includegraphics[width=\unitlength,page=1]{composition_law.pdf}}%
    \put(0.37303156,0.04173677){\color[rgb]{0,0,0}\makebox(0,0)[t]{\lineheight{1.25}\smash{\begin{tabular}[t]{c}$\tilde{\cW_1}$\end{tabular}}}}%
    \put(0.12435726,0.26720554){\color[rgb]{0,0,0}\makebox(0,0)[lt]{\lineheight{1.25}\smash{\begin{tabular}[t]{l}$\cN_1$\end{tabular}}}}%
    \put(0.04941241,0.14369126){\color[rgb]{0,0,0}\makebox(0,0)[t]{\lineheight{1.25}\smash{\begin{tabular}[t]{c}$M_0$\end{tabular}}}}%
    \put(0.26327999,0.14369126){\color[rgb]{0,0,0}\makebox(0,0)[t]{\lineheight{1.25}\smash{\begin{tabular}[t]{c}$-M_0$\end{tabular}}}}%
    \put(0.50190105,0.14374459){\color[rgb]{0,0,0}\makebox(0,0)[t]{\lineheight{1.25}\smash{\begin{tabular}[t]{c}$M_1$\end{tabular}}}}%
    \put(0,0){\includegraphics[width=\unitlength,page=2]{composition_law.pdf}}%
    \put(0.80540236,0.0425018){\color[rgb]{0,0,0}\makebox(0,0)[t]{\lineheight{1.25}\smash{\begin{tabular}[t]{c}$\tilde{\cW_2}$\end{tabular}}}}%
    \put(0.55672835,0.26797086){\color[rgb]{0,0,0}\makebox(0,0)[lt]{\lineheight{1.25}\smash{\begin{tabular}[t]{l}$\cN_2$\end{tabular}}}}%
    \put(0.69565099,0.14445819){\color[rgb]{0,0,0}\makebox(0,0)[t]{\lineheight{1.25}\smash{\begin{tabular}[t]{c}$-M_1$\end{tabular}}}}%
    \put(0.93427186,0.14813478){\color[rgb]{0,0,0}\makebox(0,0)[t]{\lineheight{1.25}\smash{\begin{tabular}[t]{c}$M_2$\end{tabular}}}}%
  \end{picture}%
\endgroup%

%% file: concordance_surgery.bbl
\begin{bibdiv}
\begin{biblist}

\bib{Akbulut}{article}{
      author={Akbulut, Selman},
       title={Variations on {F}intushel-{S}tern knot surgery on 4-manifolds},
        date={2002},
     journal={Turkish J. Math.},
      volume={26},
      number={1},
       pages={81\ndash 92},
}

\bib{AiPetersTwisted}{article}{
      author={Ai, Yinghua},
      author={Peters, Thomas},
       title={The twisted {F}loer homology of torus bundles},
        date={2010},
     journal={Algebr. Geom. Topol.},
      volume={10},
       pages={679\ndash 695},
}

\bib{BaldwinSivekContactInvariant}{article}{
      author={Baldwin, John~A.},
      author={Sivek, Steven},
       title={A contact invariant in sutured monopole homology},
        date={2016},
        ISSN={2050-5094},
     journal={Forum Math. Sigma},
      volume={4},
       pages={e12, 82},
         url={https://doi.org/10.1017/fms.2016.11},
      review={\MR{3510331}},
}

\bib{FSKnotSurgery}{article}{
      author={Fintushel, Ronald},
      author={Stern, Ronald},
       title={Knots, links, and 4-manifolds},
        date={1998},
     journal={Invent. Math.},
      volume={134},
       pages={363\ndash 400},
}

\bib{HKMTQFT}{unpublished}{
      author={Honda, Ko},
      author={Kazez, William},
      author={Mati\'c, Gordana},
       title={Contact structures, sutured {F}loer homology and {TQFT}},
        date={2008},
        note={e-print, \url{arXiv:0807.2431}},
}

\bib{Jabuka-Mark}{article}{
      author={Jabuka, Stanislav},
      author={Mark, Thomas~E.},
       title={Product formulae for {O}zsv\'{a}th-{S}zab\'{o} 4-manifold
  invariants},
        date={2008},
        ISSN={1465-3060},
     journal={Geom. Topol.},
      volume={12},
      number={3},
       pages={1557\ndash 1651},
         url={https://doi.org/10.2140/gt.2008.12.1557},
      review={\MR{2421135}},
}

\bib{JMConcordance}{article}{
      author={Juh\'asz, Andr\'as},
      author={Marengon, Marco},
       title={Concordance maps in knot {F}loer homology},
        date={2016},
     journal={Geom. Topol.},
      volume={20},
      number={6},
       pages={3623\ndash 3673},
}

\bib{JMComputeCobordismMaps}{article}{
      author={Juh\'asz, Andr\'as},
      author={Marengon, Marco},
       title={Computing cobordism maps in link {F}loer homology and the reduced
  {K}hovanov {TQFT}},
        date={2018},
     journal={Sel. Math. New. Ser.},
      volume={24},
      number={2},
       pages={1315\ndash 1390},
}

\bib{JMZExotic}{article}{
      author={Juh\'asz, Andr\'as},
      author={Miller, Maggie},
      author={Zemke, Ian},
       title={Exotic surfaces in the 4-ball},
        note={In preparation},
}

\bib{JTNaturality}{unpublished}{
      author={Juh\'{a}sz, Andr\'{a}s},
      author={Thurston, Dylan},
      author={Zemke, Ian},
       title={Naturality and mapping class groups in {H}eegaard {F}loer
  homology},
        date={2012},
        note={e-print, \url{arXiv:1210.4996}},
}

\bib{JDisks}{article}{
      author={Juh\'asz, Andr\'as},
       title={Holomorphic discs and sutured manifolds},
        date={2006},
        ISSN={1472-2747},
     journal={Algebr. Geom. Topol.},
      volume={6},
       pages={1429\ndash 1457},
}

\bib{JCob}{article}{
      author={Juh\'asz, Andr\'as},
       title={Cobordisms of sutured manifolds and the functoriality of link
  {F}loer homology},
        date={2016},
        ISSN={0001-8708},
     journal={Adv. Math.},
      volume={299},
       pages={940\ndash 1038},
}

\bib{JuhaszZemkeContactHandles}{article}{
      author={Juh\'{a}sz, Andr\'{a}s},
      author={Zemke, Ian},
       title={Contact handles, duality, and sutured {F}loer homology},
        date={2020},
        ISSN={1465-3060},
     journal={Geom. Topol.},
      volume={24},
      number={1},
       pages={179\ndash 307},
         url={https://doi.org/10.2140/gt.2020.24.179},
      review={\MR{4080483}},
}

\bib{Kronheimer-Mrowka-Sutured-Monopole}{article}{
      author={Kronheimer, Peter},
      author={Mrowka, Tomasz},
       title={Knots, sutures, and excision},
        date={2010},
        ISSN={0022-040X},
     journal={J. Differential Geom.},
      volume={84},
      number={2},
       pages={301\ndash 364},
         url={http://projecteuclid.org/euclid.jdg/1274707316},
      review={\MR{2652464}},
}

\bib{LekiliBrokenFibrations}{article}{
      author={Lekili, Yank\i},
       title={Heegaard-{F}loer homology of broken fibrations over the circle},
        date={2013},
        ISSN={0001-8708},
     journal={Adv. Math.},
      volume={244},
       pages={268 \ndash  302},
}

\bib{ZhenkunSM1}{article}{
      author={Li, Zhenkun},
       title={Gluing maps and cobordism maps for sutured monopole floer
  homology},
        date={2018},
        note={e-print, \url{arXiv:1810.13071}},
}

\bib{LipshitzCylindrical}{article}{
      author={Lipshitz, Robert},
       title={A cylindrical reformulation of {H}eegaard {F}loer homology},
        date={2006},
     journal={Geom. Topol.},
      volume={10},
       pages={955\ndash 1097},
}

\bib{LOTSpectralSequenceII}{article}{
      author={Lipshitz, Robert},
      author={Ozsv\'{a}th, Peter~S.},
      author={Thurston, Dylan~P.},
       title={Bordered {F}loer homology and the spectral sequence of a branched
  double cover {II}: the spectral sequences agree},
        date={2016},
        ISSN={1753-8416},
     journal={J. Topol.},
      volume={9},
      number={2},
       pages={607\ndash 686},
         url={https://doi.org/10.1112/jtopol/jtw003},
      review={\MR{3509974}},
}

\bib{LutzS1InvariantContactStructures}{article}{
      author={Lutz, Robert},
       title={Structures de contact sur les fibr\'es principaux en cercles de
  dimension trois},
        date={1977},
        ISSN={0373-0956},
     journal={Ann. Inst. Fourier (Grenoble)},
      volume={27},
      number={3},
       pages={ix, 1\ndash 15},
         url={http://www.numdam.org/item?id=AIF_1977__27_3_1_0},
}

\bib{Mark}{article}{
      author={Mark, Thomas},
       title={Knotted surfaces in 4-manifolds},
        date={2013},
     journal={Forum Math.},
      volume={25},
      number={3},
       pages={597\ndash 637},
}

\bib{genusbounds}{article}{
      author={Ozsv\'ath, Peter},
      author={Szab\'o, Zolt\'an},
       title={Holomorphic disks and genus bounds},
        date={2004},
     journal={Geom. Topol.},
      volume={8},
       pages={311\ndash 334},
}

\bib{OSDisks}{article}{
      author={Ozsv\'ath, Peter},
      author={Szab\'o, Zolt\'an},
       title={Holomorphic disks and topological invariants for closed
  three-manifolds},
        date={2004},
     journal={Ann. of Math. (2)},
      volume={159},
      number={3},
       pages={1027\ndash 1158},
}

\bib{OSKnots}{article}{
      author={Ozsv{\'a}th, Peter~S.},
      author={Szab{\'o}, Zolt{\'a}n},
       title={Holomorphic disks and knot invariants},
        date={2004},
     journal={Adv. Math.},
      volume={186},
      number={1},
       pages={58\ndash 116},
}

\bib{OSContactStructures}{article}{
      author={Ozsv\'ath, Peter},
      author={Szab\'o, Zolt\'an},
       title={{H}eegaard {F}loer homology and contact structures},
        date={2005},
     journal={Duke Math. J.},
      volume={129},
      number={1},
       pages={39\ndash 61},
}

\bib{OSTriangles}{article}{
      author={Ozsv\'ath, Peter},
      author={Szab\'o, Zolt\'an},
       title={Holomorphic triangles and invariants for smooth four-manifolds},
        date={2006},
     journal={Adv. Math.},
      volume={202},
      number={2},
       pages={326\ndash 400},
}

\bib{OSLinks}{article}{
      author={Ozsv{\'a}th, Peter~S.},
      author={Szab{\'o}, Zolt{\'a}n},
       title={Holomorphic disks, link invariants and the multi-variable
  {A}lexander polynomial},
        date={2008},
     journal={Algebr. Geom. Topol.},
      volume={8},
      number={2},
       pages={615\ndash 692},
}

\bib{SarkarMovingBasepoints}{article}{
      author={Sarkar, Sucharit},
       title={Moving basepoints and the induced automorphisms of link {F}loer
  homology},
        date={2015},
     journal={Algebr. Geom. Topol.},
      volume={15},
      number={5},
       pages={2479\ndash 2515},
}

\bib{TangeConcordance}{article}{
      author={Tange, Motoo},
       title={On the diffeomorphisms for {A}kbulut's knot concordance surgery},
        date={2005},
        ISSN={0218-2165},
     journal={J. Knot Theory Ramifications},
      volume={14},
      number={5},
       pages={539\ndash 563},
         url={https://doi.org/10.1142/S0218216505003944},
      review={\MR{2162113}},
}

\bib{WuPerturbedFibered}{article}{
      author={Wu, Zhongtao},
       title={Perturbed {F}loer homology of some fibered three-manifolds},
        date={2009},
     journal={Algebr. Geom. Topol.},
      volume={9},
       pages={337\ndash 350},
}

\bib{ZemGraphTQFT}{article}{
      author={Zemke, Ian},
       title={Graph cobordisms and {H}eegaard {F}loer homology},
        date={2015},
        note={e-print, \url{arXiv:1512.01184}},
}

\bib{ZemQuasi}{article}{
      author={Zemke, Ian},
       title={Quasistabilization and basepoint moving maps in link {F}loer
  homology},
        date={2017},
        ISSN={1472-2747},
     journal={Algebr. Geom. Topol.},
      volume={17},
      number={6},
       pages={3461\ndash 3518},
         url={https://doi.org/10.2140/agt.2017.17.3461},
      review={\MR{3709653}},
}

\bib{ZemkeConnectedSums}{article}{
      author={Zemke, Ian},
       title={Connected sums and involutive knot {F}loer homology},
        date={2019},
        ISSN={0024-6115},
     journal={Proc. Lond. Math. Soc. (3)},
      volume={119},
      number={1},
       pages={214\ndash 265},
         url={https://doi.org/10.1112/plms.12227},
      review={\MR{3957835}},
}

\end{biblist}
\end{bibdiv}
